\documentclass[12pt,leqno]{amsart}
\usepackage{pifont}
\usepackage{color}
\usepackage{mathrsfs}
\usepackage{amsmath}
\usepackage{amssymb}
\usepackage[hidelinks]{hyperref}
\usepackage{bookmark}
\usepackage[bottom]{footmisc}
\usepackage{verbatim}
\usepackage{mathtools}
\usepackage{orcidlink}

\numberwithin{equation}{section}
\newtheorem{thm}{Theorem}[section]
\newtheorem{lem}[thm]{Lemma}

\newtheorem{prop}[thm]{Proposition}
\newtheorem{cor}[thm]{Corollary}

\newtheorem{rmk}[thm]{Remark}

\newcommand{\End}{\operatorname{End}}
\newcommand{\Hom}{\operatorname{Hom}}
\newcommand{\Ext}{\operatorname{Ext}}
\newcommand{\Soc}{\operatorname{Soc}}
\newcommand{\rad}{\operatorname{rad}}
\newcommand{\ch}{\operatorname{ch}}
\newcommand{\one}{\mathbf{1}}
\newcommand{\id}{\operatorname{id}}
\newcommand{\wt}{\operatorname{wt}}

\newcommand{\Res}{\operatorname{Res}}

\newcommand{\Vir}{\mathfrak{Vir}}
\newcommand{\TL}{\mathcal{TL}}

\newcommand{\KL}{KL}
\newcommand{\Oc}{\mathcal O}
\newcommand{\Kac}{\mathcal K}

\def\Z{\mathbb{Z}}

\def\C{\mathbb{C}}

\allowdisplaybreaks

\begin{document}

\title[Exactness and DS realization]{Exactness and Drinfeld--Sokolov Realization of the McRae--Yang Tensor Functors}

\author{Shun Xu~\orcidlink{0009-0006-8080-8107}}

\address{School of Mathematical Sciences, Anhui University, Anhui, Hefei, 230601, China}

\email{shunxu@ahu.edu.cn}



\begin{abstract}
Let $p,q\ge2$ be coprime, set $k=-2+p/q$, $k'=-2+q/p$, and let $c_{p,q}=1-6(p-q)^2/(pq)$. We prove that the McRae--Yang functors $F_{p,q}:\KL^k(\mathfrak{sl}_2)\to\Oc_{c_{p,q}}$ and $F_{q,p}:\KL^{k'}(\mathfrak{sl}_2)\to\Oc_{c_{p,q}}$ are exact and faithful and, as $\C$-linear functors, are naturally isomorphic to the corresponding principal quantum Drinfeld--Sokolov reductions. Their restrictions to projective objects are fully faithful.
\end{abstract}

\maketitle

\section{Introduction}

Quantum Drinfeld--Sokolov reduction is a fundamental bridge between representations of affine Lie algebras and representations of $W$-algebras, realized through BRST cohomology. For $\widehat{\mathfrak{sl}}_2$, the principal $W$-algebra is the Virasoro vertex operator algebra. At admissible non-integral levels the representation categories relevant here are non-semisimple, and the interaction of reduction with extensions, projective objects, and tensor products becomes essential. McRae and Yang constructed tensor-categorical functors from the corresponding affine Kazhdan--Lusztig categories to Virasoro categories and conjectured that, after forgetting the tensor structure, these functors are naturally realized by principal quantum Drinfeld--Sokolov reduction. The purpose of this paper is to prove that conjecture and to determine the resulting behavior on projective objects.

Fix coprime integers $p,q\geq2$ and set
\begin{equation}\label{eq:parameters} k=-2+\frac pq,\qquad k'=-2+\frac qp,\qquad c_{p,q}=1-\frac{6(p-q)^2}{pq}.\end{equation}
The notation used below is fixed in Section~\ref{sec:prelim}. For $r\geq1$, let $V_r$ be the affine generalized Verma module induced from the $r$-dimensional irreducible horizontal $\mathfrak{sl}_2$-module, let $L_r$ be its unique irreducible quotient, and let $P_r\twoheadrightarrow L_r$ be the projective cover. We write $\KL^k(\mathfrak{sl}_2)$ for the finite-length category of grading-restricted generalized $V^k(\mathfrak{sl}_2)$-modules, equivalently the level-$k$ category whose composition factors are the $L_r$, and $\mathcal P^k$ for its full projective subcategory. On the Virasoro side, $\Oc_{c_{p,q}}$ denotes the $C_1$-cofinite grading-restricted generalized $V^{\mathrm{Vir}}_{c_{p,q}}$-module category of Subsection~\ref{subsec:MY-input}. For $j\geq1$, set $K_j:=\Kac^{(p,q)}_{j,1}$ and $S_j:=L\bigl(c_{p,q},h^{(p,q)}_{j,1}\bigr)$; thus $S_j$ is the simple head of the one-row Kac module $K_j$.

McRae and Yang construct a $\C$-linear right exact braided strong monoidal functor $F_{p,q}:\KL^k(\mathfrak{sl}_2)\to\Oc_{c_{p,q}}$, normalized by $F_{p,q}(V_2)\cong K_2$ together with the rigid self-duality data of $K_2$ \cite[Theorem~7.15]{McRaeYang}. For the level $k'$, they construct a second functor $F_{q,p}:\KL^{k'}(\mathfrak{sl}_2)\to\Oc_{c_{p,q}}$, normalized by the rigid datum of $\Kac^{(p,q)}_{1,2}$. Their Conjecture~7.16 asserts that these functors are exact and that their underlying $\C$-linear functors agree naturally with the corresponding principal quantum Drinfeld--Sokolov reductions \cite[Conjecture~7.16]{McRaeYang}.

Our main result is the following. Here $H^0_{DS,+}$ denotes the degree-zero principal rank-one BRST reduction in the conventions of Subsection~\ref{subsec:DS-prelim}.

\begin{thm}\label{thm:main}
Let $p,q\geq2$ be coprime and let $k,k'$ and $c_{p,q}$ be as in \eqref{eq:parameters}. There is a natural isomorphism
\(
F_{p,q}\cong H^0_{DS,+}\big|_{\KL^k(\mathfrak{sl}_2)}
\)
of $\C$-linear functors. Consequently $F_{p,q}$ is exact and faithful, and $F_{p,q}|_{\mathcal P^k}$ is fully faithful. The same conclusions hold for the second McRae--Yang functor $F_{q,p}$ with source $\KL^{k'}(\mathfrak{sl}_2)$; see Corollary~\ref{cor:transposed-reduction}. In particular, Conjecture~7.16 of McRae--Yang holds.
\end{thm}

The proof separates the tensor-categorical reconstruction of the McRae--Yang functor from the BRST analysis. We first determine the images of affine projectives under $F_{p,q}$. For $n\geq1$ and $1\leq r\leq p-1$, set $a=a_{n,r}=np-r$, $b=b_{n,r}=np+r$, and $c=c_{n,r}=(n+2)p-r$, as in Subsection~\ref{subsec:Felder-one-row}. The non-wall projective $P_b$ sits in the nonsplit generalized-Verma sequence $0\to V_a\to P_b\to V_b\to0$. Since $F_{p,q}$ is initially known only to be right exact, this sequence cannot simply be transported across the functor. Instead, using the projective tensor recursions for $V_2\boxtimes P_j$, the exact functor $K_2\boxtimes-$, and the conformal-residue decomposition of Subsection~\ref{subsec:one-row-residue}, we reconstruct an intrinsic Virasoro filtration $0\to K_a\to F_{p,q}(P_b)\to K_b\to0$. Thus the required monomorphism is produced inside $\Oc_{c_{p,q}}$, rather than inferred from a source exact sequence through right exactness.

A separate argument proves faithfulness on projectives before any target Hom-space computation. The indecomposable affine projectives are organized into reflection chains with explicitly known adjacent Hom spaces and one-dimensional radical lines. For $p\geq3$, the Temperley--Lieb/tilting realization of Subsection~\ref{subsec:TL-prelim} reduces faithfulness to tensor-ideal detection by the first critical Jones--Wenzl projector. For $p=2$, logarithmicity of the reconstructed images, together with preservation of the affine minus twist, detects the radical endomorphisms directly. This gives faithfulness of $F_{p,q}|_{\mathcal P^k}$ uniformly for all $p\geq2$.

We then identify the reconstructed Virasoro extensions. The one-row labels $a,b,c$ are matched in Subsection~\ref{subsec:Felder-one-row} with consecutive terms in Nakano's Felder complexes, and the relevant extensions lie in the Serre subcategory $\mathscr C\subset\Oc_{c_{p,q}}$ fixed in Subsection~\ref{subsec:Nakano-extension-interface}. For $a\geq2$, Nakano's logarithmic-extension theorem yields the required recognition of a logarithmic extension $0\to K_a\to E\to K_b\to0$; the exceptional vacuum edge $a=1$ is handled by the Kyt\"ol\"a--Ridout staggered-module input assembled in Subsection~\ref{subsec:Felder-one-row}. Consequently every non-wall image $F_{p,q}(P_{np+r})$ is identified with the corresponding logarithmic module singled out by the Nakano interface. This recognition supplies the target Hom-space upper bounds, while projective faithfulness supplies the matching lower bounds, proving full faithfulness on $\mathcal P^k$.

The Drinfeld--Sokolov side is developed independently. Set $H:=H^0_{DS,+}|_{\KL^k(\mathfrak{sl}_2)}$. Using the higher-cohomology vanishing for simple objects and finite-length d\'evissage, we prove that $H:\KL^k(\mathfrak{sl}_2)\to\Oc_{c_{p,q}}$ is exact. Since every nonzero simple object has nonzero reduction, exactness also implies faithfulness. Moreover, $H(V_j)\cong K_j$ and $H(L_j)\cong S_j$ for every $j\geq1$. Applying exactness to the affine projective filtrations and then the same logarithmic-recognition theorem gives $H(P_j)\cong F_{p,q}(P_j)$ for every indecomposable projective $P_j$. Thus the exactness and faithfulness of $H$ do not depend on the projective reconstruction for $F_{p,q}$.

The remaining issue is naturality. At the beginning of Section~\ref{sec:prelim}, every grading-restricted generalized module $M$ is equipped with its canonical conformal nilpotent $D_M$, the nilpotent part of $L_0$ on generalized conformal-weight spaces. For a non-wall affine projective $P$, the corresponding operator $D_P$, defined using $L^{\mathrm{aff}}_0$, spans $\rad\End(P)$. The affine minus twist records this nilpotent, and the BRST exponential identity carries it to the canonical nilpotent on $H(P)$; twist preservation for $F_{p,q}$ gives the parallel identity on $F_{p,q}(P)$. Hence an objectwise isomorphism between the two projective images can be normalized to intertwine the distinguished radical element. Along each reflection chain, one scalar normalization on an adjacent arrow then forces naturality for the reverse arrow as well. This produces a natural isomorphism on $\mathcal P^k$, which extends uniquely to all of $\KL^k(\mathfrak{sl}_2)$ by the extension theorem for additive right exact functors on an abelian category with enough projectives.

Finally, transporting the braided strong monoidal structure of $F_{p,q}$ across the natural isomorphism equips $H$ with a braided strong monoidal structure. Write $F^{(q,p)}$ for the McRae--Yang functor obtained by applying the ordered-pair construction of Subsection~\ref{subsec:MY-input} to $(q,p)$. The same argument applies to $F^{(q,p)}$; the rigid normalization and the parameter-swap isomorphism $\Kac^{(q,p)}_{2,1}\cong\Kac^{(p,q)}_{1,2}$ then identify $F^{(q,p)}$ with McRae--Yang's second functor $F_{q,p}$. This yields the second assertion of Theorem~\ref{thm:main}.

The paper is organized as follows. Section~\ref{sec:prelim} fixes the notation and packages the external structural input used later: affine projectives and Hom spaces, the McRae--Yang functors, one-row Virasoro and Felder data, logarithmic-extension and staggered-module results, rank-one Drinfeld--Sokolov conventions, and Temperley--Lieb/tilting theory. Section~\ref{sec:DS} proves the independent Drinfeld--Sokolov statements, including higher-cohomology vanishing, exactness, faithfulness, generalized-Verma and simple images, and exponential compatibility. Section~\ref{sec:images} reconstructs the images of projective modules under $F_{p,q}$ and proves projective faithfulness. Section~\ref{sec:fullness} identifies the non-wall logarithmic images and computes their Hom spaces, yielding full faithfulness on projectives. Section~\ref{sec:comparison} compares the two functors on projectives, proves naturality by matching the canonical conformal nilpotents, and extends the comparison to the whole category.

\section{Preliminaries and structural input}\label{sec:prelim}

We collect here the conventions and external structural results used throughout the paper.  The affine Kazhdan--Lusztig and projective input, including the McRae--Yang tensor functors, is taken from McRae--Yang \cite{McRaeYang}; the one-row Virasoro Kac-module input is taken from McRae--Sopin \cite{McRaeSopin}; the finite Virasoro category-\(\mathcal O\) calculations are taken from Boe--Nakano--Wiesner \cite{BoeNakanoWiesner}; the Felder and logarithmic-extension input is taken from Nakano \cite{Nakano}; the staggered-module input is taken from Kyt\"ol\"a--Ridout \cite{KytolaRidout}; the rank-one Drinfeld--Sokolov input is taken from Arakawa--Frenkel and Arakawa \cite{ArakawaFrenkel,ArakawaW}; and the Temperley--Lieb/tilting input is taken from Sutton--Tubbenhauer--Wedrich--Zhu and Goodman--Wenzl \cite{SuttonTubbenhauerWedrichZhu,GoodmanWenzl}.  As in the preliminary section of the comparison manuscript supplied with this paper, we separate \emph{definitions and conventions} from \emph{literature interfaces}: for each source used repeatedly, the precise results needed later are collected in a numbered package below.

Let $V$ be a vertex operator algebra with conformal vector $\omega$, and let $W$ be a weak $V$-module, with \(Y_W(\omega,z)=\sum_{n\in\Z}L_W(n)z^{-n-2}.\) Following the standard generalized-module convention, cf.~\cite{HLZ3}, $W$ is a \emph{generalized $V$-module} if \(W=\bigoplus_{\lambda\in\C}W_{[\lambda]}, W_{[\lambda]}:=\{w\in W:(L_W(0)-\lambda)^Nw=0\text{ for some }N\ge1\}.\) It is \emph{grading restricted} if $\dim W_{[\lambda]}<\infty$ for every $\lambda$ and, for every fixed $\lambda\in\C$, one has $W_{[\lambda-n]}=0$ for all sufficiently large integers $n$. We write \(\operatorname{Spec}(W):=\{\lambda\in\C:W_{[\lambda]}\ne0\}\) and define the canonical conformal nilpotent $D_W$ by \(D_W|_{W_{[\lambda]}}:=L_W(0)-\lambda\id_{W_{[\lambda]}}.\) If $v\in V$ is homogeneous, then $[L_W(0),v_m]=(\wt v-m-1)v_m$ and $\wt v-m-1\in\Z$; hence $D_Wv_m=v_mD_W$, so $D_W\in\End_V(W)$ and is locally nilpotent. For $\xi\in\C/\Z$ set \(W^{[\xi]}:=\bigoplus_{\lambda\in\xi}W_{[\lambda]};\) each $W^{[\xi]}$ is a $V$-submodule, called the \emph{conformal-residue sector} of $W$ corresponding to $\xi$. A grading-restricted generalized module is \emph{ordinary} if $D_W=0$, equivalently if $L_W(0)$ acts semisimply, and \emph{logarithmic} otherwise. When no confusion is possible we write $L_0$ for $L_W(0)$; on the affine side we use $L^{\mathrm{aff}}_0$ for the Sugawara conformal zero mode.

\subsection{Affine generalized Verma modules, projectives, Hom spaces, and the minus twist}\label{subsec:affine-projectives}

We first fix the affine notation used throughout the paper, following the formal conventions of \cite[Section~2.1]{McRaeYang}.  Let $\mathfrak{sl}_2$ have basis $e,f,h$ with $[e,f]=h$, $[h,e]=2e$, and $[h,f]=-2f$, and normalize the invariant bilinear form by $(e,f)=1$ and $(h,h)=2$.  The corresponding affine Lie algebra is \(\widehat{\mathfrak{sl}}_2=\mathfrak{sl}_2\otimes\C[t,t^{-1}]\oplus\C\mathbf{k},\) where $\mathbf{k}$ is central and \([a\otimes t^m,b\otimes t^n]=[a,b]\otimes t^{m+n}+m(a,b)\delta_{m+n,0}\mathbf{k}.\) For $a\in\mathfrak{sl}_2$ and $n\in\Z$, write $a(n):=a\otimes t^n$; the copy $\mathfrak{sl}_2\otimes1$ is called the horizontal $\mathfrak{sl}_2$, and in particular its Cartan zero mode is $h(0)=h\otimes1$. Put
\[
(\widehat{\mathfrak{sl}}_2)_+=\mathfrak{sl}_2\otimes t\C[t],\qquad (\widehat{\mathfrak{sl}}_2)_0=\mathfrak{sl}_2\otimes1\oplus\C\mathbf{k},\qquad (\widehat{\mathfrak{sl}}_2)_{\ge0}=(\widehat{\mathfrak{sl}}_2)_+\oplus(\widehat{\mathfrak{sl}}_2)_0.
\]
For $r\ge1$, let $M_r$ be the $r$-dimensional irreducible $\mathfrak{sl}_2$-module, of highest weight $r-1$.  Regard $M_r$ as an $(\widehat{\mathfrak{sl}}_2)_{\ge0}$-module by letting $(\widehat{\mathfrak{sl}}_2)_+$ act trivially and $\mathbf{k}$ act by the fixed level $k=-2+p/q$.  We define \(V_r:=U(\widehat{\mathfrak{sl}}_2)\otimes_{U((\widehat{\mathfrak{sl}}_2)_{\ge0})}M_r.\) Thus $V_r$ is precisely the generalized Verma module denoted $\mathcal V_r^k$ (or $\mathcal V_r$ when the level is understood) in \cite[Section~2.1]{McRaeYang}.  McRae and Yang remark in \cite[Section~1.1]{McRaeYang} that these modules are also called Weyl modules, but their formal construction and the structural theorems used below are stated in terms of generalized Verma modules; we therefore use \emph{generalized Verma module} as the terminology for $V_r$ throughout this paper.  In particular, $V_1=V^k(\mathfrak{sl}_2)$ is the universal affine vertex operator algebra at level $k$.

Let $L^{\mathrm{aff}}_0$ denote the zero mode $L(0)$ of the Sugawara conformal vector of $V^k(\mathfrak{sl}_2)$; this is the operator denoted $L(0)$ in \cite[Section~2.1]{McRaeYang}.  On the degree-zero subspace $M_r\subset V_r$ it acts by the scalar
\begin{equation}\label{eq:affine-lowest-weight}
\Delta_r^{\mathrm{aff}}=\frac{r^2-1}{4(k+2)}=\frac{q(r^2-1)}{4p}.
\end{equation}
For $r\ge1$, let $J_r\subset V_r$ be the unique maximal proper submodule and set $L_r:=V_r/J_r$.  Thus $L_r$ is the unique irreducible quotient of $V_r$.  Since $M_r$ is irreducible under the horizontal $\mathfrak{sl}_2$ and generates $V_r$, one has $J_r\cap M_r=0$; hence the lowest conformal-degree horizontal subspace of $L_r$ is again $M_r$.  In particular, the simple modules $L_r$ are pairwise nonisomorphic.  Following \cite[Section~2.3, Definition~2.4 and Proposition~2.5]{McRaeYang}, $\KL^k(\mathfrak{sl}_2)$ denotes the category of finite-length level-$k$ $\widehat{\mathfrak{sl}}_2$-modules whose composition factors are among the $L_r$; equivalently, it is the category of finite-length grading-restricted generalized $V^k(\mathfrak{sl}_2)$-modules. When no ambiguity is possible, we abbreviate $\KL^k(\mathfrak{sl}_2)$ to $\KL^k$.  By \cite[Theorem~2.8]{McRaeYang}, $\KL^k(\mathfrak{sl}_2)$ carries the braided tensor category structure used below; in particular, it is an abelian category.  We write $\boxtimes$ for its tensor product and let $\mathcal P^k\subset\KL^k(\mathfrak{sl}_2)$ denote the full subcategory of projective objects.  Unless otherwise indicated, all subobjects, kernels, images, $\Hom$ spaces, and $\End$ spaces in this subsection are taken in $\KL^k(\mathfrak{sl}_2)$.  Thus ``finite length'' means finite length as a $V^k(\mathfrak{sl}_2)$-module, equivalently the existence of a finite composition series by $V^k(\mathfrak{sl}_2)$-submodules, and an object is \emph{indecomposable} if it is not isomorphic to a direct sum of two nonzero objects.  For an object $M$, we write $[M]$ for its class in the Grothendieck group $K_0(\KL^k(\mathfrak{sl}_2))$, defined by the relations $[B]=[A]+[C]$ for every short exact sequence $0\to A\to B\to C\to0$ in $\KL^k(\mathfrak{sl}_2)$.

We shall use the following standard finite-length terminology.  A projective cover of a simple object $L$ is an essential epimorphism $P\twoheadrightarrow L$ with $P$ projective.  For a finite-length object $M$, write $\rad M$ for the intersection of its maximal proper subobjects and $\Soc M$ for the sum of its simple subobjects; the quotient $M/\rad M$ is the \emph{head} (or \emph{top}) of $M$.  A \emph{Loewy series} of a finite-length module \(M\) is a finite filtration
\( M=M_0\supseteq M_1\supseteq\cdots\supseteq M_\ell=0
\) whose successive quotients \(M_i/M_{i+1}\) are semisimple.  The radical filtration
\( M\supseteq \rad M\supseteq \rad^2M\supseteq\cdots
\) is a canonical Loewy series; its successive quotients
\(
\rad^iM/\rad^{i+1}M
\) are called the \emph{radical Loewy layers}.  For a projective cover $P\twoheadrightarrow L$ of a simple object, its kernel is $\rad P$, so the head of $P$ is $P/\rad P\cong L$.

For a unital associative $\C$-algebra $A$, write $\rad A$ for its Jacobson radical, namely the intersection of its maximal left ideals, and call $A$ \emph{local} if it has a unique maximal left ideal.  For a local algebra this unique maximal left ideal is $\rad A$, and it is precisely the set of nonunits.

\begin{lem}\label{lem:fitting-local}
Let $M$ be a finite-length object of an abelian category and let $f\in\End(M)$. For all sufficiently large $N$, one has \(M=\ker(f^N)\oplus\operatorname{im}(f^N).\) If $M$ is indecomposable, then $f$ is either invertible or nilpotent; consequently, $\End(M)$ is a local algebra.
\end{lem}

\begin{proof}
The ascending chain $\ker f\subseteq\ker f^2\subseteq\cdots$ and the descending chain $M\supseteq\operatorname{im}f\supseteq\operatorname{im}f^2\supseteq\cdots$ stabilize because $M$ has finite length, and the usual Fitting argument gives the decomposition. If $M$ is indecomposable, one summand vanishes. Thus either $f^N=0$, so $f$ is nilpotent, or $\ker(f^N)=0$ and $\operatorname{im}(f^N)=M$, so $f$ is invertible. Hence, for every $f\in\End(M)$, either $f$ or $1-f$ is invertible, and the standard local-ring criterion implies that $\End(M)$ is local. See \cite[Proposition~11.7 and Lemma~12.8]{AndersonFuller} for the module-theoretic form; the argument above uses only kernels, images, and finite length and therefore applies in any abelian category.
\end{proof}

We also fix the two tensor-category terms used below.  If $(\mathscr C,\boxtimes,c)$ is braided, a \emph{twist} is a natural automorphism $\theta:\mathrm{Id}_{\mathscr C}\Rightarrow\mathrm{Id}_{\mathscr C}$ satisfying $\theta_{\one}=\id_{\one}$ and \(\theta_{X\boxtimes Y}=c_{Y,X}\circ c_{X,Y}\circ(\theta_X\boxtimes\theta_Y).\) We isolate first the affine projective results of McRae--Yang that are used repeatedly later.  In the tensor-product formulas below we adopt the convention $P_j=0$ for $j\le0$.

\begin{prop}\label{prop:MY-affine-literature-package}
Let $k=-2+p/q$ with $p,q\ge2$ coprime.
\begin{enumerate}
\item The simple objects of $\KL^k(\mathfrak{sl}_2)$ are the $L_r$ ($r\ge1$), and every object is a direct sum of finite-dimensional horizontal $\mathfrak{sl}_2$-submodules; equivalently, the horizontal $\mathfrak{sl}_2$-action is locally finite.  The category has enough projectives, and every $L_r$ has an indecomposable projective cover $P_r\twoheadrightarrow L_r$.  If $1\le r<p$ or $p\mid r$, then $P_r=V_r$.  If $p\mid r$, this object is simple.  Consequently, for every $n\ge1$, the canonical quotient $V_{np}\twoheadrightarrow L_{np}$ is an isomorphism and \(P_{np}=V_{np}\cong L_{np}.\) \item For $n\ge1$ and $1\le r\le p-1$, there is a nonsplit exact sequence
\begin{equation}\label{eq:MY-projective-Verma-filtration}
0\longrightarrow V_{np-r}\longrightarrow P_{np+r}\longrightarrow V_{np+r}\longrightarrow0.
\end{equation}
Moreover the Loewy layers are
\begin{equation}\label{eq:MY-projective-Loewy-layers}
\Soc P_{np+r}\cong L_{np+r},\qquad
\frac{\rad P_{np+r}}{\Soc P_{np+r}}\cong L_{np-r}\oplus L_{(n+2)p-r},\qquad
\frac{P_{np+r}}{\rad P_{np+r}}\cong L_{np+r}.
\end{equation}
Hence
\begin{equation}\label{eq:MY-projective-K0}
[P_{np+r}]=2[L_{np+r}]+[L_{np-r}]+[L_{(n+2)p-r}]
\quad\text{in }K_0(\KL^k(\mathfrak{sl}_2)).
\end{equation}
\item The tensor product of $V_2$ with every indecomposable projective is given explicitly as follows.  If $p=2$, then for $n\ge0$,
\begin{align}
V_2\boxtimes P_{2n+1}&\cong P_{2(n-1)}\oplus P_{2n}^{\oplus2}\oplus P_{2(n+1)},\label{eq:MY-projective-recursion-p2-odd}\\
V_2\boxtimes P_{2n+2}&\cong P_{2n+3}.\label{eq:MY-projective-recursion-p2-even}
\end{align}
If $p\ge3$, then for $n\ge0$ and $1\le r\le p$,
\begin{equation}\label{eq:MY-projective-recursion-pge3}
V_2\boxtimes P_{np+r}\cong
\begin{cases}
P_{np}^{\oplus2}\oplus P_{np+2},&r=1,\\
P_{np+r-1}\oplus P_{np+r+1},&2\le r\le p-2,\\
P_{(n-1)p}\oplus P_{(n+1)p-2}\oplus P_{(n+1)p},&r=p-1,\\
P_{(n+1)p+1},&r=p.
\end{cases}
\end{equation}
In particular $V_2\boxtimes-$ preserves the full projective subcategory.  Moreover, the full projective subcategory $\mathcal P^k$ is monoidal.
\item Choose nonzero morphisms $F_r^+:P_r\to P_{2p-r}$ for $1\le r\le p-1$, where $P_r=V_r$ and $F_r^+$ is the inclusion $V_r\hookrightarrow P_{2p-r}$ from \eqref{eq:MY-projective-Verma-filtration}.  For $n\ge1$ and $1\le r\le p-1$, let \(F_{np+r}^-:P_{np+r}\longrightarrow P_{np-r}\) and \(F_{np+r}^+:P_{np+r}\longrightarrow P_{(n+2)p-r}\) denote nonzero generators of the one-dimensional Hom spaces displayed below.  In the endomorphism formula, $F_{np-r}^+$ is understood from these choices; when $n=1$ it is the initial inclusion $F_{p-r}^+$.  Then the complete list of nonzero Hom spaces between indecomposable projectives is
\[
\begin{aligned}
\Hom(P_r,P_r)&=\C\id_{P_r} &&(1\le r\le p-1),\\
\Hom(P_r,P_{2p-r})&=\C F_r^+ &&(1\le r\le p-1),\\
\Hom(P_{np},P_{np})&=\C\id_{P_{np}} &&(n\ge1),\\
\Hom(P_{np+r},P_{np-r})&=\C F_{np+r}^- &&(n\ge1,\ 1\le r\le p-1),\\
\Hom(P_{np+r},P_{(n+2)p-r})&=\C F_{np+r}^+ &&(n\ge1,\ 1\le r\le p-1),\\
\End(P_{np+r})=\C\id_{P_{np+r}}&\oplus\C\bigl(F_{np-r}^+\circ F_{np+r}^-\bigr) &&(n\ge1,\ 1\le r\le p-1).
\end{aligned}
\]
Every Hom space between indecomposable projectives not listed above is zero.
\item For $\bar\epsilon\in\Z/2\Z$, let $\KL^k_{\bar\epsilon}(\mathfrak{sl}_2)$ be the full subcategory of objects whose $h(0)$-eigenvalues all lie in $\epsilon+2\Z$, where $\epsilon\in\{0,1\}$ represents $\bar\epsilon$.  Then \(\KL^k(\mathfrak{sl}_2)=\KL^k_{\bar0}(\mathfrak{sl}_2)\oplus\KL^k_{\bar1}(\mathfrak{sl}_2), \KL^k_{\bar\epsilon}\boxtimes\KL^k_{\bar\delta}\subseteq\KL^k_{\overline{\epsilon+\delta}}.\) For the braided tensor category structure fixed above, the standard twist is $\theta^+_M=e^{2\pi iL^{\mathrm{aff}}_0}$.  On a homogeneous object $M\in\KL^k_{\bar\epsilon}(\mathfrak{sl}_2)$, the parity-modified second compatible twist is
\begin{equation}\label{eq:minus-twist}
\theta^-_M=(-1)^\epsilon\theta^+_M=(-1)^\epsilon e^{2\pi iL^{\mathrm{aff}}_0}.
\end{equation}
For a general object, $\theta^-$ is defined componentwise under the parity decomposition above.  Finally, every $P_{np+r}$ with $n\ge1$ and $1\le r\le p-1$ is logarithmic.
\end{enumerate}
\end{prop}

\begin{proof}
For (i), the definition of $\KL^k(\mathfrak{sl}_2)$ in \cite[Section~2]{McRaeYang} gives the simple objects $L_r$, and \cite[Section~5.1]{McRaeYang} states that every object is a direct sum of finite-dimensional horizontal $\mathfrak{sl}_2$-submodules.  The existence of enough projectives and of the projective covers $P_r\twoheadrightarrow L_r$ is proved in \cite[Theorem~4.8]{McRaeYang}; \cite[Theorem~1.3(1)]{McRaeYang} further states that $P_r=V_r$ for $1\le r\le p-1$ and for every label divisible by $p$, and that $P_r$ is simple when $p\mid r$.  Thus if $r=np$, the canonical quotient $V_{np}\twoheadrightarrow L_{np}$ is a nonzero epimorphism from the simple object $V_{np}=P_{np}$ onto the simple object $L_{np}$, hence is an isomorphism.  This is the reason for the final isomorphism in (i); it is a consequence of Theorem~1.3(1), rather than an additional equality stated there.

For (ii), the nonsplit sequence and the Loewy diagram are stated in \cite[Theorem~1.3(2)]{McRaeYang}, with the construction and proof in \cite[Theorem~4.4]{McRaeYang}.  The proof of the latter theorem shows that $L_{np+r}$ is the unique irreducible submodule and the unique irreducible quotient of $P_{np+r}$, while the remaining two composition factors are $L_{np-r}$ and $L_{(n+2)p-r}$ and form the middle row of the displayed Loewy diagram.  Hence $\Soc P_{np+r}\cong L_{np+r}$.  Since $P_{np+r}\twoheadrightarrow L_{np+r}$ is the projective cover from (i), its kernel is $\rad P_{np+r}$, so $P_{np+r}/\rad P_{np+r}\cong L_{np+r}$; the middle radical layer is therefore $L_{np-r}\oplus L_{(n+2)p-r}$.  This gives exactly \eqref{eq:MY-projective-Loewy-layers}.  By additivity of $K_0$ along this finite Loewy filtration, the class of $P_{np+r}$ is the sum of the classes of its four composition factors, with multiplicity.  Therefore \([P_{np+r}]=[L_{np+r}]+[L_{np-r}]+[L_{(n+2)p-r}]+[L_{np+r}],\) which is \eqref{eq:MY-projective-K0}.

Part (iii) is precisely the projective tensor-product theorem \cite[Theorem~1.4]{McRaeYang}, rewritten with our notation and with the source convention $P_j=0$ for $j\le0$.  Its part~(1) gives \eqref{eq:MY-projective-recursion-p2-odd}--\eqref{eq:MY-projective-recursion-p2-even}, and its part~(2) gives \eqref{eq:MY-projective-recursion-pge3}.  Each right-hand side is a finite direct sum of projectives, so $V_2\boxtimes-$ preserves projectives.  Finally, projective objects coincide with rigid objects by \cite[Corollary~4.10]{McRaeYang}, and tensor products of rigid objects are rigid; hence $\mathcal P^k$ is monoidal, as also recorded in \cite[Section~6.1]{McRaeYang}.

For (iv), the proof of \cite[Theorem~6.6]{McRaeYang} explicitly lists all nonzero spaces $\Hom(P_a,P_b)$.  In the source notation, the initial inclusion is $F_r^+:V_r\hookrightarrow P_{2p-r}$; for $n\ge1$ the two adjacent one-dimensional spaces are generated by maps from $P_{np+r}$ to $P_{np-r}$ and $P_{(n+2)p-r}$, denoted there by $F_{np+r}^-$ and $F_{np+r}^+$; and the two-dimensional endomorphism space has basis $\id_{P_{np+r}}$ and $F_{np-r}^+\circ F_{np+r}^-$.  Translating $\mathcal P_r,\mathcal V_r$ there to $P_r,V_r$ here gives exactly the display in (iv).  The source explicitly says that these are all nonzero morphism spaces, which gives the final vanishing assertion.  The backtracking composite is nonzero because it is the second displayed basis vector of $\End(P_{np+r})$.

For (v), the decomposition by the parity of the $h(0)$-eigenvalues, its uniqueness, and its compatibility with tensor products are established in \cite[Section~5.1]{McRaeYang}.  The standard vertex-tensor-category twist $e^{2\pi iL(0)}$ is stated in \cite[Section~2.3]{McRaeYang}; with our notation $L(0)=L^{\mathrm{aff}}_0$, this is $\theta_M^+=e^{2\pi iL^{\mathrm{aff}}_0}$.  In the proof of \cite[Theorem~5.4]{McRaeYang}, McRae--Yang show that if $\theta$ is a compatible twist, then on $W\in\KL^k_{\bar i}(\mathfrak{sl}_2)$ the rule $\widetilde\theta_W=(-1)^i\theta_W$ defines the second compatible twist.  Applying this to the standard twist gives \eqref{eq:minus-twist}; the unique direct-sum decomposition from Section~5.1 defines it componentwise on arbitrary objects.  The logarithmicity assertion is \cite[Theorem~4.13]{McRaeYang} and is also summarized in \cite[Theorem~1.3(2)]{McRaeYang}.
\end{proof}

We call a positive label $j$ a \emph{wall label} if $p\mid j$ and a
\emph{non-wall label} otherwise. For $1\le r\le p-1$, define the reflection chain
\begin{equation}\label{eq:reflection-chain}
s_{2m}(r)=2mp+r,\qquad s_{2m+1}(r)=2(m+1)p-r
\qquad(m\ge0).
\end{equation}
When the initial label is fixed, we suppress it and write $s_i=s_i(r)$.

\begin{lem}\label{lem:reflection-partition}
The non-wall positive labels are partitioned by the reflection chains: \(\{j\geq1:p\nmid j\}=\bigsqcup_{r=1}^{p-1}\{s_i(r):i\geq0\}.\) More explicitly, if $j=np+r$ with $n\geq0$ and $1\leq r\leq p-1$, then $j=s_n(r)$ when $n$ is even and $j=s_n(p-r)$ when $n$ is odd.
\end{lem}

\begin{proof}
If $n=2m$, then $j=2mp+r=s_{2m}(r)$.  If $n=2m+1$, then $j=(2m+1)p+r=2(m+1)p-(p-r)=s_{2m+1}(p-r)$.  Conversely every label in \eqref{eq:reflection-chain} is positive and not divisible by $p$.  The Euclidean decomposition $j=np+r$ with $n\ge0$ and $1\le r<p$ is unique, and its parity determines uniquely which of $r$ or $p-r$ is the initial label of the chain.  Hence distinct reflection chains do not meet.
\end{proof}

For later comparison of Hom spaces, we record Proposition~\ref{prop:MY-affine-literature-package}(iv) in reflection-chain notation.

\begin{cor}\label{cor:source-Hom-package}
For a fixed reflection chain \eqref{eq:reflection-chain},
\[
\dim_\C\Hom(P_{s_i},P_{s_j})=
\begin{cases}
2,&i=j\geq1,\\
1,&i=j=0,\\
1,&|i-j|=1,\\
0,&|i-j|\geq2.
\end{cases}
\]
If $P_{s_i(r)}$ and $P_{s_j(r')}$ belong to distinct reflection chains, then both Hom spaces between them vanish.  Moreover, the wall projectives are isolated: \(\End(P_{np})=\C\) and \(\Hom(P_{np},P_j)=\Hom(P_j,P_{np})=0\) for \(j\ne np\). For every $i\geq0$, choose nonzero morphisms \(x_i:P_{s_i}\longrightarrow P_{s_{i+1}}, y_i:P_{s_{i+1}}\longrightarrow P_{s_i}.\) Then
\(\nu_{i+1}:=x_i\circ y_i\ne0\) and \(\rad\End(P_{s_{i+1}})=\C\nu_{i+1}\).
\end{cor}

\begin{proof}
The Hom-space statements are exactly Proposition~\ref{prop:MY-affine-literature-package}(iv) after translating its labels to the reflection chain \eqref{eq:reflection-chain}.

For the last assertion, write uniquely \(s_i=np-\rho, s_{i+1}=np+\rho\) with $n\ge1$ and $1\le\rho\le p-1$.  Under the notation of Proposition~\ref{prop:MY-affine-literature-package}(iv), the two adjacent Hom spaces are generated by \(F^-_{np+\rho}:P_{np+\rho}\longrightarrow P_{np-\rho}\) and \(F^+_{np-\rho}:P_{np-\rho}\longrightarrow P_{np+\rho}\), and, with $\nu:=F^+_{np-\rho}\circ F^-_{np+\rho}$, \(\End(P_{np+\rho})=\C\id_{P_{np+\rho}}\oplus\C\nu, \nu\neq0.\) By Lemma~\ref{lem:fitting-local}, $A:=\End(P_{np+\rho})$ is local because $P_{np+\rho}$ is indecomposable and has finite length.  The element $\nu$ is not invertible.  Indeed, if $\nu=x\circ y$ were invertible, with $y:P_{np+\rho}\to P_{np-\rho}$ and $x:P_{np-\rho}\to P_{np+\rho}$, then $(\nu^{-1}\circ x)\circ y=\id_{P_{np+\rho}}$, so $y$ would be a split monomorphism.  Hence $P_{np+\rho}$ would be a nonzero direct summand of the indecomposable object $P_{np-\rho}$ and therefore $P_{np+\rho}\cong P_{np-\rho}$.  This is impossible: by Proposition~\ref{prop:MY-affine-literature-package}(i), their heads are the nonisomorphic simples $L_{np+\rho}$ and $L_{np-\rho}$, respectively.  Thus $\nu$ is a nonunit, and locality gives $0\neq\nu\in\rad A$.  Since $\dim_\C A=2$ and $\rad A$ is a proper subspace of $A$, it follows that $\rad A=\C\nu$.  Finally, both adjacent Hom spaces are one-dimensional by Proposition~\ref{prop:MY-affine-literature-package}(iv), so any nonzero choices of $x_i$ and $y_i$ are nonzero scalar multiples of the displayed generators; their composite is therefore a nonzero scalar multiple of $\nu$ and spans the same radical.
\end{proof}

We shall use the parity decomposition and the affine minus twist \eqref{eq:minus-twist} from Proposition~\ref{prop:MY-affine-literature-package}(v) without further redefinition.  For example, $V_r$ is homogeneous of parity $r-1\pmod2$, since the $h(0)$-weights of its horizontal top $M_r$ are $r-1,r-3,\ldots,1-r$ and affine modes change these weights by even integers.

For $n\geq1$ and $1\leq r\leq p-1$, Proposition~\ref{prop:MY-affine-literature-package}(v) states that $L^{\mathrm{aff}}_0$ does not act semisimply on $P_{np+r}$.  Fix a nonzero generator $N_{np+r}\in\rad\End(P_{np+r})$.  By Corollary~\ref{cor:source-Hom-package}, \(\End(P_{np+r})=\C\id\oplus\C N_{np+r}, \rad\End(P_{np+r})=\C N_{np+r}.\) Since the Jacobson radical of a finite-dimensional algebra is nilpotent, $N_{np+r}$ is nilpotent.  On the other hand, $N_{np+r}^2\in\rad\End(P_{np+r})=\C N_{np+r}$, so $N_{np+r}^2=\lambda N_{np+r}$ for some $\lambda\in\C$.  If $\lambda\ne0$, then $N_{np+r}^m=\lambda^{m-1}N_{np+r}\ne0$ for every $m\ge1$, contradicting nilpotence.  Hence \(N_{np+r}^2=0.\) 
For $P=P_{np+r}$ with $n\ge1$ and $1\le r\le p-1$, put $N_P:=N_{np+r}$.  The affine conformal nilpotent $D_P$ and residue sectors $P^{[\xi]}$ are those already fixed above.  The next lemma identifies the one-dimensional Jacobson radical $\rad\End(P)=\C N_P$ with the line spanned by $D_P$.

\begin{lem}\label{lem:source-L0-radical}
Let $P=P_{np+r}$ with $n\geq1$ and $1\leq r\leq p-1$, and retain the notation $N_P$ and $D_P$ above.  Then \(D_P=\delta_PN_P \text{for some }\delta_P\in\C^\times.\) Moreover, $P$ has a unique conformal-weight residue $\xi_P\in\C/\Z$, i.e. $P=P^{[\xi_P]}$, and if $\bar\epsilon$ is the parity of $P$ with representative $\epsilon\in\{0,1\}$, then
\begin{equation}\label{eq:source-twist-nilpotent}
\theta^-_P=\lambda_P(\id+2\pi iD_P)=\lambda_P(\id+\gamma_PN_P),\qquad
\lambda_P=(-1)^\epsilon e^{2\pi i\xi_P},\qquad
\gamma_P=2\pi i\delta_P\ne0.
\end{equation}
\end{lem}

\begin{proof}
By Corollary~\ref{cor:source-Hom-package}, write $D_P=a\id+bN_P$.  If $a\ne0$, then $D_P=a(\id+(b/a)N_P)$ is invertible because $N_P^2=0$, contradicting local nilpotence.  Thus $a=0$.  Proposition~\ref{prop:MY-affine-literature-package}(v) says that $L^{\mathrm{aff}}_0$ is not semisimple on $P$, so $D_P\ne0$.  Hence $b=\delta_P\in\C^\times$.

Since \(P=\bigoplus_{\xi\in\C/\Z}P^{[\xi]}\) and the projective cover $P$ is indecomposable, exactly one residue summand is nonzero; call its class $\xi_P$.  On $P_{[\lambda]}$ one has $L^{\mathrm{aff}}_0=\lambda\id+D_P$, and every occurring $\lambda$ lies in $\xi_P$.  Since $D_P^2=\delta_P^2N_P^2=0$, one has $e^{2\pi iD_P}=\id+2\pi iD_P$, so \eqref{eq:minus-twist} gives \eqref{eq:source-twist-nilpotent}.
\end{proof}

\subsection{The McRae--Yang functor and source tensor input}\label{subsec:MY-input}

The Virasoro algebra is
\[
\Vir=\bigoplus_{n\in\Z}\C L_n\oplus\C\mathbf c,
\qquad [L_m,L_n]=(m-n)L_{m+n}+\frac{m^3-m}{12}\delta_{m+n,0}\mathbf c,
\qquad [L_n,\mathbf c]=0.
\]
Put $\Vir_{\ge0}:=\bigoplus_{n\ge0}\C L_n\oplus\C\mathbf c$. For $c,h\in\C$, let $\C v_{c,h}$ be the one-dimensional $\Vir_{\ge0}$-module on which $\mathbf c$ acts by $c$, $L_0$ acts by $h$, and $L_n$ acts by $0$ for $n>0$. We write \(M(c,h):=U(\Vir)\otimes_{U(\Vir_{\ge0})}\C v_{c,h}\) for the Virasoro Verma module and $L(c,h)$ for its unique irreducible quotient. A nonzero vector $v$ satisfying $L_0v=hv$ and $L_nv=0$ for all $n>0$ is a \emph{highest-weight vector} of weight $h$, and a Virasoro module generated by such a vector is a \emph{highest-weight module}. If a grading-restricted generalized Virasoro module $W$ satisfies $\operatorname{Spec}(W)\subset h+\Z_{\ge0}$ and $W_{[h]}\ne0$, we call $h$ its \emph{lowest conformal weight} and $W_{[h]}$ its \emph{lowest generalized conformal-weight space}; when $W$ is ordinary, this is simply its \emph{lowest conformal-weight space}. In particular, a highest-weight module generated by a highest-weight vector of weight $h$ has lowest conformal weight $h$. We use ``highest-weight vector'' and ``highest-weight module'' for the Virasoro highest-weight notions throughout. A highest-weight vector in a highest-weight module that is not proportional to its chosen highest-weight generator is called a \emph{singular vector}. Since $L_{-1}v_{c,0}$ is singular, the quotient \(V_c^{\mathrm{Vir}}:=M(c,0)\big/U(\Vir)L_{-1}v_{c,0}\) is the universal Virasoro vertex operator algebra of central charge $c$; its conformal vector is the image $\omega=L_{-2}v_{c,0}$ and, on every $V_c^{\mathrm{Vir}}$-module $W$, one has $Y_W(\omega,z)=\sum_{n\in\Z}L_nz^{-n-2}$ \cite[Sections~2.1--2.2]{McRaeSopin}.  Thus the operators $L_n$, and in particular $L_0$, used below are the Virasoro modes of this conformal vector.

We use the module convention of \cite[Section~2.2]{McRaeSopin}.  The Virasoro modules considered here are restricted: $L_nw=0$ for every $w\in W$ and all sufficiently large $n$. Writing $Y_W(u,z)=\sum_{m\in\Z}u_mz^{-m-1}$ for the module vertex operator and $\wt u$ for the $L_0$-eigenvalue of a homogeneous vector $u\in V_c^{\mathrm{Vir}}$, set \(C_1(W):=\operatorname{span}_{\C}\{u_{-1}w:u\in V_c^{\mathrm{Vir}}\text{ homogeneous},\ \wt u>0,\ w\in W\}.\) The module $W$ is \emph{$C_1$-cofinite} if $\dim W/C_1(W)<\infty$.  Following \cite[Theorem~3.4]{McRaeSopin} and the convention used in \cite[Section~7.3]{McRaeYang}, let $\Oc_c$ denote the category of $C_1$-cofinite grading-restricted generalized $V_c^{\mathrm{Vir}}$-modules.  If $H_c:=\{h\in\C:M(c,h)\text{ is reducible}\}$, then the Virasoro tensor-category theorem identifies $\Oc_c$ with the category of finite-length $V_c^{\mathrm{Vir}}$-modules whose composition factors are $L(c,h)$ with $h\in H_c$.  Under this identification $\Oc_c$ is closed under the standard VOA contragredient, and it is a $\C$-linear abelian braided tensor category \cite[Theorem~3.1.4, Section~4, Corollary~4.1.8 and Theorem~4.2.6]{CreutzigEtAl}.  Its HLZ tensor product is denoted by $\boxtimes$, with the ambient category determining which tensor product is meant, and the tensoring functors $X\boxtimes-$ and $-\boxtimes X$ are right exact for every $X\in\Oc_c$ \cite[Proposition~4.26]{HLZ3}.  Fix coprime integers $p,q\ge2$ and set $c_{p,q}:=1-6(p-q)^2/(pq)$.  At $c=c_{p,q}$ we write $\one_{\Oc}:=V_{c_{p,q}}^{\mathrm{Vir}}$ for the tensor unit.  The standard Virasoro twist is \(\theta_M^{\mathrm{Vir}}=e^{2\pi iL_0} (M\in\Oc_{c_{p,q}}),\) as used in \cite[Theorem~7.15 and its proof]{McRaeYang}.

We next fix the Feigin--Fuchs and Kac-module conventions exactly as in \cite[Section~2.3]{McRaeSopin}.  Let $\mathfrak H$ be the Heisenberg Lie algebra with basis $\{a_n:n\in\Z\}\cup\{\mathbf k_{\mathfrak H}\}$ and relations $[a_m,a_n]=m\delta_{m,-n}\mathbf k_{\mathfrak H}$, with $\mathbf k_{\mathfrak H}$ central.  Put $\mathfrak H_{\ge0}:=\bigoplus_{n\ge0}\C a_n\oplus\C\mathbf k_{\mathfrak H}$.  For $\lambda\in\C$, let $\mathcal F_\lambda:=U(\mathfrak H)\otimes_{U(\mathfrak H_{\ge0})}\C v_\lambda$, where $\mathbf k_{\mathfrak H}v_\lambda=v_\lambda$, $a_0v_\lambda=\lambda v_\lambda$, and $a_nv_\lambda=0$ for $n>0$.  Set \(Q_{p,q}:=\sqrt{\frac{2}{pq}}(q-p)\) and \(\lambda^{(p,q)}_{r,s}:=(1-r)\sqrt{\frac{q}{2p}}-(1-s)\sqrt{\frac{p}{2q}}\). The Fock module $\mathcal F_0$ carries the standard rank-one Heisenberg vertex algebra structure, and each $\mathcal F_\lambda$ is a simple $\mathcal F_0$-module.  For $Q\in\C$, the vector $\omega_Q:=\frac12(a_{-1}^2+Qa_{-2})v_0\in\mathcal F_0$ is a conformal vector, whose modes on each $\mathcal F_\lambda$ give a Virasoro action of central charge $1-3Q^2$.  Taking $Q=Q_{p,q}$ gives central charge $1-3Q_{p,q}^2=c_{p,q}$.  We write $\mathcal F^{(p,q)}_{r,s}:=\mathcal F_{\lambda^{(p,q)}_{r,s}}$ for the resulting Feigin--Fuchs module; its generating vector has conformal weight \(h^{(p,q)}_{r,s}:=\frac{(qr-ps)^2-(p-q)^2}{4pq}.\) For $r,s\ge1$, the Kac module $\Kac^{(p,q)}_{r,s}$ is the Virasoro submodule of $\mathcal F^{(p,q)}_{r,s}$ generated by all vectors of conformal weight strictly less than $h^{(p,q)}_{r,s}+rs$ \cite[Section~2.3]{McRaeSopin}.  The superscript $(p,q)$ is bookkeeping introduced here; it is suppressed in \cite{McRaeSopin}.  When the ordered pair is fixed, we likewise write $\Kac_{r,s}:=\Kac^{(p,q)}_{r,s}$.

Since $1\le q$, \cite[Remark~2.2]{McRaeSopin} shows that every one-row Kac module $\Kac^{(p,q)}_{r,1}$ is singly generated by a highest-weight vector of lowest conformal weight $h^{(p,q)}_{r,1}$ and is the quotient of $M(c_{p,q},h^{(p,q)}_{r,1})$ by the submodule generated by a singular vector of conformal weight $h^{(p,q)}_{r,1}+r$. Equivalently, this singular vector has relative degree $r$. By \cite[Proposition~2.2.1(3)]{CreutzigEtAl}, a Virasoro Verma module has at most one singular vector up to scalar at any fixed relative degree, so the degree-$r$ singular line is unique. We set \(K_j:=\Kac^{(p,q)}_{j,1}, S_j:=L\bigl(c_{p,q},h^{(p,q)}_{j,1}\bigr) (j\ge1).\) Thus $S_j$ is the simple head of $K_j$. Since the unique relative-degree-one singular line in $M(c_{p,q},0)$ is $\C L_{-1}v_{c_{p,q},0}$, one has $K_1\cong V_{c_{p,q}}^{\mathrm{Vir}}=\one_{\Oc}$, while $K_2=\Kac^{(p,q)}_{2,1}$. More generally, \(h^{(q,p)}_{j,1}=h^{(p,q)}_{1,j} (j\ge1),\) and Remark~2.2 of \cite{McRaeSopin}, together with uniqueness of the singular line at fixed degree, gives an isomorphism of Virasoro modules
\begin{equation}\label{eq:Kac-parameter-swap}
\Kac^{(q,p)}_{j,1}\cong\Kac^{(p,q)}_{1,j}\qquad(j\ge1).
\end{equation}
In particular $\Kac^{(q,p)}_{2,1}\cong\Kac^{(p,q)}_{1,2}$, the identification used in \cite[Theorem~7.15]{McRaeYang}.

We now fix the rigid-duality conventions used in the McRae--Yang normalization.  If $X$ is a rigid self-dual object of a tensor category $\mathscr C$, chosen evaluation and coevaluation maps will be denoted \(e_X:X\boxtimes X\longrightarrow\one_{\mathscr C}, i_X:\one_{\mathscr C}\longrightarrow X\boxtimes X,\) and are required to satisfy the zig-zag identities.  For such a self-duality, its \emph{intrinsic dimension} is \(d(X):=e_X\circ i_X\in\End(\one_{\mathscr C}).\) We use the evaluation and coevaluation maps for $V_2$, $K_2=\Kac^{(p,q)}_{2,1}$, $\Kac^{(p,q)}_{1,2}$, and, after swapping the ordered pair, $\Kac^{(q,p)}_{2,1}$ that occur in the normalizations of \cite[Theorem~7.15]{McRaeYang}.  McRae--Sopin prove that the two distinguished Virasoro Kac modules are rigid and self-dual \cite[Theorem~5.7]{McRaeSopin}; their intrinsic dimensions are \cite[Proposition~5.8]{McRaeSopin}
\[
e_{K_2}\circ i_{K_2}=-(e^{\pi iq/p}+e^{-\pi iq/p})\id_{\one_{\Oc}},\qquad e_{\Kac^{(p,q)}_{1,2}}\circ i_{\Kac^{(p,q)}_{1,2}}=-(e^{\pi ip/q}+e^{-\pi ip/q})\id_{\one_{\Oc}}.
\]

We use the following strong-monoidal convention.  Here \emph{right exact} refers to the underlying additive functor, whereas \emph{strong monoidal} refers to the invertibility of the tensorator and unit constraint.  Thus the ``right exact tensor functors'' of \cite[Theorems~6.8 and~7.15]{McRaeYang} are right exact strong monoidal functors in our terminology.  A \emph{braided strong monoidal functor} $F:\mathscr A\to\mathscr B$ is equipped with coherent natural isomorphisms
\[
\Phi_{X,Y}:F(X)\boxtimes F(Y)\xrightarrow{\ \sim\ }F(X\boxtimes Y),\qquad
\Phi_0:\one_{\mathscr B}\xrightarrow{\ \sim\ }F(\one_{\mathscr A}),
\]
satisfying the associativity and unit coherence axioms and the braiding compatibility \(F(c^{\mathscr A}_{X,Y})\circ\Phi_{X,Y}=\Phi_{Y,X}\circ c^{\mathscr B}_{F(X),F(Y)},\) where $c^{\mathscr A}$ and $c^{\mathscr B}$ denote the braidings of $\mathscr A$ and $\mathscr B$, respectively. McRae--Yang use the inverse unit constraint $\varphi:F(\one_{\mathscr A})\to\one_{\mathscr B}$ in Theorem~6.8; in our convention $\varphi=\Phi_0^{-1}$.  Accordingly, for a strong monoidal functor $F:\KL^k(\mathfrak{sl}_2)\to\mathscr B$ and a fixed identification $F(V_2)\cong X$ with a rigid self-dual object $X\in\mathscr B$, we say that $F$ preserves the chosen evaluation and coevaluation of $V_2$ if
\begin{equation}\label{eq:monoidal-duality-normalization}
\Phi_0^{-1}\circ F(e_{V_2})\circ\Phi_{V_2,V_2}=e_X,
\qquad
\Phi_{V_2,V_2}^{-1}\circ F(i_{V_2})\circ\Phi_0=i_X.
\end{equation}
All monoidal structure maps are suppressed from notation after this point.  A monoidal natural transformation is understood to be compatible with the tensorators and unit constraints.

For an ordered coprime pair $(a,b)$ with $a,b\ge2$, set $c_{a,b}:=1-6(a-b)^2/(ab)$.  We denote by \(F^{(a,b)}:\KL^{-2+a/b}(\mathfrak{sl}_2)\longrightarrow\Oc_{c_{a,b}}\) the McRae--Yang functor of \cite[Theorem~7.15]{McRaeYang} normalized by the rigid datum of $\Kac^{(a,b)}_{2,1}$.  For the fixed ordered pair $(p,q)$, we abbreviate $F_{p,q}:=F^{(p,q)}$.  By contrast, $F_{q,p}$ denotes the second McRae--Yang functor associated with the fixed ordered pair $(p,q)$, with source $\KL^{-2+q/p}(\mathfrak{sl}_2)$ and normalized in $\Oc_{c_{p,q}}$ by the rigid datum of $\Kac^{(p,q)}_{1,2}$.  We keep $F^{(q,p)}$ and $F_{q,p}$ distinct until the parameter-swap comparison.

For $M\in\KL^k(\mathfrak{sl}_2)$, twist preservation for $F_{p,q}$ means \(F_{p,q}(\theta_M^-)=\theta^{\mathrm{Vir}}_{F_{p,q}(M)},\) where $\theta^-$ is the affine minus twist fixed in Subsection~\ref{subsec:affine-projectives}.

\begin{prop}\label{prop:MY-functor-package}
Let $k=-2+p/q$ with $p,q\ge2$ coprime.
\begin{enumerate}
\item The functor \(F_{p,q}:\KL^k(\mathfrak{sl}_2)\longrightarrow\Oc_{c_{p,q}}\) is $\C$-linear and right exact and is braided strong monoidal.  With the chosen normalization $F_{p,q}(V_2)\cong K_2$, its monoidal structure satisfies \eqref{eq:monoidal-duality-normalization} with $X=K_2$.  It preserves the source minus twist and the target Virasoro twist in the sense fixed above.
\item Put $k'=-2+q/p$.  Let \(G:\KL^{k'}(\mathfrak{sl}_2)\longrightarrow\Oc_{c_{p,q}}\) be a right exact strong monoidal functor equipped with an identification $G(V_2)\cong\Kac^{(p,q)}_{1,2}$ satisfying \eqref{eq:monoidal-duality-normalization} with $X=\Kac^{(p,q)}_{1,2}$.  Then $G$ is naturally isomorphic to $F_{q,p}$ as a tensor functor.  Moreover, $F_{q,p}$ is braided and preserves the corresponding source minus twist and the target Virasoro twist.
\end{enumerate}
\end{prop}

\begin{proof}
(i) By \cite[Theorem~7.15]{McRaeYang}, $F_{p,q}$ is the right exact braided tensor functor normalized by the rigid datum $(K_2,e_{K_2},i_{K_2})$, and it preserves twists when the source is equipped with the minus twist.  In the convention fixed above, its tensor structure is strong monoidal.  The construction of the underlying functor in \cite[Theorem~6.8]{McRaeYang} uses the right exact $\C$-linear extension theorem \cite[Theorem~6.7]{McRaeYang}, so the underlying functor is $\C$-linear.  Finally, McRae--Yang use the unit isomorphism $\varphi:F_{p,q}(V_1)\to\one_{\Oc}$; since $\varphi=\Phi_0^{-1}$ in our convention, their rigid normalization is exactly \eqref{eq:monoidal-duality-normalization}.

(ii) For the level $k'=-2+q/p$, the intrinsic dimension required in \cite[Theorem~6.8]{McRaeYang} is \(-\bigl(e^{\pi ip/q}+e^{-\pi ip/q}\bigr),\) which is precisely the intrinsic dimension of $\Kac^{(p,q)}_{1,2}$ by \cite[Proposition~5.8]{McRaeSopin}.  The target tensor product on $\Oc_{c_{p,q}}$ is right exact in each variable, as recorded above.  Hence the universal property of \cite[Theorem~6.8]{McRaeYang} shows that any right exact strong monoidal functor with the stated rigid datum is naturally isomorphic, as a tensor functor, to the functor normalized by $\Kac^{(p,q)}_{1,2}$, namely $F_{q,p}$.  Its braidedness and twist preservation are the corresponding conclusions of \cite[Theorem~7.15]{McRaeYang}.
\end{proof}

\subsection{One-row Virasoro Kac modules and conformal-residue sectors}\label{subsec:one-row-residue}

The Virasoro category $\Oc_{c_{p,q}}$, its braided tensor product and twist, the ordered-pair Kac modules, and the abbreviations $K_j$ and $S_j$ were fixed in Subsection~\ref{subsec:MY-input}.  For later arithmetic set
\begin{equation}\label{eq:weights} h_j:=h^{(p,q)}_{j,1}=\frac{(qj-p)^2-(p-q)^2}{4pq}.\end{equation}
For the arithmetic formulas below only, we extend $h_j$ to every $j\in\Z$ by \eqref{eq:weights}. No Virasoro module $K_j$ or $S_j$ is attached to a non-positive label. The symbol $K_0=0$ will be used only as a bookkeeping convention in recursion formulas; it is not a Virasoro Kac module. The labels $S_j$ are likewise labels rather than an assertion that all $S_j$ are pairwise nonisomorphic; the only collision relevant below is isolated explicitly in Corollary~\ref{cor:positive-label-collisions}.

For $M\in\Oc_{c_{p,q}}$, let $\ell(M)$ denote its Jordan--H\"older length and write $[M]$ for its class in the Grothendieck group $K_0(\Oc_{c_{p,q}})$. For a rigid object $X$ in a tensor category, $X^\vee$ denotes a chosen dual object. Our Yoneda convention is that $\Ext^1_{\mathcal A}(X,Y)$ consists of Yoneda equivalence classes of short exact sequences $0\longrightarrow Y\longrightarrow E\longrightarrow X\longrightarrow0$. For a grading-restricted Virasoro module $M$, set \(\ch_zM:=\sum_{\lambda\in\operatorname{Spec}(M)}\dim M_{[\lambda]}z^\lambda.\) For ordinary $M$, this equals $\operatorname{Tr}_M z^{L_0}$.

\begin{lem}\label{lem:Oc-Hom-finite-KS}
The category $\Oc_{c_{p,q}}$ is Hom-finite and Krull--Schmidt. Explicitly, every object is a finite direct sum of indecomposable objects, and such a decomposition is unique up to permutation and isomorphism of the summands.
\end{lem}

\begin{proof}
For simple objects $S=L(c_{p,q},h)$ and $T=L(c_{p,q},h')$, one has $\dim_\C\Hom(S,T)\le1$: a nonzero morphism is an isomorphism, while $S_{[h]}=\C v_h$ and $S$ is generated by $v_h$, so every endomorphism of $S$ is scalar. We prove Hom-finiteness by induction on $\ell(M)+\ell(N)$. If $M$ and $N$ are simple, this gives the base case. Suppose $\ell(M)+\ell(N)>2$. If $\ell(N)>1$, choose $0\longrightarrow N_0\longrightarrow N\longrightarrow T\longrightarrow0$ with $T$ simple. Left exactness of $\Hom(M,-)$ gives $0\longrightarrow\Hom(M,N_0)\longrightarrow\Hom(M,N)\longrightarrow\Hom(M,T)$; the outer spaces are finite dimensional by induction, so $\Hom(M,N)$ is finite dimensional. If $\ell(N)=1$, then $\ell(M)>1$; choose $0\longrightarrow M_0\longrightarrow M\longrightarrow S\longrightarrow0$ with $S$ simple. Left exactness of $\Hom(-,N)$ gives $0\longrightarrow\Hom(S,N)\longrightarrow\Hom(M,N)\longrightarrow\Hom(M_0,N)$, and induction again gives finite dimensionality. Thus $\Oc_{c_{p,q}}$ is Hom-finite.

It remains to verify the Krull--Schmidt property. Since every object has finite length, existence of a finite indecomposable decomposition follows by induction on length: if a nonzero object $M$ is not indecomposable, write $M=A\oplus B$ with $A,B\neq0$; then $\ell(A),\ell(B)<\ell(M)$, and the induction hypothesis applies to both summands. By Lemma~\ref{lem:fitting-local}, the endomorphism algebra of every indecomposable object is local. The Krull--Schmidt--Azumaya theorem \cite[Chapter~2]{FacchiniModuleTheory}, in the form that a finite direct-sum decomposition into objects with local endomorphism algebras is unique up to permutation and isomorphism of the summands, therefore gives the asserted uniqueness.
\end{proof}

We collect the McRae--Sopin one-row results used in the sequel into one interface.

\begin{prop}\label{prop:Virasoro-literature-package}
For the ordered pair $(p,q)$ and the objects $K_j,S_j$ fixed above, the following hold.
\begin{enumerate}
\item Each $K_j$ is ordinary, and \(\ch_zK_j=\frac{z^{h_j}(1-z^j)}{\prod_{m\ge1}(1-z^m)}.\) Moreover, $\End(K_j)=\C\id_{K_j}$.
\item For every $n\ge1$, the wall module $K_{np}$ is simple. For $n\ge0$ and $1\le r\le p-1$ there is a nonsplit exact sequence
\begin{equation}\label{eq:Kac-length-two}
0\longrightarrow S_{(n+2)p-r}\longrightarrow K_{np+r}\longrightarrow S_{np+r}\longrightarrow0.
\end{equation}
Equivalently, for $n\ge1$, \(0\longrightarrow S_{np+r}\longrightarrow K_{np-r}\longrightarrow S_{np-r}\longrightarrow0.\)
\item The distinguished object $K_2$ is rigid and self-dual. Consequently $K_2\boxtimes-$ is exact.
\item For every $j\ge1$ there is an exact sequence
\begin{equation}\label{eq:K2-fusion}
0\longrightarrow K_{j-1}\longrightarrow K_2\boxtimes K_j\longrightarrow K_{j+1}\longrightarrow0,
\end{equation}
with the convention $K_0=0$; it splits if and only if $p\nmid j$. If $p=2$ and $m\ge1$, then $K_2\boxtimes K_{2m}$ is logarithmic.
\end{enumerate}
\end{prop}

\begin{proof}
For (i), in the notation of \cite[Section~2.3]{McRaeSopin}, $K_j=\mathcal K_{j,1}$ is a Virasoro submodule of the Feigin--Fuchs module $\mathcal F_{j,1}$. The $L_0$-action on the standard Heisenberg PBW basis of $\mathcal F_{j,1}$ is diagonal, so $\mathcal F_{j,1}$, and hence $K_j$, is ordinary. The Introduction of \cite{McRaeSopin} records that $\mathcal K_{r,s}$ has the same character as the Verma module of lowest conformal weight $h_{r,s}$ modulo a singular vector of conformal weight $h_{r,s}+rs$. Specializing to $s=1$ gives \(\ch_zK_j=\frac{z^{h_j}(1-z^j)}{\prod_{m\ge1}(1-z^m)}\). By the one-row description above, $K_j$ is a quotient of $M(c_{p,q},h_j)$ by a submodule generated in positive relative degree $j$. Hence $(K_j)_{[h_j]}$ is the image of the one-dimensional line $\C v_{c_{p,q},h_j}$ and is therefore one-dimensional. Since $K_j$ is generated by this highest-weight vector, every endomorphism of $K_j$ is determined by its scalar action on $(K_j)_{[h_j]}$. Thus $\End(K_j)=\C\id_{K_j}$.

For (ii), specialize the Kac-module composition-series diagrams in \cite[Section~2.3]{McRaeSopin} to $\mathcal K_{j,1}$. For $j=np$ with $n\ge1$, the relevant boundary diagram has length one, so $K_{np}$ is simple. For $j=np+r$ with $n\ge0$ and $1\le r\le p-1$, the relevant bulk diagram has length two: its head is $S_{np+r}$ and its simple submodule is $S_{(n+2)p-r}$. This gives the exact sequence \eqref{eq:Kac-length-two}, which is nonsplit. Indeed, the one-row description above says that $K_{np+r}$ is generated by a highest-weight vector of lowest conformal weight $h_{np+r}$. A direct calculation from \eqref{eq:weights} gives \(h_{(n+2)p-r}-h_{np+r}=(p-r)((n+1)q-1)>0\). If the sequence split, then $K_{np+r}\cong S_{(n+2)p-r}\oplus S_{np+r}$, while every vector of conformal weight $h_{np+r}$ would lie in the second summand; such a vector cannot generate the nonzero first summand, contradicting the stated cyclicity. Replacing $n$ by $n-1$ and $r$ by $p-r$ gives, for $n\ge1$, \(0\longrightarrow S_{np+r}\longrightarrow K_{np-r}\longrightarrow S_{np-r}\longrightarrow0\), which is the reflected form.

For (iii), by \cite[Theorem~5.7]{McRaeSopin}, $K_2=\mathcal K_{2,1}$ is rigid and self-dual in $\Oc_{c_{p,q}}$. Rigidity gives a right adjoint to $K_2\boxtimes-$, so this functor is left exact. It is right exact by the general tensor-category property recorded in Subsection~\ref{subsec:MY-input}. Therefore $K_2\boxtimes-$ is exact.

For (iv), the second exact sequence of \cite[Theorem~6.5]{McRaeSopin} is
\[
0\longrightarrow\mathcal K_{r-1,s}\longrightarrow\mathcal K_{2,1}\boxtimes\mathcal K_{r,s}\longrightarrow\mathcal K_{r+1,s}\longrightarrow0,
\]
and it splits if and only if $p\nmid r$. Taking $(r,s)=(j,1)$ gives the asserted one-row sequence. Finally, \cite[Remark~6.11]{McRaeSopin} states more generally that $\mathcal K_{2,1}\boxtimes\mathcal K_{mp,s}$ is logarithmic for every $m\ge1$ and $s\ge1$; taking $p=2$ and $s=1$ gives that $K_2\boxtimes K_{2m}$ is logarithmic.
\end{proof}

\begin{cor}\label{cor:positive-label-collisions}
For $j,\ell>0$, \(S_j\cong S_\ell \Longleftrightarrow \bigl(j=\ell\bigr)\ \text{or}\ \bigl(q=2\ \text{and}\ j+\ell=p\bigr).\) In the second case $p$ is odd.
\end{cor}

\begin{proof}
For arbitrary integers $u,v$, direct subtraction in \eqref{eq:weights} gives \(h_u-h_v=(u-v)(q(u+v)-2p)/(4p).\) Hence $h_j=h_\ell$ if and only if $j=\ell$ or $q(j+\ell)=2p$. Since an irreducible highest-weight Virasoro module at fixed central charge is determined by its lowest conformal weight, this is exactly the condition for $S_j\cong S_\ell$. In the second case coprimality and $q\ge2$ force $q=2$ and $j+\ell=p$, and then $p$ is odd.
\end{proof}

Since $K_j=\Kac_{j,1}$ is generated by a highest-weight vector of conformal weight $h_j$ and Virasoro modes shift conformal weights by integers, $K_j=K_j^{[h_j]}$ for every $j\ge1$.

\begin{lem}\label{lem:block-projection}
For every short exact sequence $0\to A\to B\to C\to0$ in $\Oc_{c_{p,q}}$ and every $\xi\in\C/\Z$, the induced sequence \(0\longrightarrow A^{[\xi]}\longrightarrow B^{[\xi]}\longrightarrow C^{[\xi]}\longrightarrow0\) is exact. If $M$ and $N$ are supported in distinct conformal-residue classes, then
\(\Hom_{\Oc_{c_{p,q}}}(M,N)=0\) and \(\Ext^1_{\Oc_{c_{p,q}}}(M,N)=0\).
\end{lem}

\begin{proof}
Exactness at $A^{[\xi]}$ and $B^{[\xi]}$ is immediate since module homomorphisms preserve generalized conformal weights modulo $\Z$. For surjectivity, let $c\in C^{[\xi]}$ and choose $b\in B$ mapping to $c$. Write $b=\sum_\eta b_\eta$ according to the conformal-residue decomposition of $B$. Since $b_\eta$ maps into $C^{[\eta]}$ and $c\in C^{[\xi]}$, the component $b_\xi$ maps to $c$. Thus $B^{[\xi]}\to C^{[\xi]}$ is surjective, proving exactness.

Now suppose that $M=M^{[\xi]}$ and $N=N^{[\eta]}$ with $\xi\ne\eta$. If $f:M\to N$ is a module homomorphism, then $f(M^{[\xi]})\subseteq N^{[\xi]}=0$. Hence \(\Hom_{\Oc_{c_{p,q}}}(M,N)=0\). 
Finally, consider an extension \(0\longrightarrow N\longrightarrow E\longrightarrow M\longrightarrow0\). Taking the $\xi$-residue sector and using the exactness just proved gives \(0\longrightarrow N^{[\xi]}\longrightarrow E^{[\xi]}\longrightarrow M^{[\xi]}\longrightarrow0\). Since $N^{[\xi]}=0$ and $M^{[\xi]}=M$, the map $E^{[\xi]}\to M$ is an isomorphism. Its inverse, followed by the inclusion $E^{[\xi]}\hookrightarrow E$, is a section of $E\to M$. Thus every such extension splits. By the Yoneda convention fixed above, its extension class is therefore zero; since the extension was arbitrary,
\(\Ext^1_{\Oc_{c_{p,q}}}(M,N)=0\).
\end{proof}

\begin{lem}\label{lem:exponential-logarithmic}
Let $M\in\Oc_{c_{p,q}}$. Then $e^{2\pi iL_0}$ is diagonalizable on $M$ if and only if $L_0$ is semisimple on $M$, equivalently, if and only if $M$ is ordinary.
\end{lem}

\begin{proof}
Fix $\lambda\in\operatorname{Spec}(M)$ and set $D_\lambda:=D_M|_{M_{[\lambda]}}$. By the definition of $D_M$, the operator $D_\lambda$ is nilpotent and \(L_0|_{M_{[\lambda]}}=\lambda\id_{M_{[\lambda]}}+D_\lambda\). Consequently, \(e^{2\pi iL_0}|_{M_{[\lambda]}}=e^{2\pi i\lambda}e^{2\pi iD_\lambda}\), where $e^{2\pi iD_\lambda}$ is the finite sum of its power series.

Suppose that $e^{2\pi iL_0}$ is diagonalizable on $M$. Its restriction to the invariant finite-dimensional subspace $M_{[\lambda]}$ is diagonalizable. Since \(e^{2\pi iD_\lambda}=e^{-2\pi i\lambda}e^{2\pi iL_0}|_{M_{[\lambda]}}\), the operator $e^{2\pi iD_\lambda}$ is also diagonalizable. On the other hand, $D_\lambda$ is nilpotent, so $e^{2\pi iD_\lambda}$ is unipotent. A diagonalizable unipotent operator has all eigenvalues equal to $1$ and hence is the identity. Therefore $e^{2\pi iD_\lambda}=\id_{M_{[\lambda]}}$. Write $e^{2\pi iz}-1=zQ(z)$, where $Q(0)=2\pi i\ne0$. Then $D_\lambda Q(D_\lambda)=0$. Since $D_\lambda$ is nilpotent and $Q(0)\ne0$, the operator $Q(D_\lambda)$ is invertible; hence $D_\lambda=0$. As $\lambda$ was arbitrary, $D_M=0$ on $M$, so $L_0$ is semisimple.

Conversely, if $L_0$ is semisimple, then $M$ is the algebraic direct sum of its $L_0$-eigenspaces, and $e^{2\pi iL_0}$ acts on each eigenspace by the scalar $e^{2\pi i\lambda}$. Hence $e^{2\pi iL_0}$ is diagonalizable.
\end{proof}

\subsection{Nakano's extension interface and Ext input}\label{subsec:Nakano-extension-interface}

Set $p_+=\min\{p,q\}$ and $p_-=\max\{p,q\}$. Because $p,q\ge2$ are coprime, $p\ne q$, and hence $2\le p_+<p_-$ and $\gcd(p_+,p_-)=1$. Thus the standing numerical hypotheses of \cite{Nakano} are satisfied. Since the Virasoro central charge is symmetric in its two parameters, $c_{p,q}=c_{q,p}=c_{p_+,p_-}$.

When the central charge is $c_{p,q}$ and there is no ambiguity, write $M(h):=M(c_{p,q},h)$ and $L(h):=L(c_{p,q},h)$.

For a Virasoro module $M$ on which the central element acts as a scalar, we use the same generalized $L_0$-weight convention as above: \(M_{[\lambda]}:=\{m\in M:(L_0-\lambda)^Nm=0\text{ for some }N\ge1\}\) and \(\operatorname{Spec}(M):=\{\lambda\in\C:M_{[\lambda]}\ne0\}.\) We call $M$ \emph{grading-restricted generalized} if \(M=\bigoplus_{\lambda\in\C}M_{[\lambda]},\) every $M_{[\lambda]}$ is finite dimensional, and, for each fixed $\lambda\in\C$, one has $M_{[\lambda-n]}=0$ for all sufficiently large integers $n$. By $[L_0,L_n]=-nL_n$, such a module is automatically restricted in the Virasoro sense used in Subsection~\ref{subsec:MY-input}. Hence, when $\mathbf c$ acts as $c\,\id$, it may equivalently be regarded as a grading-restricted generalized module for the universal Virasoro vertex operator algebra $V_c^{\mathrm{Vir}}$ in the convention fixed there.

Write \(\Vir_-:=\bigoplus_{n<0}\C L_n\), \(\Vir_0:=\C L_0\oplus\C\mathbf c\), and \(\Vir_{>0}:=\bigoplus_{n>0}\C L_n\), so that $\Vir=\Vir_-\oplus\Vir_0\oplus\Vir_{>0}$ as a vector space. We use the Virasoro anti-involution $\sigma$ defined by $\sigma(L_n)=L_{-n}$ and $\sigma(\mathbf c)=\mathbf c$, extended anti-multiplicatively to $U(\Vir)$; it will also be used below for Shapovalov forms. For a grading-restricted generalized $V_c^{\mathrm{Vir}}$-module $M$, let \(M':=\bigoplus_{\lambda\in\C}M_{[\lambda]}^*\) denote its standard VOA contragredient in the sense of \cite[Theorem~2.34 and Definition~2.35]{HLZ1}. On the underlying Virasoro module one has \((L_n\varphi)(m)=\varphi(L_{-n}m)\), equivalently \((x\varphi)(m)=\varphi(\sigma(x)m)\) for $x\in U(\Vir)$. Thus this VOA contragredient coincides with the restricted contragredient used in \cite[Definition~3.14]{Nakano}. From this point on, a prime attached directly to a Virasoro-module symbol, such as $M'$ or $L(h)'$, always means this contragredient; unrelated auxiliary submodules will not be denoted by a bare prime.

For a finite-length Virasoro module $M$, set $\Soc^0(M)=0$ and, inductively, define $\Soc^{j+1}(M)$ to be the inverse image of $\Soc\bigl(M/\Soc^j(M)\bigr)$ under the natural quotient map. Equivalently, \(\Soc^{j+1}(M)/\Soc^j(M)=\Soc\bigl(M/\Soc^j(M)\bigr).\) This is Nakano's socle series $\mathrm S_j(M)$ from \cite[Definition~2.2]{Nakano}; we use $\Soc^j$ to avoid conflict with the global one-row simple notation $S_j$ fixed in Subsection~\ref{subsec:MY-input}.

For comparison with the free-field notation of \cite[Sections~2 and~5]{Nakano}, we use the Heisenberg algebra $\mathfrak H$ and Fock modules $\mathcal F_\lambda$ fixed in Subsection~\ref{subsec:MY-input}. Following \cite[(2.1), (2.3), (2.4)]{Nakano}, set \(\alpha_+=\sqrt{\frac{2p_-}{p_+}}\), \(\alpha_-=-\sqrt{\frac{2p_+}{p_-}}\), \(\alpha_0=\alpha_++\alpha_-\), and \(h_\alpha=\frac12\alpha(\alpha-\alpha_0)\), and, for $r,s,n\in\Z$, \(\alpha_{r,s;n}=\frac{1-r}{2}\alpha_++\frac{1-s}{2}\alpha_-+\frac{\sqrt{2p_+p_-}}{2}n\) and \(h_{r,s;n}:=h_{\alpha_{r,s;n}}\). Thus $\alpha_0=Q_{p_+,p_-}$ and $\alpha_{r,s;0}=\lambda^{(p_+,p_-)}_{r,s}$. We also write $\alpha_{r,s}:=\alpha_{r,s;0}$ and $h_{r,s}:=h_{r,s;0}$. Equipped with the Virasoro action determined by the conformal vector $\omega_{\alpha_0}$ of Subsection~\ref{subsec:MY-input}, $\mathcal F_\alpha$ is the Fock module denoted $F_\alpha$ in \cite{Nakano}, with central charge $c_{p_+,p_-}$ and lowest conformal weight $h_\alpha$.

Following \cite[Section~5.1]{Nakano}, set \(\mathcal A_{p_+,p_-}:=\{\alpha_{r,s;n}:r,s,n\in\Z\}\) and \(H_{p_+,p_-}:=\{h_\alpha:\alpha\in\mathcal A_{p_+,p_-}\}\); our $\mathcal A_{p_+,p_-}$ is the set denoted $A_{p_+,p_-}$ in \cite{Nakano}. By a \emph{Felder complex} below we mean one of the complexes of Fock modules in \cite[Proposition~2.7]{Nakano}. That proposition gives the precise cohomology of these complexes; the present argument uses only the order of their Fock-module terms. Thus ``$\mathcal F_{\alpha_1},\mathcal F_{\alpha_2},\mathcal F_{\alpha_3}$ are three consecutive terms'' means that, in one such complex, they occur in the displayed order \(\cdots\longrightarrow \mathcal F_{\alpha_1}\longrightarrow \mathcal F_{\alpha_2}\longrightarrow \mathcal F_{\alpha_3}\longrightarrow\cdots.\) Following \cite[Definition~5.6]{Nakano}, let $\mathcal T_{p_+,p_-}\subset\mathcal A_{p_+,p_-}^3$ be the set of triples $\tau=(\alpha_1,\alpha_2,\alpha_3)$ such that $h_{\alpha_1}\le h_{\alpha_2}<h_{\alpha_3}$ and the three Fock modules occur as consecutive terms of one Felder complex. We shall also need only the exceptional subset $\mathcal T^0_{p_+,p_-}$ of \cite[Definition~5.7]{Nakano}; by \cite[Remark~5.8]{Nakano} it is equivalently the explicit set
\[
\begin{aligned}
\mathcal T^0_{p_+,p_-}={}&\{(\alpha_{r,p_-;-1},\alpha_{r,p_-;0},\alpha_{r,p_-;1}):1\le r<p_+\}\\
&\cup\{(\alpha_{p_+,p_--s;1},\alpha_{p_+,s;0},\alpha_{p_+,p_--s;-1}):1\le s<p_-\}.
\end{aligned}
\]
In particular, every $\tau\in\mathcal T^0_{p_+,p_-}$ has $h_{\alpha_1}=h_{\alpha_2}$, so a triple with $h_{\alpha_1}<h_{\alpha_2}<h_{\alpha_3}$ lies in $\mathcal T_{p_+,p_-}\setminus\mathcal T^0_{p_+,p_-}$.

Nakano's abelian category $\mathcal L_{c_{p_+,p_-}}\text{-Mod}$ is defined in \cite[Definition~3.14]{Nakano}. In \cite[Definition~5.1]{Nakano}, he defines the full subcategory $\mathcal L_{c_{p_+,p_-}}\text{-mod}$ by requiring a finite socle series and simple composition factors with lowest conformal weights in $H_{p_+,p_-}$; his notation $\Ext^1_{\mathcal L}$ refers to extensions in this subcategory. To distinguish this source category from the Serre subcategory used below, write
\[
\mathscr N:=\mathcal L_{c_{p_+,p_-}}\text{-mod},
\qquad \Ext^1_{\mathscr N}:=\Ext^1_{\mathcal L}\ \text{of \cite{Nakano}}.
\]
Recall that a full subcategory of an abelian category is called a \emph{Serre subcategory} if it is closed under subobjects, quotients, and extensions.

For comparison with \cite{BoeNakanoWiesner}, let $\mathcal O^{\mathrm{Vir}}_{c_{p,q}}$ denote the full subcategory of their Virasoro category $\mathcal O$ on which the central element acts as $c_{p,q}$. Thus its objects are finitely generated Virasoro modules on which $L_0$ acts semisimply and $\Vir_{>0}$ acts locally finitely; see \cite[Section~2.2]{BoeNakanoWiesner}.

Since the central charge is fixed, identify the affine hyperplane \(\{\lambda\in\Vir_0^*: \lambda(\mathbf c)=c_{p,q}\}\) with $\C$ by sending $h\in\C$ to the weight $\lambda_h$ satisfying $\lambda_h(L_0)=h$ and $\lambda_h(\mathbf c)=c_{p,q}$. Under this identification, the Verma and simple modules attached to $\lambda_h$ are precisely $M(h)$ and $L(h)$.

Following \cite[Section~2.2 and (2)]{BoeNakanoWiesner}, for $h,h'\in\C$ define
\[
h\preceq h'\quad\Longleftrightarrow\quad L(h)\text{ is a subquotient of }M(h').
\]
Let $\sim$ be the equivalence relation generated by $\preceq$. For an equivalence class $\Lambda\subset\C$, let $\mathcal B_\Lambda$ be the full subcategory of $\mathcal O^{\mathrm{Vir}}_{c_{p,q}}$ whose simple subquotients are labelled by elements of $\Lambda$. These are the blocks of $\mathcal O^{\mathrm{Vir}}_{c_{p,q}}$. More precisely, \cite[Section~2.2]{BoeNakanoWiesner} states that every $M\in\mathcal O^{\mathrm{Vir}}_{c_{p,q}}$ decomposes as \(M=\bigoplus_{\Lambda} M_{\Lambda},\; M_{\Lambda}\in\mathcal B_{\Lambda}\). Consequently, the block decomposition gives the following two $\Ext^1$ interfaces.  If $S$ and $T$ are simple objects in the same block $\mathcal B$, then every extension of $S$ by $T$ has all its composition factors in $\mathcal B$; since $\mathcal B$ is full, the identity on short exact sequences induces a canonical isomorphism
\begin{equation}\label{eq:Vir-same-block-Ext}
\Ext^1_{\mathcal O^{\mathrm{Vir}}_{c_{p,q}}}(S,T)
\cong \Ext^1_{\mathcal B}(S,T).
\end{equation}
If $S$ and $T$ lie in distinct blocks, then
\begin{equation}\label{eq:Vir-cross-block-Ext}
\Ext^1_{\mathcal O^{\mathrm{Vir}}_{c_{p,q}}}(S,T)=0.
\end{equation}
Indeed, in the latter case decomposing the middle term of an extension into its block components splits the sequence.

For a block $\mathcal B$, write \(\Lambda_{\mathcal B}:=\{h\in\C:L(h)\in\mathcal B\}\). We call $\mathcal B$ \emph{finite} or \emph{infinite} according as $\Lambda_{\mathcal B}$ is finite or infinite. The terms \emph{minimal element}, \emph{maximal element}, and \emph{cover} refer to the order $\preceq$ on $\Lambda_{\mathcal B}$. Thus, writing $h\prec h'$ for $h\preceq h'$ and $h\ne h'$, we write $h\lessdot h'$ if $h\prec h'$ and there is no $u\in\Lambda_{\mathcal B}$ with $h\prec u\prec h'$. We call $(\Lambda_{\mathcal B},\preceq)$ the \emph{weight poset} of $\mathcal B$. By the explicit block classification in \cite[Theorem~1]{BoeNakanoWiesner}, every block is either finite, infinite with a minimal element, or infinite with a maximal element.

For an infinite block with a minimal element, by a \emph{finite truncation at $\xi$} we mean the full subcategory whose simple subquotients are labelled by those $\nu$ with $\nu\preceq\xi$, for a choice of $\xi$ for which this set is finite, as in \cite[Section~4.1]{BoeNakanoWiesner}. For an infinite block with a maximal element, by a \emph{finite quotient} we mean one of the quotient categories $\mathcal C(\Omega)$ constructed in \cite[Section~4.2]{BoeNakanoWiesner}, where the finite coideal $\Omega\subset\Lambda_{\mathcal B}$ labels its simple objects. Below, a \emph{finite category associated with $\mathcal B$} means $\mathcal B$ itself when it is finite, such a finite truncation when $\mathcal B$ is infinite with a minimal element, or such a finite quotient when $\mathcal B$ is infinite with a maximal element. For such a category $\mathcal C$, let $\Lambda_{\mathcal C}\subset\Lambda_{\mathcal B}$ be its finite set of simple labels, equipped with the order induced from $\Lambda_{\mathcal B}$: the retained labels for a truncation, $\Omega$ for $\mathcal C(\Omega)$, and $\Lambda_{\mathcal B}$ for a finite block. Write $L_{\mathcal C}(\lambda)$ for the corresponding simple, $\lambda\in\Lambda_{\mathcal C}$. In each case $\Lambda_{\mathcal C}$ is order-convex in $\Lambda_{\mathcal B}$: it is downward closed for a truncation, upward closed for a quotient, and equal to $\Lambda_{\mathcal B}$ for a finite block.

\begin{prop}\label{prop:BNW-literature-package}
Let $\mathcal B$ be a block of $\mathcal O^{\mathrm{Vir}}_{c_{p,q}}$.
\begin{enumerate}
\item For every finite $\Sigma\subset\Lambda_{\mathcal B}$, there exists a finite category $\mathcal C$ associated with $\mathcal B$ such that $\Sigma\subset\Lambda_{\mathcal C}$ and, for all $\lambda,\mu\in\Sigma$,
\[
\Ext^1_{\mathcal B}\bigl(L(\lambda),L(\mu)\bigr)
\cong
\Ext^1_{\mathcal C}\bigl(L_{\mathcal C}(\lambda),L_{\mathcal C}(\mu)\bigr).
\]
\item For any finite category $\mathcal C$ associated with $\mathcal B$ and $\lambda,\mu\in\Lambda_{\mathcal C}$,
\[
\dim_\C\Ext^1_{\mathcal C}\bigl(L_{\mathcal C}(\lambda),L_{\mathcal C}(\mu)\bigr)=
\begin{cases}
1,&\lambda\lessdot\mu\text{ or }\mu\lessdot\lambda,\\
0,&\text{otherwise},
\end{cases}
\]
where $\lessdot$ denotes the cover relation in $\Lambda_{\mathcal C}$.
\end{enumerate}
\end{prop}

\begin{proof}
For (i), if $\mathcal B$ is finite, take $\mathcal C=\mathcal B$. Suppose that $\mathcal B$ is infinite with a minimal element. By the explicit block posets in \cite[Theorem~1]{BoeNakanoWiesner}, choose a finite truncation $\mathcal C(\xi)$ whose retained labels contain $\Sigma$. If
\[
0\longrightarrow L(\mu)\longrightarrow E\longrightarrow L(\lambda)\longrightarrow0
\]
is an extension in $\mathcal B$ with $\lambda,\mu\in\Sigma$, then the composition factors of $E$ are $L(\lambda)$ and $L(\mu)$, so $E$ belongs to $\mathcal C(\xi)$. Conversely, every such extension in $\mathcal C(\xi)$ is an extension in $\mathcal B$. Since the truncation is full, equivalences and splittings of these extensions are the same in the two categories, and hence their Yoneda $\Ext^1$ groups are canonically isomorphic. Finally, if $\mathcal B$ is infinite with a maximal element, choose a finite quotient sufficiently large for all labels in $\Sigma$; the required isomorphisms are the case $j=1$ of \cite[Proposition~9]{BoeNakanoWiesner}.

Part (ii) is the $\Ext^1$-quiver description of \cite[Section~4.4]{BoeNakanoWiesner}, deduced there from \cite[Theorem~10 and Proposition~9]{BoeNakanoWiesner}.
\end{proof}

Define the full subcategory
\[\begin{aligned}
\mathscr C^{\mathrm{fl}}_{p_+,p_-}:=\bigl\{M\in\Oc_{c_{p,q}}:\;&\text{every simple Jordan--H\"older constituent of $M$}\\
&\text{has lowest conformal weight in }H_{p_+,p_-}\bigr\},
\end{aligned}\]
and, throughout the rest of the paper, abbreviate $\mathscr C:=\mathscr C^{\mathrm{fl}}_{p_+,p_-}$. Thus $H_{p_+,p_-}$ always has the explicit free-field meaning fixed above and is not an additional ambient-support condition.

In this subsection, unless another category is displayed as a subscript, $\Hom$ and $\Ext^1$ are computed in $\mathscr C$; outside this subsection the ambient category will either be written explicitly or fixed locally.

\begin{lem}\label{lem:gr-extension} Let $0\longrightarrow A\longrightarrow B\longrightarrow C\longrightarrow0$ be a short exact sequence of Virasoro modules on which the central element $\mathbf c$ acts as $c\,\id$. If $A$ and $C$ are grading-restricted generalized, then so is $B$. Moreover, for every $\lambda\in\C$, the induced sequence \(0\longrightarrow A_{[\lambda]}\longrightarrow B_{[\lambda]}\longrightarrow C_{[\lambda]}\longrightarrow0\) is exact. Consequently, \(\operatorname{Spec}(B)=\operatorname{Spec}(A)\cup\operatorname{Spec}(C).\)
\end{lem}

\begin{proof}
Let $b\in B$, and let $\bar b$ be its image in $C$. Since $C=\bigoplus_{\lambda\in\C}C_{[\lambda]}$, every vector of $C$ is a finite sum of generalized $L_0$-eigenvectors, and therefore is annihilated by a nonzero polynomial in $L_0$; equivalently, $L_0$ acts locally finitely on $C$. Thus there is a nonzero polynomial $p$ such that $p(L_0)\bar b=0$, so $p(L_0)b\in A$. The same argument shows that $L_0$ acts locally finitely on $A$, and hence there is a nonzero polynomial $q$ such that $q(L_0)p(L_0)b=0$. Therefore $L_0$ acts locally finitely on $B$. Every vector of $B$ lies in a finite-dimensional $L_0$-stable subspace, and the finite-dimensional primary decomposition implies \(B=\bigoplus_{\lambda\in\C}B_{[\lambda]}.\)

Fix $\lambda\in\C$ and $w\in C_{[\lambda]}$, and choose a lift $b\in B$ of $w$. Write the resulting finite generalized-eigenspace decomposition as $b=\sum_\mu b_\mu$ with $b_\mu\in B_{[\mu]}$. Since the quotient map commutes with $L_0$, the image of $b_\mu$ lies in $C_{[\mu]}$. As $w$ lies in $C_{[\lambda]}$ and the generalized-eigenspace decomposition of $C$ is direct, the $\lambda$-component $b_\lambda$ maps to $w$, while every $b_\mu$ with $\mu\ne\lambda$ maps to $0$. Hence $B_{[\lambda]}\to C_{[\lambda]}$ is surjective. Exactness at $A_{[\lambda]}$ and $B_{[\lambda]}$ follows directly from the original short exact sequence, because the injection and quotient map commute with $L_0$. Therefore \(0\longrightarrow A_{[\lambda]}\longrightarrow B_{[\lambda]}\longrightarrow C_{[\lambda]}\longrightarrow0\) is exact. In particular, \(\dim B_{[\lambda]}=\dim A_{[\lambda]}+\dim C_{[\lambda]}<\infty\), and the same exact sequence gives \(\operatorname{Spec}(B)=\operatorname{Spec}(A)\cup\operatorname{Spec}(C).\)

Finally, fix $\lambda\in\C$. Since $A$ and $C$ are grading restricted, there are integers $N_A,N_C$ such that $A_{[\lambda-n]}=0$ for $n\ge N_A$ and $C_{[\lambda-n]}=0$ for $n\ge N_C$. Exactness on generalized eigenspaces gives $B_{[\lambda-n]}=0$ for \(n\ge\max\{N_A,N_C\}.\) Thus $B$ is grading-restricted generalized, as required.
\end{proof}

\begin{lem}\label{lem:Nakano-ambient} Let $M\in\Oc_{c_{p,q}}$. Then $L_0$ acts locally finitely on $M$, every generalized $L_0$-eigenspace is finite dimensional, and $\operatorname{Spec}(M)\subset H_0+\Z_{\ge0}$ for a finite set $H_0\subset\C$. The standard contragredient satisfies \((M')_{[\lambda]}\cong M_{[\lambda]}^*\) for every $\lambda\in\C$ and belongs again to $\Oc_{c_{p,q}}$. Moreover, \((-)':\Oc_{c_{p,q}}^{\mathrm{op}}\to\Oc_{c_{p,q}}\) is an exact contravariant equivalence whose square is canonically isomorphic to the identity, and it restricts to such an equivalence on $\mathscr C$.
\end{lem}

\begin{proof}
By definition, every object of $\Oc_{c_{p,q}}$ is grading-restricted generalized. Hence its generalized $L_0$-eigenspaces are finite dimensional, and their algebraic direct sum implies that $L_0$ acts locally finitely. Choose a composition series of $M$ with factors $L(c_{p,q},h_1),\ldots,L(c_{p,q},h_r)$. Since \(\operatorname{Spec}L(c_{p,q},h_i)\subset h_i+\Z_{\ge0}\), Lemma~\ref{lem:gr-extension} gives \(\operatorname{Spec}(M)\subset\{h_1,\ldots,h_r\}+\Z_{\ge0}.\)

The contragredient closure of $\Oc_{c_{p,q}}$ recorded in Subsection~\ref{subsec:MY-input} gives $M'\in\Oc_{c_{p,q}}$. By \cite[Theorem~2.34 and the paragraph following Definition~2.35]{HLZ1}, the assignment $M\mapsto M'$ reverses module maps, \((M')_{[\lambda]}\cong M_{[\lambda]}^*\), and the canonical evaluation $M\to(M')'$ is an isomorphism of $V_{c_{p,q}}^{\mathrm{Vir}}$-modules. Thus $M'$ has the same spectrum, generalized-weight multiplicities, and lower truncation as $M$.

If $0\to A\to B\to C\to0$ is exact in $\Oc_{c_{p,q}}$, Lemma~\ref{lem:gr-extension} and finite-dimensional duality on each generalized weight space give \(0\to C'\to B'\to A'\to0\); the arrows are module maps by the cited HLZ result. Hence contragredience is exact. It follows that every simple $L(c_{p,q},h)$ is self-contragredient: $L(c_{p,q},h)'$ is simple and has the same lowest conformal weight $h$, so uniqueness of the simple quotient of $M(c_{p,q},h)$ gives $L(c_{p,q},h)'\cong L(c_{p,q},h)$. Therefore dualizing a composition series preserves its simple isomorphism classes, and $M\in\mathscr C$ implies $M'\in\mathscr C$. Together with $M\cong(M')'$, this proves the asserted exact contravariant equivalences.
\end{proof}

\begin{lem}\label{lem:Nakano-Serre}
The full subcategory $\mathscr C$ is a Serre subcategory of $\Oc_{c_{p,q}}$. Every object of $\mathscr C$ has a finite socle series.
\end{lem}

\begin{proof}
By definition, $\mathscr C$ consists of those finite-length objects of $\Oc_{c_{p,q}}$ whose simple Jordan--H\"older constituents have lowest conformal weights in $H_{p_+,p_-}$. Let $M\in\mathscr C$ and let $N\subseteq M$. Since $M$ has finite length, so do $N$ and $M/N$, and every Jordan--H\"older constituent of either $N$ or $M/N$ occurs among those of $M$. Hence $N,M/N\in\mathscr C$.

Now let $0\longrightarrow A\longrightarrow B\longrightarrow C\longrightarrow0$ be exact in $\Oc_{c_{p,q}}$ with $A,C\in\mathscr C$. The Jordan--H\"older multiplicities of $B$ are the sums of the corresponding multiplicities for $A$ and $C$. Thus every simple constituent of $B$ has lowest conformal weight in $H_{p_+,p_-}$, so $B\in\mathscr C$. Therefore $\mathscr C$ is a Serre subcategory.

Finally, let $M\in\mathscr C$. If $\Soc^j(M)\ne M$, then $M/\Soc^j(M)$ is a nonzero finite-length object and hence contains a simple subobject. Therefore $\Soc\bigl(M/\Soc^j(M)\bigr)\ne0$, and the definition of the socle filtration gives $\Soc^j(M)\subsetneq\Soc^{j+1}(M)$. Consequently the composition length increases by at least one at each proper step. Since $\ell(\Soc^j(M))\le\ell(M)$, the filtration reaches $M$ after at most $\ell(M)$ steps. Thus every object of $\mathscr C$ has a finite socle series.
\end{proof}

\begin{lem}\label{lem:Nakano-exact-compatibility} For $X,Y\in\mathscr C$, the inclusion $\mathscr C\hookrightarrow\Oc_{c_{p,q}}$ induces an isomorphism of Yoneda extension spaces \(\Ext^1_{\mathscr C}(X,Y) \xrightarrow{\ \sim\ } \Ext^1_{\Oc_{c_{p,q}}}(X,Y).\) \end{lem}

\begin{proof}
By Lemma~\ref{lem:Nakano-Serre}, $\mathscr C$ is a full extension-closed subcategory of $\Oc_{c_{p,q}}$, hence a fully exact subcategory with the induced exact structure. The assertion is the standard identification of Yoneda $\Ext^1$ for a fully exact subcategory; see \cite[Section~2.1]{BodzentaBondal}.
\end{proof}

\begin{lem}\label{lem:prelim-Hom-Yoneda} Let $0\longrightarrow A\longrightarrow B\longrightarrow C\longrightarrow0$ be a short exact sequence in a $\C$-linear abelian category. For every object $X$ there are natural exact sequences \[\begin{aligned} 0\longrightarrow&\Hom(C,X)\longrightarrow\Hom(B,X) \longrightarrow\Hom(A,X)\\ \longrightarrow&\Ext^1(C,X)\longrightarrow\Ext^1(B,X) \longrightarrow\Ext^1(A,X)
\end{aligned}\]
and \[\begin{aligned} 0\longrightarrow&\Hom(X,A)\longrightarrow\Hom(X,B) \longrightarrow\Hom(X,C)\\ \longrightarrow&\Ext^1(X,A)\longrightarrow\Ext^1(X,B) \longrightarrow\Ext^1(X,C).
\end{aligned}\]
Here $\Ext^1$ denotes Yoneda $\Ext^1$.
\end{lem}

\begin{proof}
This is the standard low-degree exact sequence for Yoneda extensions; see, for example, \cite[Section~3.4]{Weibel}.
\end{proof}

\begin{lem}\label{lem:one-dimensional-ext}
Let $\mathscr E$ be a full extension-closed subcategory of a $\C$-linear abelian category, endowed with the induced exact structure, and suppose $\dim_\C\Ext^1_{\mathscr E}(X,Y)=1$. Then all nonzero extension classes of $X$ by $Y$ have isomorphic middle terms.
\end{lem}

\begin{proof}
If $0\to Y\xrightarrow{i}E\xrightarrow{p}X\to0$ represents $\xi\in\Ext^1_{\mathscr E}(X,Y)$ and $a\in\C^\times$, then $0\to Y\xrightarrow{a^{-1}i}E\xrightarrow{p}X\to0$ represents $a\xi$; see, for example, \cite[Section~3.4]{Weibel}. Since any two nonzero elements of the one-dimensional space $\Ext^1_{\mathscr E}(X,Y)$ differ by a nonzero scalar, all nonzero extension classes have isomorphic middle terms.
\end{proof}

We next record the extension-theoretic input from \cite{Nakano}. Throughout the following proposition, all $\Ext^1$-groups and short exact sequences are taken in Nakano's category \(\mathscr N=\mathcal L_{c_{p_+,p_-}}\text{-mod}\) defined in \cite[Definition~5.1]{Nakano}. For the remainder of this discussion, fix \(\tau=(\alpha_1,\alpha_2,\alpha_3)\in\mathcal T_{p_+,p_-}\) with \(h_{\alpha_1}<h_{\alpha_2}<h_{\alpha_3},\) and write \(\mathsf{L}_i(\tau):=L(h_{\alpha_i})\) for $i=1,2,3$. This notation is distinct from the one-row simples $S_j$ fixed in Subsection~\ref{subsec:MY-input}. Following \cite[Definition~5.9]{Nakano}, if $h<h'$ and $\Ext^1_{\mathscr N}(L(h),L(h'))\cong\C$, we write $E_{\mathrm N}(h,h')$ for Nakano's chosen middle term of a nonzero class (denoted $L(h,h')$ in \cite{Nakano}); thus
\[
0\longrightarrow L(h')\longrightarrow E_{\mathrm N}(h,h')\longrightarrow L(h)\longrightarrow0
\]
is a nonsplit exact sequence in $\mathscr N$.

\begin{prop}\label{prop:Nakano-literature-package}
Let $\tau=(\alpha_1,\alpha_2,\alpha_3)\in\mathcal T_{p_+,p_-}$ satisfy $h_{\alpha_1}<h_{\alpha_2}<h_{\alpha_3}$.
\begin{enumerate}
\item \[
\Ext^1_{\mathscr N}\bigl(\mathsf L_1(\tau),\mathsf L_2(\tau)\bigr)\cong\C,
\qquad
\Ext^1_{\mathscr N}\bigl(\mathsf L_2(\tau),\mathsf L_3(\tau)\bigr)\cong\C,
\]
and the chosen nonzero classes give exact sequences
\[
0\to\mathsf L_2(\tau)\to E_{\mathrm N}(h_{\alpha_1},h_{\alpha_2})\to\mathsf L_1(\tau)\to0,
\quad
0\to\mathsf L_3(\tau)\to E_{\mathrm N}(h_{\alpha_2},h_{\alpha_3})\to\mathsf L_2(\tau)\to0.
\]
\item One has
\[
\Ext^1_{\mathscr N}\bigl(E_{\mathrm N}(h_{\alpha_2},h_{\alpha_3}),\mathsf L_1(\tau)\bigr)\cong\C.
\]
Let $K_{\mathrm N}(\tau)$ denote the module called $K(\tau)$ in \cite[Definition~5.10]{Nakano}. It is the middle term of a nonzero class and hence fits into a nonsplit exact sequence
\[
0\longrightarrow\mathsf L_1(\tau)\longrightarrow K_{\mathrm N}(\tau)
\longrightarrow E_{\mathrm N}(h_{\alpha_2},h_{\alpha_3})\longrightarrow0.
\]
\item Nakano constructs a logarithmic module $E$ fitting into
\[
0\longrightarrow E_{\mathrm N}(h_{\alpha_1},h_{\alpha_2})\longrightarrow E
\longrightarrow E_{\mathrm N}(h_{\alpha_2},h_{\alpha_3})\longrightarrow0.
\]
For every logarithmic module representing a class in
\[
\Ext^1_{\mathscr N}\bigl(E_{\mathrm N}(h_{\alpha_2},h_{\alpha_3}),
E_{\mathrm N}(h_{\alpha_1},h_{\alpha_2})\bigr),
\]
\cite[Theorem~5.11]{Nakano} states that the quotient by the copy of $\mathsf L_2(\tau)$ inside $E_{\mathrm N}(h_{\alpha_1},h_{\alpha_2})$ is indecomposable. Consequently
\[
0\longrightarrow\mathsf L_1(\tau)\longrightarrow E/\mathsf L_2(\tau)
\longrightarrow E_{\mathrm N}(h_{\alpha_2},h_{\alpha_3})\longrightarrow0
\]
is nonsplit.
\item Nakano exhibits a logarithmic module representing a nonzero class in
\[
\Ext^1_{\mathscr N}\bigl(K_{\mathrm N}(\tau),\mathsf L_2(\tau)\bigr).
\]
\end{enumerate}
\end{prop}

\begin{proof}
For (i), the defining Felder adjacency of $\tau$ together with \cite[Proposition~5.5]{Nakano} gives the two one-dimensional simple $\Ext^1_{\mathscr N}$ spaces, and \cite[Definition~5.9]{Nakano} fixes the displayed middle terms. Part (ii) is exactly the nonzero extension used in \cite[Definition~5.10]{Nakano}, where the corresponding $\Ext^1_{\mathscr N}$ space is also recorded as one-dimensional. For (iii), the hypothesis $h_{\alpha_1}<h_{\alpha_2}$ implies $\tau\notin\mathcal T^0_{p_+,p_-}$ by the definition of $\mathcal T^0_{p_+,p_-}$. Nakano constructs the stated logarithmic extension immediately before \cite[Theorem~5.11]{Nakano}, and that theorem gives the asserted indecomposability of the quotient. A split quotient would be the direct sum of its two nonzero endpoints, so it is nonsplit. Part (iv) is the nonzero class exhibited in the paragraph following \cite[Theorem~5.11]{Nakano} and preceding \cite[Proposition~5.12]{Nakano}.
\end{proof}

\begin{lem}\label{lem:local-category-O}
Let $S,T\in\mathscr C$ be nonisomorphic simple objects. The short exact sequences \(0\longrightarrow T\longrightarrow E\longrightarrow S\longrightarrow0\) whose middle term lies in $\mathscr C$ are exactly those whose middle term lies in $\mathcal O^{\mathrm{Vir}}_{c_{p,q}}$. Every such middle term is ordinary. Consequently, the identity on such sequences induces a canonical isomorphism
\[
\Ext^1_{\mathscr C}(S,T)\xrightarrow{\ \sim\ }\Ext^1_{\mathcal O^{\mathrm{Vir}}_{c_{p,q}}}(S,T).
\]
\end{lem}

\begin{proof}
First suppose that $E\in\mathscr C$. The conformal nilpotent $D_E$ defined at the beginning of Section~\ref{sec:prelim} restricts to zero on $T$ and induces zero on $S$, because the simple modules $S$ and $T$ are ordinary. Hence $\operatorname{Im}D_E\subseteq T$, and $D_E$ factors through a morphism $S\to T$. Since $S$ and $T$ are nonisomorphic simples, $\Hom(S,T)=0$; therefore $D_E=0$, so $L_0$ acts semisimply on $E$.

Write $T=L(c_{p,q},h_T)$ and $S=L(c_{p,q},h_S)$. By Lemma~\ref{lem:gr-extension},
\[
\operatorname{Spec}(E)=\operatorname{Spec}(S)\cup\operatorname{Spec}(T)
\subseteq(h_S+\Z_{\ge0})\cup(h_T+\Z_{\ge0}),
\]
and every $L_0$-weight space of $E$ is finite dimensional. Choose highest-weight vectors $t\in T$ and $s\in S$ and a lift $\widetilde s\in E$ of $s$. Then $E=U(\Vir)t+U(\Vir)\widetilde s$, so $E$ is finitely generated. If $v\in E_{[\lambda]}$, then \(U(\Vir_{>0})v\subseteq\bigoplus_{n\ge0}E_{[\lambda-n]}\). Only finitely many summands on the right are nonzero, and each is finite dimensional; hence $\Vir_{>0}$ acts locally finitely. Thus $E$ satisfies the defining conditions of the Virasoro category $\mathcal O$ in \cite[Section~2.2]{BoeNakanoWiesner}, and therefore $E\in\mathcal O^{\mathrm{Vir}}_{c_{p,q}}$.

Conversely, suppose that the displayed sequence is exact in $\mathcal O^{\mathrm{Vir}}_{c_{p,q}}$. Lemma~\ref{lem:gr-extension} shows that $E$ is grading-restricted generalized, with spectrum the union of the spectra of $S$ and $T$. Since objects of $\mathcal O^{\mathrm{Vir}}_{c_{p,q}}$ are weight modules, $L_0$ is semisimple on $E$, so $E$ is ordinary. In particular $E$ is a restricted Virasoro module of central charge $c_{p,q}$ and hence a weak module for the universal Virasoro vertex operator algebra $V^{\mathrm{Vir}}_{c_{p,q}}$. Since $V^{\mathrm{Vir}}_{c_{p,q}}$ is generated by its conformal vector, its submodules are exactly the Virasoro submodules; thus $E$ has precisely the two composition factors $S$ and $T$. The finite-length characterization of $\Oc_{c_{p,q}}$ recalled in Subsection~\ref{subsec:MY-input} therefore gives $E\in\Oc_{c_{p,q}}$. Since $S,T\in\mathscr C$, both composition factors have lowest conformal weights in $H_{p_+,p_-}$, and the definition of $\mathscr C$ now gives $E\in\mathscr C$.

Thus the two categories contain exactly the same short exact sequences with endpoints $T$ and $S$. Both are full subcategories of the category of Virasoro modules, so equivalence of such extensions is the same in the two categories. Their Yoneda extension spaces are therefore canonically identified.
\end{proof}

\begin{lem}\label{lem:Felder-cover-chain}
For the fixed $\tau=(\alpha_1,\alpha_2,\alpha_3)\in\mathcal T_{p_+,p_-}$ with $h_{\alpha_1}<h_{\alpha_2}<h_{\alpha_3}$, the three simples $\mathsf{L}_i(\tau)=L(h_{\alpha_i})$ lie in a common block $\mathcal B$ of $\mathcal O^{\mathrm{Vir}}_{c_{p,q}}$. Writing $\lambda_i:=h_{\alpha_i}\in\Lambda_{\mathcal B}$, one has \(\lambda_3\lessdot\lambda_2\lessdot\lambda_1\). In particular, $\lambda_1$ and $\lambda_3$ are not related by a cover relation.
\end{lem}

\begin{proof}
Proposition~\ref{prop:Nakano-literature-package}(i) gives two nonsplit exact sequences in $\mathscr N$. Their middle terms are $E_{\mathrm N}(h_{\alpha_1},h_{\alpha_2})$ and $E_{\mathrm N}(h_{\alpha_2},h_{\alpha_3})$. Each is grading-restricted generalized of composition length two, and its simple factors have lowest conformal weights in $H_{p_+,p_-}$. Hence the finite-length characterization of $\Oc_{c_{p,q}}$ from Subsection~\ref{subsec:MY-input} places both middle terms in $\mathscr C$. Since $\mathscr N$ is full in Nakano's ambient Virasoro-module category, these sequences remain nonsplit in $\mathscr C$.

The inequalities $h_{\alpha_1}<h_{\alpha_2}<h_{\alpha_3}$ make the three simple endpoints pairwise nonisomorphic. Lemma~\ref{lem:local-category-O} therefore identifies the two sequences with nonzero classes in $\mathcal O^{\mathrm{Vir}}_{c_{p,q}}$. By \eqref{eq:Vir-cross-block-Ext}, all three simples lie in a common block $\mathcal B$. Set $\lambda_i:=h_{\alpha_i}\in\Lambda_{\mathcal B}$.

Apply Proposition~\ref{prop:BNW-literature-package}(i) to $\mathcal B$ with $\Sigma=\{\lambda_1,\lambda_2,\lambda_3\}$. The two nonzero classes remain nonzero in the resulting finite category $\mathcal C$. Proposition~\ref{prop:BNW-literature-package}(ii) then shows that $\lambda_1,\lambda_2$ are related by a cover relation in $\Lambda_{\mathcal C}$, and likewise $\lambda_2,\lambda_3$. Since $\Lambda_{\mathcal C}$ is order-convex in $\Lambda_{\mathcal B}$, both are also cover relations in the original block.

It remains to determine their directions. By definition, $x\preceq y$ means that $L(x)$ is a subquotient of $M(y)$. From the PBW construction of the Virasoro Verma module recalled above, every $L_0$-weight of $M(y)$ lies in $y+\Z_{\ge0}$; hence $x\preceq y$ implies $x-y\in\Z_{\ge0}$. Since $\lambda_1<\lambda_2<\lambda_3$, the cover relation between $\lambda_1$ and $\lambda_2$ must therefore be $\lambda_2\lessdot\lambda_1$, and the one between $\lambda_2$ and $\lambda_3$ must be $\lambda_3\lessdot\lambda_2$. Thus $\lambda_3\lessdot\lambda_2\lessdot\lambda_1$. In particular, $\lambda_2$ lies strictly between $\lambda_3$ and $\lambda_1$ in the block order, so the endpoints are not related by a cover relation.
\end{proof}

\begin{lem}\label{lem:local-simple-Ext}
For the fixed $\tau=(\alpha_1,\alpha_2,\alpha_3)\in\mathcal T_{p_+,p_-}$ with $h_{\alpha_1}<h_{\alpha_2}<h_{\alpha_3}$,
\[
\Ext^1_{\mathscr C}\bigl(\mathsf L_i(\tau),\mathsf L_{i+1}(\tau)\bigr)
\cong
\Ext^1_{\mathscr C}\bigl(\mathsf L_{i+1}(\tau),\mathsf L_i(\tau)\bigr)
\cong\C
\qquad (i=1,2),
\]
and
\[
\Ext^1_{\mathscr C}\bigl(\mathsf L_1(\tau),\mathsf L_3(\tau)\bigr)
=
\Ext^1_{\mathscr C}\bigl(\mathsf L_3(\tau),\mathsf L_1(\tau)\bigr)
=0.
\]
\end{lem}

\begin{proof}
Let $\mathcal B$ and $\lambda_i=h_{\alpha_i}$ be as in Lemma~\ref{lem:Felder-cover-chain}.  The inequalities among the $h_{\alpha_i}$ make the three simples pairwise nonisomorphic, so Lemma~\ref{lem:local-category-O} and \eqref{eq:Vir-same-block-Ext} give, for $i\ne j$,
\[
\Ext^1_{\mathscr C}\bigl(\mathsf L_i(\tau),\mathsf L_j(\tau)\bigr)
\cong
\Ext^1_{\mathcal B}\bigl(L(\lambda_i),L(\lambda_j)\bigr).
\]
Apply Proposition~\ref{prop:BNW-literature-package}(i) to $\mathcal B$ with
$\Sigma=\{\lambda_1,\lambda_2,\lambda_3\}$, and let $\mathcal C_0$ be the resulting finite category.  Since $\Lambda_{\mathcal C_0}$ carries the order induced from $\Lambda_{\mathcal B}$ and contains all three labels, Lemma~\ref{lem:Felder-cover-chain} implies that $\lambda_1,\lambda_2$ and $\lambda_2,\lambda_3$ are cover-related in $\Lambda_{\mathcal C_0}$, while $\lambda_1$ and $\lambda_3$ are not: indeed, $\lambda_2$ lies strictly between them.  Proposition~\ref{prop:BNW-literature-package}(ii) therefore gives dimension one for the first two $\Ext^1$-groups in either orientation and dimension zero for the endpoint pair.  Proposition~\ref{prop:BNW-literature-package}(i), followed by the displayed isomorphism above, transfers these dimensions to $\mathscr C$, proving the assertions.
\end{proof}

Set
\(
A(\tau):=E_{\mathrm N}(h_{\alpha_1},h_{\alpha_2}),
B(\tau):=E_{\mathrm N}(h_{\alpha_2},h_{\alpha_3}).
\)
By the proof of Lemma~\ref{lem:Felder-cover-chain}, these modules belong to $\mathscr C$ and occur in the nonsplit exact sequences
\begin{align}
0&\longrightarrow \mathsf{L}_2(\tau)\longrightarrow A(\tau) \longrightarrow \mathsf{L}_1(\tau)\longrightarrow0,
\label{eq:Nakano-tower-A}\\
0&\longrightarrow \mathsf{L}_3(\tau)\longrightarrow B(\tau) \longrightarrow \mathsf{L}_2(\tau)\longrightarrow0. \label{eq:Nakano-tower-B}
\end{align}
We regard the two displayed inclusions as fixed throughout. By Lemmas~\ref{lem:local-simple-Ext} and~\ref{lem:one-dimensional-ext}, $A(\tau)$ and $B(\tau)$ are, up to isomorphism, the unique middle terms of nonzero classes in the corresponding one-dimensional $\Ext^1_{\mathscr C}$ spaces.

Applying Lemma~\ref{lem:prelim-Hom-Yoneda} to \eqref{eq:Nakano-tower-B} with target $\mathsf{L}_1(\tau)$, and using Lemma~\ref{lem:local-simple-Ext}, gives
\begin{equation}\label{eq:Nakano-first-Ext} \Ext^1_{\mathscr C}\bigl(B(\tau),\mathsf{L}_1(\tau)\bigr)\cong\C.\end{equation}
Indeed, the three preceding Hom spaces vanish, and Lemma~\ref{lem:local-simple-Ext} gives \(\Ext^1_{\mathscr C}(\mathsf{L}_2(\tau),\mathsf{L}_1(\tau))\cong\C, \Ext^1_{\mathscr C}(\mathsf{L}_3(\tau),\mathsf{L}_1(\tau))=0.\) We denote by $K(\tau)$ the middle term of a nonzero class; this parenthesized notation is distinct from the one-row Kac modules $K_j$. Thus
\begin{equation}\label{eq:Nakano-tower-K} 0\longrightarrow \mathsf{L}_1(\tau)\longrightarrow K(\tau) \longrightarrow B(\tau)\longrightarrow0.\end{equation}
Nakano's module $K_{\mathrm N}(\tau)$ from Proposition~\ref{prop:Nakano-literature-package}(ii) has composition length three, and all three simple factors have lowest conformal weights in $H_{p_+,p_-}$. Hence the same finite-length characterization places it in $\mathscr C$. Its defining nonsplit sequence is then a nonzero class in \eqref{eq:Nakano-first-Ext}; hence Lemma~\ref{lem:one-dimensional-ext} gives $K_{\mathrm N}(\tau)\cong K(\tau)$.

We use the terminology of \cite[Section~3]{KytolaRidout}: a \emph{staggered Virasoro module} is an indecomposable Virasoro module $E$ fitting into a short exact sequence $0\to H^{\mathrm L}\to E\to H^{\mathrm R}\to0$, where $H^{\mathrm L}$ and $H^{\mathrm R}$ are highest-weight modules and $L_0$ is nonsemisimple on $E$ with nilpotent part squaring to zero. In the notation fixed above, this is $D_E\ne0$ and $D_E^2=0$.

\begin{lem}\label{lem:local-highest-weight-realization}
The modules $A(\tau)$ and $B(\tau)$ are ordinary highest-weight Virasoro modules.
\end{lem}

\begin{proof}
By \eqref{eq:Nakano-tower-A} and \eqref{eq:Nakano-tower-B}, each of $A(\tau)$ and $B(\tau)$ is the middle term of a nonsplit extension \(0\longrightarrow T\longrightarrow M\longrightarrow S\longrightarrow0\) of two nonisomorphic simple highest-weight modules, with lowest conformal weights $h_S<h_T$. Lemma~\ref{lem:local-category-O} shows that $M$ is ordinary. Let $v$ be a nonzero highest-weight vector of $S$. By Lemma~\ref{lem:gr-extension}, the map $M_{[h_S]}\to S_{[h_S]}$ is surjective, so choose $\widetilde v\in M_{[h_S]}$ mapping to $v$. The same lemma gives
\[
\operatorname{Spec}(M)=\operatorname{Spec}(S)\cup\operatorname{Spec}(T)
\subseteq (h_S+\mathbb Z_{\ge0})\cup(h_T+\mathbb Z_{\ge0}),
\]
so $M_{[h_S-n]}=0$ for every $n>0$. Hence $L_n\widetilde v=0$ for $n>0$. Thus $U:=U(\Vir)\widetilde v$ is a highest-weight submodule and maps surjectively onto $S$.

Since $T$ is simple, $U\cap T$ is either $0$ or $T$. If $U\cap T=0$, then $U\to S$ is an isomorphism and gives a section of $M\twoheadrightarrow S$, contradicting nonsplitting. Hence $T\subseteq U$. Let $\pi:M\twoheadrightarrow S$ be the quotient map. Since $\pi(U)=S$, for every $m\in M$ there is $u\in U$ with $\pi(u)=\pi(m)$. Hence $m-u\in\ker\pi=T\subseteq U$, so $m\in U$. Thus $M=U$, and $M$ is a highest-weight module. Applying this to the two displayed nonsplit sequences proves the lemma.
\end{proof}

\begin{lem}\label{lem:local-staggered-realization}
For the fixed $\tau=(\alpha_1,\alpha_2,\alpha_3)\in\mathcal T_{p_+,p_-}$ with
$h_{\alpha_1}<h_{\alpha_2}<h_{\alpha_3}$, let
\[
0\longrightarrow A(\tau)\xrightarrow{\iota}E\xrightarrow{\pi}B(\tau)\longrightarrow0
\]
be an exact sequence in $\mathscr C$ whose middle term $E$ is logarithmic. Then $E$ is a staggered Virasoro module with left module $A(\tau)$ and right module $B(\tau)$. If $v_{\alpha_1}$ is a nonzero highest-weight vector of $A(\tau)$, set \(\eta_{12}:=\eta(h_{\alpha_1},h_{\alpha_2})\in U(\Vir_-)\), where $\eta(h_{\alpha_1},h_{\alpha_2})$ denotes the Shapovalov element chosen immediately before \cite[Theorem~5.11]{Nakano}. Then $\eta_{12}v_{\alpha_1}$ is a nonzero singular vector generating the fixed submodule $\mathsf L_2(\tau)\subset A(\tau)$ in \eqref{eq:Nakano-tower-A}; equivalently,
\(\mathsf L_2(\tau)=U(\Vir)\eta_{12}v_{\alpha_1}\).
\end{lem}

\begin{proof}
By Lemma~\ref{lem:local-highest-weight-realization}, $A(\tau)$ and $B(\tau)$ are ordinary highest-weight Virasoro modules. Because $\iota$ and $\pi$ intertwine $L_0$, the canonical conformal nilpotent satisfies \(D_E\iota=\iota D_{A(\tau)}=0, \; \pi D_E=D_{B(\tau)}\pi=0\). Hence $D_E(E)\subseteq\ker\pi=A(\tau)$ and $D_E$ vanishes on $A(\tau)$, so $D_E^2=0$. Since $E$ is logarithmic, $D_E\ne0$.

It remains to verify indecomposability. The highest-weight generator of $A(\tau)$ has weight $h_{\alpha_1}$, while $B(\tau)$ has no conformal weights below $h_{\alpha_2}>h_{\alpha_1}$; hence $\Hom(A(\tau),B(\tau))=0$. Moreover, each endpoint is cyclic with one-dimensional lowest conformal-weight space, so $\End(A(\tau))=\End(B(\tau))=\C$. Let $e\in\End(E)$ be idempotent. Since $\pi e\iota=0$, it induces idempotents $e_A,e_B\in\{0,\id\}$ on the endpoints. If $e_A=e_B=0$, then $e(E)\subseteq A(\tau)$ and $e|_{A(\tau)}=0$, so $e=e^2=0$; if $e_A=e_B=\id$, apply the same argument to $\id-e$. If $e_A=0$ and $e_B=\id$, then $\pi e=\pi$ and $e(E)\cap A(\tau)=0$, so $\pi:e(E)\xrightarrow{\sim}B(\tau)$ supplies a section; the other mixed case follows by replacing $e$ with $\id-e$. Thus any nontrivial idempotent would split the exact sequence. But a splitting would give $E\cong A(\tau)\oplus B(\tau)$ with semisimple $L_0$, contrary to logarithmicity. Hence $\End(E)$ has no nontrivial idempotents. Since an idempotent endomorphism splits a Virasoro module as the direct sum of its image and kernel, $E$ is indecomposable. Together with $D_E\ne0$ and $D_E^2=0$, this verifies the defining conditions of a staggered Virasoro module.

Finally, $A(\tau)=E_{\mathrm N}(h_{\alpha_1},h_{\alpha_2})$ by definition. Immediately before \cite[Theorem~5.11]{Nakano}, Nakano chooses $\eta(h_{\alpha_1},h_{\alpha_2})\in U(\Vir_-)$ so that its action on a highest-weight generator of this module is a nonzero singular vector generating its simple submodule of conformal weight $h_{\alpha_2}$. Under the fixed sequence \eqref{eq:Nakano-tower-A}, that submodule is precisely $\mathsf L_2(\tau)$, which proves the displayed identity for $\eta_{12}$.
\end{proof}

For a logarithmic extension
\[
0\longrightarrow A(\tau)\longrightarrow E\longrightarrow B(\tau)\longrightarrow0,
\]
fix a nonzero highest-weight vector $v_1$ of $A(\tau)$. By Lemma~\ref{lem:gr-extension}, choose $v_2\in E_{[h_{\alpha_2}]}$ mapping to a nonzero highest-weight vector of $B(\tau)$. With $\eta_{12}$ as in Lemma~\ref{lem:local-staggered-realization}, there are unique scalars $c_E,d_E\in\C$ such that \(\sigma(\eta_{12})v_2=c_Ev_1, \; (L_0-h_{\alpha_2})v_2=d_E\eta_{12}v_1\). Indeed, $\sigma(\eta_{12})$ is a sum of positive-mode monomials, so the first vector maps to zero in $B(\tau)$; it has conformal weight $h_{\alpha_1}$ and therefore lies in $A(\tau)_{[h_{\alpha_1}]}=\C v_1$. The second vector is $D_Ev_2$. Since $D_E$ is a Virasoro-module endomorphism, vanishes on $A(\tau)$, and $v_2$ maps to a highest-weight vector of $B(\tau)$, the vector $D_Ev_2$ is singular of conformal weight $h_{\alpha_2}$. Its image in the simple highest-weight quotient $\mathsf L_1(\tau)$ of $A(\tau)$ is a singular vector of weight $h_{\alpha_2}>h_{\alpha_1}$, hence zero. Thus $D_Ev_2\in\mathsf L_2(\tau)_{[h_{\alpha_2}]}=\C\eta_{12}v_1$. Moreover $d_E\ne0$: otherwise $D_E$ would induce a Virasoro homomorphism $B(\tau)\to A(\tau)$ vanishing on the highest-weight generator of the cyclic module $B(\tau)$, hence vanishing identically; since $D_E$ also vanishes on $A(\tau)$, this would give $D_E=0$, contrary to logarithmicity. Following the definition immediately preceding \cite[Theorem~5.11]{Nakano}, set \(\beta_E:=\frac{c_E}{d_E}\). Nakano states there that this ratio is independent of the choices of $v_1$ and $v_2$ and calls it the logarithmic coupling of $E$. We always use $\beta_E$ with this normalization.

\begin{lem}\label{lem:split-forces-zero-coupling}
For the fixed $\tau$, let $0\longrightarrow A(\tau)\longrightarrow E\longrightarrow B(\tau)\longrightarrow0$ be logarithmic in $\mathscr C$. Quotienting by the fixed submodule $\mathsf L_2(\tau)\subset A(\tau)$ from \eqref{eq:Nakano-tower-A} gives an exact sequence
\begin{equation}\label{eq:distinguished-quotient-general}
0\longrightarrow \mathsf L_1(\tau)\longrightarrow E/\mathsf L_2(\tau)\longrightarrow B(\tau)\longrightarrow0.
\end{equation}
If \eqref{eq:distinguished-quotient-general} splits, then $\beta_E=0$.
\end{lem}

\begin{proof}
Put $h_i=h_{\alpha_i}$, let $q:E\twoheadrightarrow E/\mathsf L_2(\tau)$ be the quotient map, and let $s:B(\tau)\to E/\mathsf L_2(\tau)$ be a section. For a nonzero highest-weight vector $b_2$ of $B(\tau)$, set $\overline v_2=s(b_2)$ and, by Lemma~\ref{lem:gr-extension}, choose $v_2\in E_{[h_2]}$ with $q(v_2)=\overline v_2$. Then $v_2$ is an allowed choice in the definition of $\beta_E$. Since $\eta_{12}$ has positive relative degree $h_2-h_1$, $\sigma(\eta_{12})$ is a sum of positive-mode monomials, and hence \(q\bigl(\sigma(\eta_{12})v_2\bigr)=\sigma(\eta_{12})\overline v_2=0\). Thus $\sigma(\eta_{12})v_2\in\mathsf L_2(\tau)$. It has conformal weight $h_1<h_2$, whereas $\mathsf L_2(\tau)$ has lowest conformal weight $h_2$, so it is zero. Hence $c_E=0$, and since $d_E\ne0$, one has $\beta_E=c_E/d_E=0$.
\end{proof}

\begin{lem}\label{lem:local-auxiliary-extension}
For the fixed $\tau$, assume $\Ext^1_{\mathscr C}(\mathsf L_2(\tau),\mathsf L_2(\tau))=0$. Let
\[
0\longrightarrow \mathsf L_2(\tau)\longrightarrow E_2\xrightarrow{\pi} B(\tau)\longrightarrow0
\]
represent a nonzero class in $\Ext^1_{\mathscr C}(B(\tau),\mathsf L_2(\tau))$. Then $E_2$ is logarithmic and fits into an exact sequence
\begin{equation}\label{eq:auxiliary-E2}
0\longrightarrow B(\tau)'\longrightarrow E_2\longrightarrow \mathsf L_2(\tau)\longrightarrow0.
\end{equation}
\end{lem}

\begin{proof}
Let $P:=\pi^{-1}(\mathsf L_3(\tau))$, where $\mathsf L_3(\tau)\subset B(\tau)$ is the fixed submodule in \eqref{eq:Nakano-tower-B}. Then
\[
0\longrightarrow \mathsf L_2(\tau)\longrightarrow P\longrightarrow \mathsf L_3(\tau)\longrightarrow0
\]
is the restriction of the given extension to $\mathsf L_3(\tau)$. Applying Lemma~\ref{lem:prelim-Hom-Yoneda} to \eqref{eq:Nakano-tower-B} with target $\mathsf L_2(\tau)$ gives the exact segment
\[
\Ext^1_{\mathscr C}(\mathsf L_2(\tau),\mathsf L_2(\tau))
\longrightarrow
\Ext^1_{\mathscr C}(B(\tau),\mathsf L_2(\tau))
\longrightarrow
\Ext^1_{\mathscr C}(\mathsf L_3(\tau),\mathsf L_2(\tau)).
\]
The first group vanishes by hypothesis, so the second arrow is injective. Hence the extension with middle term $P$ is nonzero.

By Lemma~\ref{lem:Nakano-ambient}, dualizing the nonsplit sequence \eqref{eq:Nakano-tower-B} and using self-contragredience of the simple endpoints gives a nonsplit exact sequence
\[
0\longrightarrow \mathsf L_2(\tau)\longrightarrow B(\tau)'\longrightarrow \mathsf L_3(\tau)\longrightarrow0.
\]
Lemma~\ref{lem:local-simple-Ext} gives $\Ext^1_{\mathscr C}(\mathsf L_3(\tau),\mathsf L_2(\tau))\cong\C$, so Lemma~\ref{lem:one-dimensional-ext} yields $P\cong B(\tau)'$. Since $P\subset E_2$, this identifies a submodule $B(\tau)'\subset E_2$. Moreover, the quotient map $\pi$ induces \(E_2/P\cong B(\tau)/\mathsf L_3(\tau)\cong\mathsf L_2(\tau)\), which proves \eqref{eq:auxiliary-E2}.

It remains to prove that $E_2$ is logarithmic. Suppose instead that $L_0$ acts semisimply on $E_2$. Let $v$ be a nonzero highest-weight vector of $B(\tau)$, of conformal weight $h_{\alpha_2}$. By Lemma~\ref{lem:gr-extension}, choose $\widetilde v\in(E_2)_{[h_{\alpha_2}]}$ with $\pi(\widetilde v)=v$. The same lemma gives \((E_2)_{[h_{\alpha_2}-n]}=0\;(n>0)\), because both endpoints have no conformal weights below $h_{\alpha_2}$. Hence every positive Virasoro mode annihilates $\widetilde v$, and $U:=U(\Vir)\widetilde v$ is a highest-weight submodule with $\pi(U)=B(\tau)$.

Since $\ker\pi=\mathsf L_2(\tau)$ is simple, either $U\cap\mathsf L_2(\tau)=0$ or $U\cap\mathsf L_2(\tau)=\mathsf L_2(\tau)$. In the first case, $\pi|_U:U\to B(\tau)$ is an isomorphism and its inverse gives a section of $\pi$, contradicting the nonzero extension class. In the second case, a nonzero vector $w$ in the lowest conformal-weight space $\mathsf L_2(\tau)_{[h_{\alpha_2}]}$ belongs to $U$. By the PBW grading of the highest-weight module generated by $\widetilde v$, one has $U_{[h_{\alpha_2}]}=\C\widetilde v$. But $\pi(w)=0$ whereas $\pi(\widetilde v)=v\ne0$, so $w$ and $\widetilde v$ are linearly independent. This contradicts the one-dimensionality $U_{[h_{\alpha_2}]}=\C\widetilde v$. Thus $L_0$ is not semisimple on $E_2$, and hence $E_2$ is logarithmic.
\end{proof}

\begin{lem}\label{lem:Nakano-existence-input}
Let $\tau=(\alpha_1,\alpha_2,\alpha_3)\in\mathcal T_{p_+,p_-}$ satisfy $h_{\alpha_1}<h_{\alpha_2}<h_{\alpha_3}$. There exists a logarithmic short exact sequence in $\mathscr C$
\begin{equation}\label{eq:Nakano-existence-E}
0\longrightarrow A(\tau)\longrightarrow E_\tau\longrightarrow B(\tau)\longrightarrow0.
\end{equation}
\end{lem}

\begin{proof}
Since $h_{\alpha_1}<h_{\alpha_2}$, the definition of $\mathcal T^0_{p_+,p_-}$ gives $\tau\notin\mathcal T^0_{p_+,p_-}$. The construction immediately preceding \cite[Theorem~5.11]{Nakano}, recorded in Proposition~\ref{prop:Nakano-literature-package}(iii), gives a logarithmic extension in $\mathscr N$ with endpoints $E_{\mathrm N}(h_{\alpha_1},h_{\alpha_2})$ and $E_{\mathrm N}(h_{\alpha_2},h_{\alpha_3})$. By the definitions of $A(\tau)$ and $B(\tau)$, this is exactly \eqref{eq:Nakano-existence-E}.

It remains to verify that $E_\tau$ belongs to $\mathscr C$. Since $A(\tau),B(\tau)\in\mathscr C$, both endpoints are grading-restricted generalized and have finite length. Lemma~\ref{lem:gr-extension} therefore shows that $E_\tau$ is grading-restricted generalized, while the short exact sequence shows that $E_\tau$ has finite length and that its Jordan--H\"older constituents are precisely those of $A(\tau)$ and $B(\tau)$. The finite-length characterization of $\Oc_{c_{p,q}}$ recalled in Subsection~\ref{subsec:MY-input} gives $E_\tau\in\Oc_{c_{p,q}}$. All of its simple constituents have lowest conformal weights in $H_{p_+,p_-}$, so the definition of $\mathscr C$ gives $E_\tau\in\mathscr C$.
\end{proof}

\begin{thm}\label{thm:Nakano-input}
Let $\tau=(\alpha_1,\alpha_2,\alpha_3)\in\mathcal T_{p_+,p_-}$ satisfy $h_{\alpha_1}<h_{\alpha_2}<h_{\alpha_3}$. Assume
\begin{equation}\label{eq:self-Ext-S2-hypothesis}
\Ext^1_{\mathscr C}(\mathsf{L}_2(\tau),\mathsf{L}_2(\tau))=0,
\end{equation}
and assume also that, for every logarithmic exact sequence \(0\longrightarrow A(\tau)\longrightarrow E\longrightarrow B(\tau)\longrightarrow0\) in $\mathscr C$, quotienting by the fixed submodule $\mathsf L_2(\tau)\subset A(\tau)$ in \eqref{eq:Nakano-tower-A} gives a nonsplit exact sequence
\[
0\longrightarrow \mathsf{L}_1(\tau)\longrightarrow E/\mathsf{L}_2(\tau)\longrightarrow B(\tau)\longrightarrow0.
\]
Then
\begin{equation}\label{eq:extracted-P-Ext}
\Ext^1_{\mathscr C}\bigl(K(\tau),\mathsf{L}_2(\tau)\bigr)\cong\C.
\end{equation}
\end{thm}

\begin{proof}
We use throughout the Hom--Yoneda exact sequences of Lemma~\ref{lem:prelim-Hom-Yoneda}. Fix the logarithmic extension supplied by Lemma~\ref{lem:Nakano-existence-input}:
\begin{equation}\label{eq:local-E1}
0\longrightarrow A(\tau)\longrightarrow E_1\longrightarrow B(\tau)\longrightarrow0.
\end{equation}
Quotienting by the fixed $\mathsf L_2(\tau)\subset A(\tau)$ gives
\[
0\longrightarrow \mathsf L_1(\tau)\longrightarrow E_1/\mathsf L_2(\tau)\longrightarrow B(\tau)\longrightarrow0.
\]
By the nonsplitting hypothesis of the theorem, this sequence is nonsplit. Since \eqref{eq:Nakano-first-Ext} is one-dimensional, Lemma~\ref{lem:one-dimensional-ext} gives $E_1/\mathsf L_2(\tau)\cong K(\tau)$. Hence
\begin{equation}\label{eq:E1-over-K}
0\longrightarrow \mathsf L_2(\tau)\longrightarrow E_1\longrightarrow K(\tau)\longrightarrow0
\end{equation}
is exact.

We first show
\begin{equation}\label{eq:E1-Ext-vanishing}
\Ext^1_{\mathscr C}(E_1,\mathsf L_2(\tau))=0.
\end{equation}
Apply $\Hom(-,\mathsf L_2(\tau))$ to \eqref{eq:Nakano-tower-A}. The connecting map
\(
\End(\mathsf L_2(\tau))\to
\Ext^1_{\mathscr C}(\mathsf L_1(\tau),\mathsf L_2(\tau))
\)
sends $\id_{\mathsf L_2(\tau)}$ to the nonsplit class of $A(\tau)$. Both spaces are one-dimensional, by Schur's lemma and Lemma~\ref{lem:local-simple-Ext}, so this connecting map is an isomorphism. Together with \eqref{eq:self-Ext-S2-hypothesis}, exactness gives
\begin{equation}\label{eq:A-S2-Ext-zero}
\Ext^1_{\mathscr C}(A(\tau),\mathsf L_2(\tau))=0.
\end{equation}

Suppose, to the contrary, that \eqref{eq:E1-Ext-vanishing} fails. Applying $\Hom(-,\mathsf L_2(\tau))$ to \eqref{eq:local-E1} and using \eqref{eq:A-S2-Ext-zero} shows that
\[
\Ext^1_{\mathscr C}(B(\tau),\mathsf L_2(\tau))
\longrightarrow
\Ext^1_{\mathscr C}(E_1,\mathsf L_2(\tau))
\]
is surjective; hence the source is nonzero. Choose a nonzero extension
\[
0\longrightarrow \mathsf L_2(\tau)\longrightarrow E_2\longrightarrow B(\tau)\longrightarrow0.
\]
Lemma~\ref{lem:local-auxiliary-extension} gives that $E_2$ is logarithmic and that
\begin{equation}\label{eq:E2-second-sequence-in-thm}
0\longrightarrow B(\tau)'\longrightarrow E_2\longrightarrow \mathsf L_2(\tau)\longrightarrow0
\end{equation}
is exact.

Dualizing \eqref{eq:Nakano-tower-A} by Lemma~\ref{lem:Nakano-ambient} gives the nonsplit sequence
\[
0\longrightarrow \mathsf L_1(\tau)\longrightarrow A(\tau)'\longrightarrow \mathsf L_2(\tau)\longrightarrow0.
\]
Since $B(\tau)'$ has composition factors $\mathsf L_2(\tau)$ and $\mathsf L_3(\tau)$, one has $\Hom(B(\tau)',\mathsf L_1(\tau))=0$. Applying $\Hom(-,\mathsf L_1(\tau))$ to \eqref{eq:E2-second-sequence-in-thm} therefore gives an injection
\[
\Ext^1_{\mathscr C}(\mathsf L_2(\tau),\mathsf L_1(\tau))
\hookrightarrow
\Ext^1_{\mathscr C}(E_2,\mathsf L_1(\tau)).
\]
Let $X$ represent the image of the nonzero class of $A(\tau)'$. By the naturality of the Yoneda map in Lemma~\ref{lem:prelim-Hom-Yoneda}, there is a commutative diagram with exact rows
\[
\begin{array}{ccccccccc}
0&\longrightarrow&\mathsf L_1(\tau)&\longrightarrow&X&\longrightarrow&E_2&\longrightarrow&0\\
&&\Vert&&\downarrow&&\downarrow\\
0&\longrightarrow&\mathsf L_1(\tau)&\longrightarrow&A(\tau)'&\longrightarrow&\mathsf L_2(\tau)&\longrightarrow&0,
\end{array}
\]
where the right vertical arrow is the quotient in \eqref{eq:E2-second-sequence-in-thm}. The middle vertical arrow is surjective with kernel $B(\tau)'$, so
\[
0\longrightarrow B(\tau)'\longrightarrow X\longrightarrow A(\tau)'\longrightarrow0
\]
is also exact.

Set $E_3:=X'$. Exact contragredience gives
\[
0\longrightarrow E_2'\longrightarrow E_3\longrightarrow \mathsf L_1(\tau)\longrightarrow0,
\qquad
0\longrightarrow A(\tau)\longrightarrow E_3\longrightarrow B(\tau)\longrightarrow0,
\]
and dualizing the preceding commutative diagram gives
\[
\begin{array}{ccccccccc}
0&\longrightarrow&\mathsf L_2(\tau)&\longrightarrow&A(\tau)&\longrightarrow&\mathsf L_1(\tau)&\longrightarrow&0\\
&&\downarrow&&\downarrow&&\Vert\\
0&\longrightarrow&E_2'&\longrightarrow&E_3&\longrightarrow&\mathsf L_1(\tau)&\longrightarrow&0.
\end{array}
\]
The left vertical arrow is the kernel inclusion in the dual of \eqref{eq:E2-second-sequence-in-thm}, namely
\[
0\longrightarrow \mathsf L_2(\tau)\longrightarrow E_2'\longrightarrow B(\tau)\longrightarrow0.
\]
Since $E_2$ is logarithmic and $E_2\cong(E_2')'$, the module $E_2'$ is logarithmic; hence so is $E_3$.

Let $M:=E_3/\mathsf L_2(\tau)$, where the copy of $\mathsf L_2(\tau)$ is the fixed one in the top row of the last diagram. Quotienting the bottom row by this copy gives
\begin{equation}\label{eq:E3-first-split}
0\longrightarrow B(\tau)\longrightarrow M\longrightarrow \mathsf L_1(\tau)\longrightarrow0.
\end{equation}
The map $A(\tau)\to E_3$ in the diagram descends to a map $A(\tau)/\mathsf L_2(\tau)\cong\mathsf L_1(\tau)\to M$, and the right-hand identity shows that this map is a section of the quotient in \eqref{eq:E3-first-split}. Thus \eqref{eq:E3-first-split} splits.

On the other hand, quotienting the exact sequence $0\to A(\tau)\to E_3\to B(\tau)\to0$ by the same fixed $\mathsf L_2(\tau)$ gives the distinguished sequence
\[
0\longrightarrow \mathsf L_1(\tau)\xrightarrow{j}M\xrightarrow{q}B(\tau)\longrightarrow0.
\]
We show that this sequence also splits. Choose a splitting of \eqref{eq:E3-first-split} and write \(M=i_B(B(\tau))\oplus s(\mathsf L_1(\tau))\). Because $\Hom(\mathsf L_1(\tau),B(\tau))=0$, the $i_B(B(\tau))$-component of $j$ is zero. Hence $j(\mathsf L_1(\tau))\subseteq s(\mathsf L_1(\tau))$, and since both are nonzero simple submodules they are equal. Thus $q$ vanishes on $s(\mathsf L_1(\tau))$. It follows that $q\circ i_B:B(\tau)\to B(\tau)$ is surjective; because $B(\tau)$ has finite length, this endomorphism is an automorphism. Therefore $i_B\circ(q\circ i_B)^{-1}$ is a section of $q$. This contradicts the nonsplitting hypothesis of the theorem, because $E_3$ is logarithmic. Hence \eqref{eq:E1-Ext-vanishing} holds.

We now compute \eqref{eq:extracted-P-Ext}. From \eqref{eq:Nakano-tower-B} and $\Hom(\mathsf L_3(\tau),\mathsf L_2(\tau))=0$, \(\Hom(B(\tau),\mathsf L_2(\tau))\cong\End(\mathsf L_2(\tau))\cong\C\). Also $\Hom(A(\tau),\mathsf L_2(\tau))=0$: if a map $A(\tau)\to\mathsf L_2(\tau)$ were nonzero on the fixed submodule $\mathsf L_2(\tau)$, it would give a retraction and split \eqref{eq:Nakano-tower-A}; if it vanished there, it would factor through $\mathsf L_1(\tau)$, but $\Hom(\mathsf L_1(\tau),\mathsf L_2(\tau))=0$. Hence \eqref{eq:local-E1} gives \(\Hom(E_1,\mathsf L_2(\tau))\cong\Hom(B(\tau),\mathsf L_2(\tau))\cong\C\). Similarly, \eqref{eq:Nakano-tower-K} gives \(\Hom(K(\tau),\mathsf L_2(\tau))\cong\Hom(B(\tau),\mathsf L_2(\tau))\cong\C\). Under these identifications, the map \(\Hom(K(\tau),\mathsf L_2(\tau))\longrightarrow \Hom(E_1,\mathsf L_2(\tau))\) induced by \eqref{eq:E1-over-K} sends the nonzero map $K(\tau)\twoheadrightarrow B(\tau)\twoheadrightarrow\mathsf L_2(\tau)$ to $E_1\twoheadrightarrow B(\tau)\twoheadrightarrow\mathsf L_2(\tau)$, and is therefore an isomorphism. The restriction map \(\Hom(E_1,\mathsf L_2(\tau))\longrightarrow\End(\mathsf L_2(\tau))\) is zero, because its generator factors through $E_1\twoheadrightarrow B(\tau)$ and hence vanishes on the kernel $\mathsf L_2(\tau)$ in \eqref{eq:E1-over-K}. Applying Lemma~\ref{lem:prelim-Hom-Yoneda} to \eqref{eq:E1-over-K} and using \eqref{eq:E1-Ext-vanishing} therefore yields an isomorphism
\[
\End(\mathsf L_2(\tau))\xrightarrow{\ \sim\ }
\Ext^1_{\mathscr C}(K(\tau),\mathsf L_2(\tau)).
\]
Since $\End(\mathsf L_2(\tau))\cong\C$, this proves \eqref{eq:extracted-P-Ext}.

Finally, Proposition~\ref{prop:Nakano-literature-package}(iv) gives a nonsplit extension in $\mathscr N$ with endpoints $K_{\mathrm N}(\tau)$ and $\mathsf L_2(\tau)$. Using $K_{\mathrm N}(\tau)\cong K(\tau)$, denote its middle term by $P_{\mathrm N}(\tau)$, so
\[
0\longrightarrow \mathsf L_2(\tau)\longrightarrow P_{\mathrm N}(\tau)\longrightarrow K(\tau)\longrightarrow0.
\]
Since $K(\tau),\mathsf L_2(\tau)\in\mathscr C$, Lemma~\ref{lem:gr-extension} and the displayed short exact sequence show that $P_{\mathrm N}(\tau)$ is grading-restricted generalized, has finite length, and has only simple constituents with lowest conformal weights in $H_{p_+,p_-}$. Hence $P_{\mathrm N}(\tau)\in\mathscr C$. Since $\mathscr N$ is full in Nakano's ambient Virasoro-module category, the sequence remains nonsplit when regarded in $\mathscr C$. Thus $P_{\mathrm N}(\tau)$ represents a nonzero class in the one-dimensional space \eqref{eq:extracted-P-Ext}.
\end{proof}

By Lemma~\ref{lem:one-dimensional-ext}, all nonzero classes in \eqref{eq:extracted-P-Ext} have isomorphic middle terms. We fix one representative and denote it by $P(\tau)$, so that
\[
0\longrightarrow \mathsf L_2(\tau)\longrightarrow P(\tau)\longrightarrow K(\tau)\longrightarrow0
\]
is nonsplit. The final paragraph of the proof shows that Nakano's module $P_{\mathrm N}(\tau)$ is another such middle term; hence $P(\tau)\cong P_{\mathrm N}(\tau)$, agreeing with the notation of \cite[Definition~5.13]{Nakano}.

\subsection{The Felder dictionary for the one-row chain}\label{subsec:Felder-one-row}

With the notation $\mathcal F_\alpha$, $\alpha_{u,v;m}$, $h_{u,v;m}$, $H_{p_+,p_-}$, $\mathcal T_{p_+,p_-}$, and $\mathcal T^0_{p_+,p_-}$ fixed in Subsection~\ref{subsec:Nakano-extension-interface}, we now make the one-row Felder labels explicit. Following Nakano's reflection notation \cite[(2.8)]{Nakano}, whenever $1\le R<p_+$ and $1\le S<p_-$ we write \(R^\vee:=p_+-R,\; S^\vee:=p_--S\). Immediately after (2.4), Nakano records the identities \cite[immediately after (2.4)]{Nakano}
\begin{equation}\label{eq:Nakano-weight-shift}
h_{u,v;m}=h_{u-mp_+,v}=h_{u,v+mp_-}\qquad(u,v,m\in\Z).
\end{equation}
For $n\geq1$ and $1\leq r\leq p-1$, define the standing one-row labels
\begin{equation}\label{eq:abc}
a_{n,r}:=np-r,\qquad b_{n,r}:=np+r,\qquad c_{n,r}:=(n+2)p-r.
\end{equation}
Whenever $n,r$ are fixed in a statement or proof, we abbreviate $a:=a_{n,r}$, $b:=b_{n,r}$, and $c:=c_{n,r}$.

We also fix the Shapovalov conventions used below. Give $U(\Vir_-)$ the relative-grade grading determined by $\deg L_{-m}=m$ for $m>0$. For $N\ge0$, a partition $\lambda\vdash N$ is written $\lambda=(\lambda_1\geq\cdots\geq\lambda_\ell>0)$; for $N=0$ we use the empty partition $\varnothing$. Set \(L_{-\lambda}:=L_{-\lambda_\ell}\cdots L_{-\lambda_1},\; L_{-\varnothing}:=1\). For $N>0$, write $(1^N)$ for the partition of $N$ into $N$ parts equal to $1$. Let $\mathfrak p(N)$ denote the number of partitions of $N$, with $\mathfrak p(0)=1$ and $\mathfrak p(N)=0$ for $N<0$.

At central charge $c_{p,q}=c_{p_+,p_-}$, let $M(c_{p,q},h)$ be the Verma module fixed above and write $v_h$ for a chosen nonzero highest-weight generator. Let $\langle-,-\rangle_h$ be its standard contravariant Shapovalov form, normalized by
\[
\langle v_h,v_h\rangle_h=1,\qquad
\langle Xu,v\rangle_h=\langle u,\sigma(X)v\rangle_h
\quad(X\in U(\Vir)),
\]
where $\sigma$ is the anti-involution fixed in Subsection~\ref{subsec:Nakano-extension-interface}. The radical of this form is the maximal proper submodule of the Verma module; hence the induced contravariant form on the simple quotient $L(h)$ is nondegenerate.

For positive integers $R,S$, Nakano chooses a Shapovalov element $S_{R,S}\in U(\Vir_-)$ such that $S_{R,S}v_{h_{R,S}}$ is a nonzero singular vector of relative grade $RS$, with normalization
\[
S_{R,S}-L_{-1}^{RS}\in
\operatorname{span}_{\C}\{L_{-\lambda}:\lambda\vdash RS,\ \lambda\neq(1^{RS})\};
\]
see \cite[(3.2)]{Nakano}. To avoid collision with the global simple-module notation $S_j$, we denote Nakano's element $S_{R,S}$ by $\mathfrak S_{R,S}$.

By \cite[Proposition~5.2]{Nakano}, for $R,S>0$ one has
\[
\bigl\langle v_h,\sigma(\mathfrak S_{R,S})\mathfrak S_{R,S}v_h\bigr\rangle_h
=R_{R,S}(h-h_{R,S})+O\bigl((h-h_{R,S})^2\bigr).
\]
To avoid collision with the Felder coordinate $R$, we write $\mathscr R_{R,S}$ for this scalar $R_{R,S}$. For $N\ge0$, let
\[
\mathsf G_N(h):=\bigl(\langle L_{-\lambda}v_h,L_{-\mu}v_h\rangle_h\bigr)_{\lambda,\mu\vdash N}
\]
be the Gram matrix of the Shapovalov form on relative grade $N$ in the PBW basis fixed above. Its determinant $\det\mathsf G_N(h)$ is the level-$N$ Kac--Shapovalov determinant in this basis. For a nonzero polynomial $f(h)$, $\operatorname{ord}_{h=h_0}f$ denotes the multiplicity of $h_0$ as a zero of $f$. A change of PBW basis independent of $h$ multiplies $\det\mathsf G_N(h)$ by a nonzero constant, so its order of vanishing is basis-independent.

\begin{lem}\label{lem:primitive-Shapovalov-coefficient}
Let $R,S$ be positive integers with $1\le R<p_+$. Then $\mathscr R_{R,S}\ne0$.
\end{lem}

\begin{proof}
The same proposition gives
\[
\mathscr R_{R,S}=2\!\!\prod_{\substack{1-R\le k\le R,\ 1-S\le l\le S\\(k,l)\ne(0,0),(R,S)}}
\left(k\Bigl(\frac{p_+}{p_-}\Bigr)^{-1/2}+l\Bigl(\frac{p_+}{p_-}\Bigr)^{1/2}\right).
\]
If a factor vanished, then $kp_-+lp_+=0$. Since $\gcd(p_+,p_-)=1$, one has $p_+\mid k$. But $1-R\le k\le R$ and $R<p_+$ imply $|k|<p_+$, hence $k=0$, and then $l=0$, contrary to the omission of $(0,0)$. Thus every factor is nonzero.
\end{proof}

\begin{lem}\label{lem:Shapovalov-HC-reduction}
Let $X\in U(\Vir_-)$ be homogeneous. There is a polynomial $P_X(z)\in\C[z]$, for the fixed central charge $c_{p,q}$, such that \(\sigma(X)Xw=P_X(L_0)w\) for every vector $w$ on which $\mathbf c$ acts as $c_{p,q}$ and $\Vir_{>0}w=0$. If $X=\mathfrak S_{R,S}$, then
\[
P_X(h)=\langle Xv_h,Xv_h\rangle_h,
\qquad
P_X(h_{R,S})=0,
\qquad
P_X'(h_{R,S})=\mathscr R_{R,S}.
\]
\end{lem}

\begin{proof}
The PBW theorem for
\(
\Vir=\Vir_-\oplus(\C L_0\oplus\C\mathbf c)\oplus\Vir_{>0}
\)
gives
\[
\sigma(X)X=Q_X(L_0,\mathbf c)+Y,
\qquad
Q_X(z,\kappa)\in\C[z,\kappa],\quad
Y\in U(\Vir)\Vir_{>0},
\]
where $z$ and $\kappa$ are indeterminates. Indeed, $\sigma(X)X$ has relative grade zero. After PBW reordering, every term containing a positive mode belongs to $U(\Vir)\Vir_{>0}$; among the remaining terms, relative grade zero excludes every negative mode. Hence the remaining part is a polynomial in $L_0$ and $\mathbf c$. Set $P_X(z):=Q_X(z,c_{p,q})$. If $\Vir_{>0}w=0$, then $Yw=0$, proving the operator identity.

For a highest-weight generator $v_h$ of $M(c_{p,q},h)$, contravariance and $\langle v_h,v_h\rangle_h=1$ give \(P_X(h)=\langle v_h,\sigma(X)Xv_h\rangle_h=\langle Xv_h,Xv_h\rangle_h\). For $X=\mathfrak S_{R,S}$, the expansion fixed above therefore reads \(P_X(h)=\mathscr R_{R,S}(h-h_{R,S})+O\bigl((h-h_{R,S})^2\bigr)\). Thus $P_X(h_{R,S})=0$ and $P_X'(h_{R,S})=\mathscr R_{R,S}$.
\end{proof}

\begin{lem}\label{lem:one-row-self-Ext}
Fix $n\ge1$ and $1\le r\le p-1$, and use $a=a_{n,r}$, $b=b_{n,r}$, and $c=c_{n,r}$ from \eqref{eq:abc}. Then
\[
\Ext^1_{\mathscr C}(S_a,S_a)=\Ext^1_{\mathscr C}(S_b,S_b)=\Ext^1_{\mathscr C}(S_c,S_c)=0.
\]
\end{lem}

\begin{proof}
We first place the three lowest weights in the range of Lemma~\ref{lem:primitive-Shapovalov-coefficient}. If $p<q$, so $(p_+,p_-)=(p,q)$, the weight-shift identity \eqref{eq:Nakano-weight-shift} gives
\[
(h_a,h_b,h_c)=\bigl(h_{r,\,q-1+(n-1)q},\ h_{p-r,\,q-1+nq},\ h_{r,\,q-1+(n+1)q}\bigr).
\]
If $p>q$, so $(p_+,p_-)=(q,p)$, a direct calculation from \eqref{eq:weights} gives
\[
(h_a,h_b,h_c)=\bigl(h_{1,\,p-r+(n-1)p},\ h_{1,\,r+np},\ h_{1,\,p-r+(n+1)p}\bigr).
\]
Thus each $h\in\{h_a,h_b,h_c\}$ has a representative $h=h_{R,S}$ with $R,S>0$ and $1\le R<p_+$, so Lemma~\ref{lem:primitive-Shapovalov-coefficient} gives $\mathscr R_{R,S}\ne0$.

Fix one such $h=h_{R,S}$ and put $T:=L(h)$. Consider a self-extension \(0\longrightarrow T\xrightarrow{\iota}E\xrightarrow{\pi}T\longrightarrow0\) in $\mathscr C$. We first exclude the logarithmic case, following the argument of \cite[Proposition~5.4]{Nakano}. Since the canonical conformal nilpotent $D_E$ restricts to zero on $\iota(T)$ and induces zero on the quotient, it factors as \(D_E=\iota\circ\overline D\circ\pi, \; \overline D\in\End(T)=\C\id_T\). If $E$ is logarithmic, then $D_E\ne0$, so $\overline D=c\id_T$ for some $c\in\C^\times$; in particular $D_E^2=0$. Let $v_h$ be a nonzero highest-weight vector of $T$. By Lemma~\ref{lem:gr-extension}, choose $u_0\in E_{[h]}$ with $\pi(u_0)=v_h$, and put $u_1:=\iota(v_h)$. Then \((L_0-h)u_0=D_Eu_0=cu_1\). Moreover, $\operatorname{Spec}(E)\subset h+\Z_{\ge0}$, so $\Vir_{>0}E_{[h]}=0$.

Set $X:=\mathfrak S_{R,S}$. Since $Xv_h=0$ in the simple quotient $T$, Lemma~\ref{lem:Shapovalov-HC-reduction} and $L_0=h\id+D_E$ on $E_{[h]}$ give \(\sigma(X)Xu_0=P_X(L_0)u_0 =P_X'(h)D_Eu_0 =c\mathscr R_{R,S}u_1\ne0\). On the other hand, $\pi(Xu_0)=Xv_h=0$, hence $Xu_0\in\iota(T)$. Therefore $\sigma(X)Xu_0\in\iota(T)_{[h]}=\C u_1$. Transport the nondegenerate contravariant form on $T$ to $\iota(T)$. Then
\[
\bigl\langle u_1,\sigma(X)Xu_0\bigr\rangle
   =\bigl\langle Xu_1,Xu_0\bigr\rangle=0,
\]
because $Xu_1=0$. Hence $\sigma(X)Xu_0=0$, a contradiction. Thus no logarithmic self-extension exists.

It remains to consider an ordinary self-extension. It is generated by a highest-weight vector of the kernel and a lift of a highest-weight vector of the quotient; its $L_0$-weight spaces are finite dimensional and its spectrum is contained in $h+\Z_{\ge0}$. Hence $\Vir_{>0}$ acts locally finitely and $E\in\mathcal O^{\mathrm{Vir}}_{c_{p,q}}$. Let $\mathcal B$ be the block containing $T$. Then $E\in\mathcal B$. Proposition~\ref{prop:BNW-literature-package}(i), applied to $\Sigma=\{h\}$, identifies $\Ext^1_{\mathcal B}(T,T)$ with the corresponding self-extension group in a finite category; Proposition~\ref{prop:BNW-literature-package}(ii) makes that group zero, since a label is not a cover of itself. Hence the ordinary extension splits as well. This proves the three asserted vanishings.
\end{proof}

\begin{prop}\label{prop:Felder-labels}
Fix $n\ge1$ and $1\le r\le p-1$, and use the standing abbreviations $a=a_{n,r}$, $b=b_{n,r}$, and $c=c_{n,r}$ from \eqref{eq:abc}. Define
\[
\tau_{n,r}:=
\begin{cases}
\bigl(\alpha_{r,q-1;n-1},\alpha_{p-r,q-1;n},\alpha_{r,q-1;n+1}\bigr),&p<q,\\
\bigl(\alpha_{q-1,r;-n+1},\alpha_{q-1,p-r;-n},\alpha_{q-1,r;-n-1}\bigr),&p>q.
\end{cases}
\]
Then $\tau_{n,r}\in\mathcal T_{p_+,p_-}\setminus\mathcal T^0_{p_+,p_-}$. Writing $\tau_{n,r}=(\alpha_1,\alpha_2,\alpha_3)$, one has \((h_{\alpha_1},h_{\alpha_2},h_{\alpha_3})=(h_a,h_b,h_c)\), and hence
\[
\mathsf L_1(\tau_{n,r})\cong S_a,\qquad
\mathsf L_2(\tau_{n,r})\cong S_b,\qquad
\mathsf L_3(\tau_{n,r})\cong S_c.
\]
Moreover, \(A(\tau_{n,r})\cong K_a,\; B(\tau_{n,r})\cong K_b\). If \(0\longrightarrow K_a\longrightarrow E\longrightarrow K_b\longrightarrow0\) is exact in $\Oc_{c_{p,q}}$, then $E\in\mathscr C$. Via the fixed inclusion $S_b\hookrightarrow K_a$ in Proposition~\ref{prop:Virasoro-literature-package}(ii), regard $S_b$ as a submodule of $E$. Then $E/S_b\in\mathscr C$; we call $E/S_b$ the \emph{distinguished quotient} of the extension.
\end{prop}

\begin{proof}
Suppose first that $p<q$, so $(p_+,p_-)=(p,q)$. In \cite[Proposition~2.7(1)]{Nakano}, take $(R,S,N)=(p-r,q-1,n)$. Since $R^\vee=p_+-R=r$, the three consecutive Fock modules are precisely those labelled by \(\alpha_{r,q-1;n-1},\; \alpha_{p-r,q-1;n},\; \alpha_{r,q-1;n+1}\). Using \eqref{eq:Nakano-weight-shift} and the explicit conformal-weight formula fixed in Subsection~\ref{subsec:MY-input}, a direct calculation gives \((h_{\alpha_1},h_{\alpha_2},h_{\alpha_3})=(h_a,h_b,h_c)\). 
If $p>q$, so $(p_+,p_-)=(q,p)$, take $(R,S,N)=(q-1,p-r,-n)$ in \cite[Proposition~2.7(2)]{Nakano}. Here $S^\vee=p_- -S=r$, and the three consecutive Fock modules are precisely those labelled by \(\alpha_{q-1,r;-n+1},\; \alpha_{q-1,p-r;-n},\; \alpha_{q-1,r;-n-1}\). The same direct calculation, now with $(p_+,p_-)=(q,p)$, gives \((h_{\alpha_1},h_{\alpha_2},h_{\alpha_3})=(h_a,h_b,h_c)\). 
A direct calculation from \eqref{eq:weights} gives \(h_b-h_a=r(nq-1)>0, \; h_c-h_b=(p-r)((n+1)q-1)>0\). Thus the three Fock modules are consecutive in a Felder complex and their conformal weights satisfy $h_{\alpha_1}<h_{\alpha_2}<h_{\alpha_3}$. By the definition of $\mathcal T_{p_+,p_-}$, this proves $\tau_{n,r}\in\mathcal T_{p_+,p_-}$; the strict first inequality also gives $\tau_{n,r}\notin\mathcal T^0_{p_+,p_-}$. The identifications $\mathsf L_1(\tau_{n,r})\cong S_a$, $\mathsf L_2(\tau_{n,r})\cong S_b$, and $\mathsf L_3(\tau_{n,r})\cong S_c$ follow from the definitions of $\mathsf L_i(\tau)$ and $S_j$.

The displayed weight identities show that $h_a,h_b,h_c\in H_{p_+,p_-}$, so $S_a,S_b,S_c\in\mathscr C$. By Proposition~\ref{prop:Virasoro-literature-package}(ii),
\[
0\longrightarrow S_b\longrightarrow K_a\longrightarrow S_a\longrightarrow0,
\qquad
0\longrightarrow S_c\longrightarrow K_b\longrightarrow S_b\longrightarrow0
\]
are nonsplit. Lemma~\ref{lem:Nakano-Serre} therefore places $K_a$ and $K_b$ in $\mathscr C$, and the same sequences remain nonsplit there. Lemma~\ref{lem:local-simple-Ext} gives \(\Ext^1_{\mathscr C}(S_a,S_b)\cong\C, \; \Ext^1_{\mathscr C}(S_b,S_c)\cong\C\). By \eqref{eq:Nakano-tower-A}, \eqref{eq:Nakano-tower-B}, and Lemma~\ref{lem:one-dimensional-ext}, the middle terms of these nonzero classes are unique up to isomorphism. Hence $A(\tau_{n,r})\cong K_a$ and $B(\tau_{n,r})\cong K_b$.

Finally, $\mathscr C$ is Serre by Lemma~\ref{lem:Nakano-Serre}. Since $K_a,K_b,S_b\in\mathscr C$, every extension $0\to K_a\to E\to K_b\to0$ in $\Oc_{c_{p,q}}$ has $E\in\mathscr C$, and quotient closure gives $E/S_b\in\mathscr C$.
\end{proof}

\begin{lem}\label{lem:one-row-Nakano-scope}
Fix $n\ge1$ and $1\le r\le p-1$, and use the standing abbreviations $a=a_{n,r}$ and $b=b_{n,r}$ from \eqref{eq:abc}. Assume $a\ge2$. If \(0\longrightarrow K_a\longrightarrow E\longrightarrow K_b\longrightarrow0\) is exact in $\Oc_{c_{p,q}}$, then $E$ belongs to Nakano's category $\mathscr N=\mathcal L_{c_{p_+,p_-}}\text{-mod}$ fixed above.
\end{lem}

\begin{proof}
By Proposition~\ref{prop:Felder-labels}, $E\in\mathscr C$. Thus $E$ has finite length, all its simple composition factors have lowest conformal weights in $H_{p_+,p_-}$, and its socle series is finite by Lemma~\ref{lem:Nakano-Serre}. These are the two extra conditions in \cite[Definition~5.1]{Nakano}; it remains only to verify that $E$ lies in Nakano's ambient category of \cite[Definition~3.14]{Nakano}.

Definition~3.14(1) holds because the central element acts as $c_{p_+,p_-}=c_{p,q}$. For (2), Proposition~\ref{prop:Virasoro-literature-package}(i) gives \(\ch_zK_a=z^{h_a}\frac{1-z^a}{\prod_{m\ge1}(1-z^m)}\), so the relative degree-$N$ subspace of $K_a$ has dimension $\mathfrak p(N)-\mathfrak p(N-a)$. For $a\ge2$ this is positive for every $N\ge0$: when $N\ge a$, adjoining a part of size $a$ injects the partitions of $N-a$ into those of $N$, while the partition $1^N$ is not in the image. Hence $\operatorname{Spec}(K_a)=h_a+\Z_{\ge0}$. Since $h_b-h_a=r(nq-1)\in\Z_{>0}$, one has $\operatorname{Spec}(K_b)\subset h_a+\Z_{\ge0}$. Lemma~\ref{lem:gr-extension} therefore gives
\[
\operatorname{Spec}(E)=h_a+\Z_{\ge0},
\qquad
0<\dim E_{[h_a+N]}<\infty\quad(N\ge0).
\]
Because $E\in\Oc_{c_{p,q}}$ already has its generalized $L_0$-eigenspace decomposition, this is Definition~3.14(2) with $H_0(E)=\{h_a\}$.

For (3), Lemma~\ref{lem:Nakano-ambient} gives $E'\in\Oc_{c_{p,q}}$ with the same spectrum and generalized-weight multiplicities and a canonical isomorphism $(E')'\cong E$. Thus the same verification applies to $E'$, and $E$ and $E'$ are mutual restricted contragredients in the sense fixed above. Hence Definition~3.14(3) holds, and $E\in\mathscr N$.
\end{proof}

The remaining case $a=1$ is the one-row vacuum edge not covered by Lemma~\ref{lem:one-row-Nakano-scope}. We collect here the Kyt\"ol\"a--Ridout results needed for this edge, together with their specialization to its two endpoints. The Shapovalov form, PBW basis, Gram matrix $\mathsf G_N(h)$, partition numbers $\mathfrak p(N)$, and order of vanishing are those fixed above. We keep the rational parameter $t=q/p$ of \cite{KytolaRidout}; thus its Kac weights are our $h^{(p,q)}_{R,S}$, whereas the unadorned $h_{R,S}$ denotes the ordered pair $(p_+,p_-)$.

Let \(0\longrightarrow H^L\longrightarrow E\xrightarrow{\pi}H^R\longrightarrow0\) be a staggered module of central charge $c_{p,q}$, with highest-weight generators $x^L,x^R$ of weights $h^L<h^R$. We identify $H^L$ with its image in $E$. By Lemma~\ref{lem:gr-extension}, choose $y\in E_{[h^R]}$ with $\pi(y)=x^R$, and set
\[
\ell:=h^R-h^L,\qquad
\omega_0:=(L_0-h^R)y,\qquad \omega_1:=L_1y,\qquad \omega_2:=L_2y.
\]
These vectors lie in $H^L$ because $L_0-h^R,L_1,L_2$ annihilate $x^R$. Since $L_1,L_2$ generate $\Vir_{>0}$, the pair $(\omega_1,\omega_2)$ determines all positive-mode actions on $y$. We call this pair the \emph{data associated with the chosen lift $y$}. Any other lift $y'\in E_{[h^R]}$ of the same fixed generator $x^R$ has the form $y'=y+u$ with $u\in H^L_{[h^R]}$. The module $H^L$ is ordinary, being a quotient of an ordinary Verma module, so $(L_0-h^R)u=0$. Thus the vectors computed from $y'$ satisfy \(\omega_0'=\omega_0,\; (\omega_1',\omega_2')=(\omega_1+L_1u,\omega_2+L_2u)\). We call this change a \emph{gauge transformation}, and call two pairs of data \emph{equivalent} if they differ by such a transformation. A scalar is \emph{gauge-invariant} if its value is unchanged by every such replacement of $y$, with the endpoint generators fixed.

A proper singular vector $Xv_h$ in a Verma module is \emph{prime} if it is not a descendant of a lower-grade proper singular vector. Two singular vectors are \emph{incomparable} if neither belongs to the submodule generated by the other. A nonzero image of a prime Verma singular vector in a highest-weight quotient is called \emph{rank one}. We normalize a prime singular operator $X\in U(\Vir_-)$ of relative grade $N$ by requiring the coefficient of $L_{-1}^N$ in the PBW basis fixed above to be $1$; this coefficient is nonzero for every Verma singular vector, so the normalization is possible \cite[Section~II]{KytolaRidout}.

\begin{samepage}
When $\omega_0$ is rank one, choose the normalized prime operator $X_0$ and rescale $x^L$ so that $\omega_0=X_0x^L$. Since $\sigma(X_0)$ is a sum of positive-mode monomials of total degree $\ell>0$, its action on $y$ lies in the one-dimensional lowest-weight space of $H^L$. We define the \emph{beta-invariant} $\beta_{\mathrm{KR}}$ by \(\sigma(X_0)y=\beta_{\mathrm{KR}}x^L\). All statements about this scalar use the normalization $\omega_0=X_0x^L$.
\end{samepage}

If $H^R=M(c_{p,q},h^R)/J$, a \emph{Verma-right lift} means a staggered module $\check E$ fitting into a commutative diagram
\[
\begin{array}{ccccccccc}
0&\longrightarrow&H^L&\longrightarrow&\check E&\xrightarrow{\check\pi}&M(c_{p,q},h^R)&\longrightarrow&0\\
&&\Vert&&\big\downarrow&&\big\downarrow&&\\
0&\longrightarrow&H^L&\longrightarrow&E&\xrightarrow{\pi}&H^R&\longrightarrow&0,
\end{array}
\]
whose vertical maps are the identity on $H^L$ and the indicated quotient maps. Write $f:\check E\twoheadrightarrow E$ for the middle vertical map. The highest-weight generator $v_{h^R}$ maps to $x^R$. Choose $\check y\in\check E_{[h^R]}$ with $f(\check y)=y$; then $\check\pi(\check y)=v_{h^R}$ and, identifying the copies of $H^L$ in the two modules,
\[
(L_0-h^R)\check y=\omega_0,\qquad
L_1\check y=\omega_1,\qquad L_2\check y=\omega_2.
\]
These equalities specify precisely which vectors are preserved by the lift. Part~(ii) below supplies such a lift and describes the relation imposed to recover $E$.

\begin{prop}\label{prop:KR-literature-package}
With these conventions, the following hold.
\begin{enumerate}
\item The vector $\omega_0$ is nonzero and singular, and $\ell\in\Z_{>0}$. If $\omega_0$ is rank one, then $\beta_{\mathrm{KR}}$ is independent of the chosen lift of $x^R$: explicitly, \(\sigma(X_0)(y+u)=\beta_{\mathrm{KR}}x^L \;\bigl(u\in H^L_{[h^R]}\bigr)\). If, in addition, $H^R=M(c_{p,q},h^R)$, then
\[
\beta_{\mathrm{KR}}=0
\quad\Longleftrightarrow\quad
\text{there is }u\in H^L_{[h^R]}\text{ with }L_1(y+u)=L_2(y+u)=0.
\]

\item Suppose \(H^R=M(c_{p,q},h^R)\big/U(\Vir)Xv_{h^R}\), where $Xv_{h^R}$ is a prime singular vector of relative grade $\bar\ell$. Then a Verma-right lift $\check E$ with $\check y$ as above exists, and $X\omega_0=0$. There is $\eta\in H^L_{[h^R+\bar\ell]}$ such that
\[
\widetilde y:=X\check y-\eta,\qquad
\check\pi(\widetilde y)=Xv_{h^R},\qquad
L_n\widetilde y=0\ (n>0),\qquad
L_0\widetilde y=(h^R+\bar\ell)\widetilde y,
\]
and $E\cong\check E/U(\Vir)\widetilde y$. When $\omega_0$ is rank one, the lift preserves $\beta_{\mathrm{KR}}$.

\item Put $h_*:=(p-1)(q-1)=h^{(p,q)}_{2p-1,1}=h^{(p,q)}_{1,2q-1}$. The maximal proper submodule of $M(c_{p,q},0)$ is generated by two incomparable prime singular vectors, of relative grades $1$ and $h_*$. The module $M(c_{p,q},h_*)$ has two incomparable prime singular vectors $X_1^-v_{h_*},X_1^+v_{h_*}$, whose grades are \(\ell_1^-:=\min\{2p-1,2q-1\},\; \ell_1^+:=\max\{2p-1,2q-1\}\). Thus the degree-$(2p-1)$ prime is $X_1^-$ when $p<q$ and $X_1^+$ when $p>q$. At the fixed central charge $c_{p,q}$, the determinant of the degree-$N$ Shapovalov Gram matrix, as a polynomial in the highest weight $h$, is
\[
\det\mathsf G_N(h)
 =C_N\!\prod_{\substack{R,S\ge1\\RS\le N}}
 \bigl(h-h^{(p,q)}_{R,S}\bigr)^{\mathfrak p(N-RS)},
\]
for every $N\in\Z_{\ge0}$, where $C_N\ne0$ is independent of $h$. This identity is the \emph{Kac determinant formula}. In particular, for any $h_0\in\C$,
\[
\operatorname{ord}_{h=h_0}\det\mathsf G_N(h)
=\!\sum_{\substack{R,S\ge1,\ RS\le N\\h^{(p,q)}_{R,S}=h_0}}
\mathfrak p(N-RS).
\]

\item If $H^L=K_1$ and $H^R=K_{2p-1}$, then $\omega_0$ is rank one and, after the normalization above, $\beta_{\mathrm{KR}}\ne0$.
\end{enumerate}
\end{prop}

\begin{proof}
For (i), $H^L$ and $H^R$ are ordinary, so $D_E$ vanishes on $H^L$ and has image in $H^L$. The module $E$ is generated by $H^L$ and $y$, because $x^R$ generates its quotient. Since $D_E$ is a Virasoro-module endomorphism and $E$ is logarithmic, $D_Ey=\omega_0\ne0$. For $n>0$, the vector $L_ny$ lies in $H^L_{[h^R-n]}$, and hence \(L_n\omega_0=(L_0-h^R+n)L_ny=0\). Thus $\omega_0$ is singular and has relative grade $\ell\in\Z_{>0}$. This is the positive-grade case of \cite[Proposition~3.1 and Corollary~3.2]{KytolaRidout}.

If $\omega_0=X_0x^L$ is rank one, then $\sigma(X_0)u=0$ for every $u\in H^L_{[h^R]}$. Indeed, lift $u$ to a vector $\widehat u$ of relative grade $\ell$ in $M(c_{p,q},h^L)$. The vector $\sigma(X_0)\widehat u$ is a multiple of its highest-weight generator, and that multiple is
\[
\langle v_{h^L},\sigma(X_0)\widehat u\rangle_{h^L}
 =\langle X_0v_{h^L},\widehat u\rangle_{h^L}=0,
\]
because a proper singular vector lies in the radical of the Shapovalov form. This proves the displayed independence of $y$ and agrees with \cite[Eqs.~(3.5)--(3.6)]{KytolaRidout}. For fixed $H^L$, $H^R=M(c_{p,q},h^R)$, endpoint generators, and $\omega_0=X_0x^L$ of rank one, \cite[Proposition~3.6, Theorems~6.14--6.15, and Eq.~(6.40)]{KytolaRidout} say that two pairs of data arising from staggered modules are equivalent precisely when their beta-invariants are equal. By \cite[Proposition~6.2]{KytolaRidout}, the pair $(0,0)$ occurs as data for a staggered module with these same endpoints and normalized $\omega_0$. Its beta-invariant is zero. Thus $\beta_{\mathrm{KR}}=0$ holds precisely when there is $u\in H^L_{[h^R]}$ with $(\omega_1+L_1u,\omega_2+L_2u)=(0,0)$, which is the asserted criterion. Since $L_1,L_2$ generate $\Vir_{>0}$, this also gives $L_n(y+u)=0$ for every $n>0$.

For (ii), \cite[Proposition~4.6 and Corollary~4.7]{KytolaRidout} give a Verma-right lift with $f(\check y)=y$ and the three displayed equalities for $\omega_0,\omega_1,\omega_2$. Since $Xy\in H^L$ and $H^L$ is ordinary, \(X\omega_0=(L_0-h^R-\bar\ell)Xy=0\). The singular lift and the quotient realization are \cite[Proposition~7.2, Eqs.~(7.3)--(7.7)]{KytolaRidout}. Their precise content is that the imposed right relation has a singular preimage $X\check y-\eta$, so quotienting by the submodule it generates preserves the left module. Equality of the data, with the same normalized $\omega_0$, gives equality of the beta-invariants; see also the opening discussion of \cite[Section~VII]{KytolaRidout}.

For (iii), the Feigin--Fuchs embedding classification recalled in \cite[Section~II]{KytolaRidout} applies to
\[
0=h^{(p,q)}_{1,1}=h^{(p,q)}_{p-1,q-1},\qquad
h_*=h^{(p,q)}_{2p-1,1}=h^{(p,q)}_{1,2q-1}.
\]
Both configurations have two incomparable prime singular vectors, with grades given by the products of the displayed positive Kac labels. Indeed, $p,q\ge2$ are coprime, and neither pair of labels lies on a boundary $p\mid R$ or $q\mid S$ of this classification. Their grades are distinct because $p\ne q$. The determinant formula is \cite[Eq.~(2.11)]{KytolaRidout}, expressed in our already fixed Shapovalov basis. Its order of vanishing at $h_0$ is the sum of the exponents of the factors with $h^{(p,q)}_{R,S}=h_0$, which gives the last displayed identity in (iii).

For (iv), part~(i) makes $\omega_0$ a nonzero singular vector in $K_1$ of conformal weight $h_*$. By the one-row Verma-quotient description in Subsection~\ref{subsec:MY-input}, \(K_1=M(c_{p,q},0)\big/U(\Vir)L_{-1}v_{c_{p,q},0}\). Part~(iii) shows that its other prime singular vector, of relative grade $h_*$, survives. By Proposition~\ref{prop:Virasoro-literature-package}(ii), $K_1$ has simple submodule $S_{2p-1}$ and simple quotient $S_1$. The image of $\omega_0$ in $S_1$ is a positive-grade singular vector, hence zero. Thus $\omega_0\in S_{2p-1}$. Its lowest-weight space is one-dimensional, so $\omega_0$ is a nonzero multiple of the surviving prime vector. Rescaling $x^L$ gives $\omega_0=X_0x^L$ with $X_0$ normalized, as required for the beta-invariant.

The same one-row description gives $K_{2p-1}$ as the quotient of $M(c_{p,q},h_*)$ by its degree-$(2p-1)$ singular vector, which is prime by (iii). Part~(ii) supplies the Verma-right lift, the singular lift of this prime, and the compatibility $X\omega_0=0$. Suppose that $\beta_{\mathrm{KR}}=0$. The lift has the same beta-invariant by (ii), so (i) allows us to replace $\check y$ by $\check y+u$, with $u\in H^L_{[h_*]}$, so that $L_n\check y=0$ for every $n>0$.

If $p<q$, the imposed prime has the smallest positive grade $2p-1$. The lift, the vanishing of its beta-invariant, the compatibility $X\omega_0=0$, and its singular lift are precisely the hypotheses of \cite[Proposition~7.5]{KytolaRidout}. That proposition gives positive integers $R,S$ such that \(h_*=h^{(p,q)}_{R,S},\; p\mid R,\; q\mid S\). Writing $R=pA$, $S=qB$ and using the Kac-weight formula yields \(pq(A-B)=\pm\bigl(q(2p-1)-p\bigr)=\pm(2pq-p-q)\). Reduction modulo $p$ contradicts $\gcd(p,q)=1$.

If $p>q$, the imposed prime is the higher-grade prime $X_1^+$, with $\ell_1^-=2q-1$ and $\ell_1^+=2p-1$. For $\beta_{\mathrm{KR}}=0$ and $X_1^+\omega_0=0$, the higher-prime extension of the argument stated immediately after \cite[Proposition~7.5]{KytolaRidout} requires \(\operatorname{ord}_{h=h_*}\det\mathsf G_{2p-1}(h) \ge\mathfrak p(2p-2q)+2\) for a singular lift to exist. This is the determinant criterion in that paragraph, obtained by refining the Shapovalov-pairing calculation of \cite[Eqs.~(7.29)--(7.31)]{KytolaRidout} to allow descendants of the lower first prime. The submodule generated by that prime is a Verma module, whose descendant space at grade $\ell_1^+$ has dimension $\mathfrak p(\ell_1^+-\ell_1^-)$. Thus the lower- and higher-grade cases use the two respective assertions of that source.

We compute the actual order from (iii). Put $d:=2pq-p-q>0$. Equality $h^{(p,q)}_{R,S}=h_*$ is equivalent to $|qR-pS|=d$, whose integer solutions are \((R,S)=(2p-1+mp,1+mq) \;\text{or}\; (R,S)=(1+mp,2q-1+mq),\; m\in\Z\). Positivity of both coordinates forces $m\ge0$ in either family. For $m\ge1$, their products exceed $2p-1$; thus the only positive solutions with $RS\le2p-1$ are $(2p-1,1)$ and $(1,2q-1)$. The determinant formula therefore gives
\[
\operatorname{ord}_{h=h_*}\det\mathsf G_{2p-1}(h)
=\mathfrak p(0)+\mathfrak p(2p-2q)
=\mathfrak p(2p-2q)+1,
\]
contradicting the required order. Hence $\beta_{\mathrm{KR}}\ne0$ in both cases.
\end{proof}

\begin{lem}\label{lem:vacuum-edge-staggered}
Let
\[
0\longrightarrow K_1\xrightarrow{\iota}E\xrightarrow{\pi}K_{2p-1}\longrightarrow0
\]
be a logarithmic exact sequence in $\Oc_{c_{p,q}}$. Identify $K_1$ with its image under $\iota$, and use the fixed inclusion $S_{2p-1}\hookrightarrow K_1$ of Proposition~\ref{prop:Virasoro-literature-package}(ii) to regard $S_{2p-1}$ as a submodule of $E$. Then the induced exact sequence
\begin{equation}\label{eq:vacuum-edge-quotient}
0\longrightarrow S_1\longrightarrow E/S_{2p-1}\xrightarrow{\bar\pi}K_{2p-1}\longrightarrow0
\end{equation}
is nonsplit.
\end{lem}

\begin{proof}
The fixed inclusion has $K_1/S_{2p-1}\cong S_1$ by Proposition~\ref{prop:Virasoro-literature-package}(ii), so quotienting the given extension yields \eqref{eq:vacuum-edge-quotient}. Write $\rho:E\twoheadrightarrow E/S_{2p-1}$ for the quotient map; the induced map $\bar\pi$ satisfies $\bar\pi\rho=\pi$.

Proposition~\ref{prop:Felder-labels}, with $(n,r)=(1,p-1)$, gives $K_1\cong A(\tau_{1,p-1})$, $K_{2p-1}\cong B(\tau_{1,p-1})$, and $E\in\mathscr C$. Under these endpoint identifications, Lemma~\ref{lem:local-staggered-realization} shows that $E$ is a staggered module. Recall that $h_*=(p-1)(q-1)=h_{2p-1}>0$. Fix a nonzero highest-weight generator $x^R$ of $K_{2p-1}$, and use Lemma~\ref{lem:gr-extension} to choose $y\in E_{[h_*]}$ with $\pi(y)=x^R$. By Proposition~\ref{prop:KR-literature-package}(iv), choose a highest-weight generator $v_1$ of $K_1$ and a normalized prime singular operator $X_0\in U(\Vir_-)$ of relative grade $h_*$ so that
\[
(L_0-h_*)y=X_0v_1,\qquad
\sigma(X_0)y=\beta_{\mathrm{KR}}v_1,\qquad
\beta_{\mathrm{KR}}\ne0.
\]

Suppose that \eqref{eq:vacuum-edge-quotient} splits. Then there is a Virasoro-module map $s:K_{2p-1}\to E/S_{2p-1}$ with $\bar\pi s=\id$. By Lemma~\ref{lem:gr-extension}, choose $y'\in E_{[h_*]}$ with $\rho(y')=s(x^R)$. The identity \(\pi(y')=\bar\pi\rho(y')=\bar\pi s(x^R)=x^R=\pi(y)\) gives $y'-y\in(K_1)_{[h_*]}$. Thus Proposition~\ref{prop:KR-literature-package}(i), with the same $v_1$ and $X_0$, gives $\sigma(X_0)y'=\beta_{\mathrm{KR}}v_1$. On the other hand, \(\rho\bigl(\sigma(X_0)y'\bigr) =\sigma(X_0)s(x^R) =s\bigl(\sigma(X_0)x^R\bigr)=0\), because $\sigma(X_0)$ is a sum of positive-mode monomials and $x^R$ is highest weight. Consequently \(\sigma(X_0)y'\in\ker\rho\cap E_{[0]}=(S_{2p-1})_{[0]}=0\), where the last equality uses the lowest conformal weight $h_*>0$ of $S_{2p-1}$. This contradicts $\sigma(X_0)y'=\beta_{\mathrm{KR}}v_1\ne0$. Hence \eqref{eq:vacuum-edge-quotient} is nonsplit.
\end{proof}

\begin{prop}\label{prop:one-row-staggered-quotient}
Let $n\ge1$ and $1\le r\le p-1$, and use the standing abbreviations
$a=a_{n,r}$ and $b=b_{n,r}$ from \eqref{eq:abc}. Then the following hold.
\begin{enumerate}
\item Let \(0\longrightarrow K_a\xrightarrow{\iota}E\xrightarrow{\pi}K_b\longrightarrow0\) be a logarithmic exact sequence in $\Oc_{c_{p,q}}$. Identify $K_a$ with its image under $\iota$, and regard $S_b$ as a submodule of $E$ using the fixed inclusion $S_b\hookrightarrow K_a$ of Proposition~\ref{prop:Virasoro-literature-package}(ii). Then the induced exact sequence
\begin{equation}\label{eq:one-row-distinguished-nonsplit}
0\longrightarrow S_a\longrightarrow E/S_b\longrightarrow K_b\longrightarrow0
\end{equation}
is nonsplit.
\item In $\mathscr C$ one has
\[
\Ext^1_{\mathscr C}(K_b,S_a)\cong\C,
\qquad
\Ext^1_{\mathscr C}\bigl(K(\tau_{n,r}),S_b\bigr)\cong\C.
\]
\end{enumerate}
\end{prop}

\begin{proof}
Put $\tau:=\tau_{n,r}=(\alpha_1,\alpha_2,\alpha_3)$. Proposition~\ref{prop:Felder-labels} gives $\tau\in\mathcal T_{p_+,p_-}\setminus\mathcal T^0_{p_+,p_-}$, $h_{\alpha_1}<h_{\alpha_2}<h_{\alpha_3}$, and the identifications
\[
K_a\cong A(\tau),\qquad K_b\cong B(\tau),\qquad
S_a\cong\mathsf L_1(\tau),\qquad S_b\cong\mathsf L_2(\tau).
\]
Fix isomorphisms $\phi:K_a\xrightarrow{\sim}A(\tau)$ and $\psi:K_b\xrightarrow{\sim}B(\tau)$. The fixed submodule $S_b$ is the unique simple submodule of $K_a$: any other simple submodule would intersect $S_b$ trivially and map isomorphically onto $S_a$, splitting the nonsplit sequence of Proposition~\ref{prop:Virasoro-literature-package}(ii). The same argument applied to \eqref{eq:Nakano-tower-A} shows that $\mathsf L_2(\tau)$ is the unique simple submodule of $A(\tau)$. Hence \(\phi(S_b)=\mathsf L_2(\tau)\). 
For (i), Proposition~\ref{prop:Felder-labels} gives $E,E/S_b\in\mathscr C$. Since $S_b\subseteq\ker\pi=K_a$ and $K_a/S_b\cong S_a$, quotienting the given extension yields \eqref{eq:one-row-distinguished-nonsplit}.
If $a\ge2$, Lemma~\ref{lem:one-row-Nakano-scope} gives $E\in\mathscr N$.
Using $\iota\phi^{-1}:A(\tau)\to E$ and $\psi\pi:E\to B(\tau)$ gives a logarithmic extension with the endpoints in Proposition~\ref{prop:Nakano-literature-package}(iii). That proposition applies to every such extension in $\mathscr N$ and gives a nonsplit quotient sequence. Since $\phi(S_b)=\mathsf L_2(\tau)$, this is precisely \eqref{eq:one-row-distinguished-nonsplit}.
If $a=1$, the inequality $a=np-r\ge(n-1)p+1$ forces $n=1$ and $r=p-1$; then $b=2p-1$, and Lemma~\ref{lem:vacuum-edge-staggered} proves the assertion.

For (ii), the endpoint identifications and \eqref{eq:Nakano-first-Ext} give
\[
\Ext^1_{\mathscr C}(K_b,S_a)
\cong
\Ext^1_{\mathscr C}\bigl(B(\tau),\mathsf L_1(\tau)\bigr)
\cong\C.
\]
To apply Theorem~\ref{thm:Nakano-input}, first note that Lemma~\ref{lem:one-row-self-Ext} gives
\[
\Ext^1_{\mathscr C}\bigl(\mathsf L_2(\tau),\mathsf L_2(\tau)\bigr)
\cong\Ext^1_{\mathscr C}(S_b,S_b)=0.
\]
Now let
\[
0\longrightarrow A(\tau)\xrightarrow{\widetilde\iota}\widetilde E\xrightarrow{\widetilde\pi}B(\tau)\longrightarrow0
\]
be any logarithmic exact sequence in $\mathscr C$. The maps $\widetilde\iota\phi:K_a\to\widetilde E$ and $\psi^{-1}\widetilde\pi:\widetilde E\to K_b$ give an exact sequence with endpoints $K_a,K_b$ in $\Oc_{c_{p,q}}$. Part~(i) says that its quotient by the image of $S_b$ is nonsplit. This image is $\widetilde\iota\phi(S_b)=\widetilde\iota\bigl(\mathsf L_2(\tau)\bigr)$, so the nonsplitting hypothesis of Theorem~\ref{thm:Nakano-input} holds for every such $\widetilde E$. That theorem therefore gives
\[
\Ext^1_{\mathscr C}\bigl(K(\tau),S_b\bigr)
\cong
\Ext^1_{\mathscr C}\bigl(K(\tau),\mathsf L_2(\tau)\bigr)
\cong\C.
\]
\end{proof}

\subsection{Rank-one Drinfeld--Sokolov conventions and external input}\label{subsec:DS-prelim}

We fix the rank-one principal $+$ BRST complex in the conventions of \cite[Sections~4.4--4.5]{ArakawaW}.  Retain the affine basis and modes of Subsection~\ref{subsec:affine-projectives}, and put $\mathfrak h:=\C h$ and $\mathfrak n_+:=\C e$.  The positive root $\alpha\in\mathfrak h^*$ satisfies $\alpha(h)=2$, with root vectors $e_\alpha=e$ and $e_{-\alpha}=f$.  Write $\varpi:=\alpha/2$ for the fundamental weight; the half-sums of positive roots and coroots are $\rho=\varpi$ and $\check\rho=h/2$, respectively.  The dual Coxeter number is $h^\vee=2$.  A level-$k$ affine module $M$ is \emph{smooth} if, for every $m\in M$, there is an integer $N\ge0$ such that $(\mathfrak{sl}_2\otimes t^N\C[t])m=0$.  Smoothness allows formal Laurent-series currents to act by the same mode expansion, since only finitely many terms act nontrivially on a fixed vector.  Objects of $\KL^k$ are smooth by Lemma~\ref{lem:KL-smooth}.

The charged-fermion ghost Fock space $\mathcal F_{\mathrm{gh}}$ is the irreducible vacuum module for the Clifford algebra with odd generators $\psi_\alpha(n),\psi_{-\alpha}(n)$ ($n\in\Z$) and relations
\[
[\psi_\alpha(m),\psi_{-\alpha}(n)]_+=\delta_{m+n,0},\qquad [\psi_\alpha(m),\psi_\alpha(n)]_+=[\psi_{-\alpha}(m),\psi_{-\alpha}(n)]_+=0,
\]
where $[A,B]_+:=AB+BA$.  Its vacuum satisfies $\psi_\alpha(n)\one_{\mathrm{gh}}=0$ for $n\ge0$ and $\psi_{-\alpha}(n)\one_{\mathrm{gh}}=0$ for $n>0$.  Thus the ordered products of distinct creation modes $\psi_\alpha(n)$ ($n<0$) and $\psi_{-\alpha}(n)$ ($n\le0$), applied to $\one_{\mathrm{gh}}$, form a basis.  The cohomological grading is determined by $\deg\one_{\mathrm{gh}}=0$, $\deg\psi_\alpha(n)=-1$, and $\deg\psi_{-\alpha}(n)=1$; let $\mathcal F_{\mathrm{gh}}^j$ denote its degree-$j$ subspace.  Its vertex superalgebra structure is determined by the generating fields
\[
\psi_\alpha(z)=\sum_{n\in\Z}\psi_\alpha(n)z^{-n-1},\qquad
\psi_{-\alpha}(z)=\sum_{n\in\Z}\psi_{-\alpha}(n)z^{-n},
\]
corresponding to the states $\psi_\alpha(-1)\one_{\mathrm{gh}}$ and $\psi_{-\alpha}(0)\one_{\mathrm{gh}}$, respectively \cite[Section~4.4]{ArakawaW}.

For a smooth level-$k$ affine module $M$, placed in cohomological degree zero, set \(C_+(M):=M\otimes\mathcal F_{\mathrm{gh}}\) and \(C_+^j(M):=M\otimes\mathcal F_{\mathrm{gh}}^j.\) The principal character $\Psi:\mathfrak n_+\otimes\C[t,t^{-1}]\to\C$ is defined by
\begin{equation}\label{eq:principal-character}
\Psi(et^n)=\delta_{n,-1}.
\end{equation}
Equivalently, on formal Laurent series it is $\Psi(e\,a(t))=\Res_{t=0}a(t)\,dt$ \cite[Section~2]{ArakawaFrenkel}.  Since $\mathfrak n_+$ is abelian, the cubic ghost term in \cite[Section~4.5]{ArakawaW} vanishes.  On $C_+(M)$ the differential is
\[
Q_+^{\mathrm{st}}=\sum_{n\in\Z}e(-n)\otimes\psi_{-\alpha}(n),\qquad
d_M=Q_+:=Q_+^{\mathrm{st}}+\id_M\otimes\psi_{-\alpha}(1).
\]
For $n\ll0$ the affine factor annihilates any fixed vector by smoothness, and for $n\gg0$ the ghost factor does so by the vacuum conditions and Clifford relations.  Hence this sum acts by finitely many terms on each vector.  The modes $e(m)$ commute with one another and the modes $\psi_{-\alpha}(n)$ anticommute, so $d_M^2=0$; also $d_M$ raises cohomological degree by one.  Define \(H^i_{DS,+}(M):=H^i(C_+(M),d_M),\; H(M):=H^0_{DS,+}(M)\;(M\in\KL^k)\). An affine-module map $g:M\to N$ gives the chain map $g\otimes\id_{\mathcal F_{\mathrm{gh}}}$, which defines the maps on cohomology.  Exactness of $H$ and its membership in the target category will be proved in Section~\ref{sec:DS}.

The vacuum complex $C_+(V^k(\mathfrak{sl}_2))$ is a differential graded vertex superalgebra, and $C_+(M)$ is a differential graded module over it.  Therefore its cohomology acts on $H^\bullet_{DS,+}(M)$ \cite[Section~4.5]{ArakawaW}, \cite[Section~2.1]{ArakawaFrenkel}.  In particular, define the universal principal $W$-algebra by \(W^k(\mathfrak{sl}_2):=H^0_{DS,+}(V^k(\mathfrak{sl}_2)).\) We specify its conformal vector explicitly.  Let $\omega_{\mathrm{aff}}$ be the Sugawara vector, let $\one$ be the affine vacuum, and set
\[
\omega_{\mathrm f}:=\psi_{-\alpha}(-1)\psi_\alpha(-1)\one_{\mathrm{gh}},\qquad
j_{\mathrm{gh}}:=\psi_\alpha(-1)\psi_{-\alpha}(0)\one_{\mathrm{gh}}.
\]
Here $\omega_{\mathrm f}$ is the ghost conformal vector of \cite[Section~4.4]{ArakawaW}.  With $Y_{\mathrm{gh}}$ denoting the ghost state-field map, the notation \(:\!\psi_\alpha(z)\psi_{-\alpha}(z)\!:\ :=Y_{\mathrm{gh}}(j_{\mathrm{gh}},z)\) specifies the normally ordered ghost current.  Let $T$ be the translation operator of the vacuum complex.  At our noncritical level $k\ne-2$, Section~4.17 of \cite{ArakawaW} gives the BRST-closed conformal state
\[
\begin{aligned}
\omega_C={}&\omega_{\mathrm{aff}}\otimes\one_{\mathrm{gh}}+\one\otimes\omega_{\mathrm f}\\
&+T\left(\tfrac12h(-1)\one\otimes\one_{\mathrm{gh}}+\one\otimes j_{\mathrm{gh}}\right).
\end{aligned}
\]
Write $\omega_{DS}:=[\omega_C]\in W^k(\mathfrak{sl}_2)$ for its cohomology class.  On a cohomology module use the mode convention \(Y(\omega_{DS},z)=\sum_{n\in\Z}L^{DS}_nz^{-n-2}\); in particular $L^{DS}_0$ is the coefficient of $z^{-2}$.  Its central charge is \(c(k):=13-6(k+2)-\frac{6}{k+2}\). The rank-one specialization of \cite[Proposition~4.12.1(iii)]{ArakawaW} gives one generator of conformal weight two and a Poincar\'e--Birkhoff--Witt (PBW) basis of its ordered negative-mode monomials.  In particular, the weight-two space is one-dimensional and the generator has a nonzero translate.  Thus the translation operator on $W^k(\mathfrak{sl}_2)$ is nonzero, so its conformal vector $\omega_{DS}$ cannot vanish.  It is consequently a nonzero scalar multiple of that generator, and the vectors
\[
L^{DS}_{-n_1}\cdots L^{DS}_{-n_s}\one_W,
\qquad n_1\ge\cdots\ge n_s\ge2,\quad s\ge0,
\]
form a basis, where $\one_W$ is the cohomology class of the vacuum and $s=0$ gives $\one_W$.  The universal property of $V^{\mathrm{Vir}}_{c(k)}$ now gives a conformal homomorphism to $W^k(\mathfrak{sl}_2)$ sending its standard vacuum PBW basis to this basis.  It is therefore an isomorphism, and we fix the identification \(W^k(\mathfrak{sl}_2)\cong V^{\mathrm{Vir}}_{c(k)}\) as conformal vertex algebras.  Since $k+2=p/q$, one has $c(k)=c_{p,q}$.  Accordingly, each $H^i_{DS,+}(M)$ is initially a weak $V^{\mathrm{Vir}}_{c_{p,q}}$-module; the finite-length and grading-restriction assertions needed for $H(M)\in\Oc_{c_{p,q}}$ are established in Section~\ref{sec:DS}.

For the Arakawa--Frenkel level convention, let $\kappa_0$ be the invariant form normalized so that the long root has squared length $2$ \cite[Section~3.1]{ArakawaFrenkel}.  Thus $\kappa_0$ is precisely the form $(\,,\,)$ fixed in Subsection~\ref{subsec:affine-projectives}.  Put
\[
\gamma_{\mathrm{AF}}:=k+h^\vee=k+2=\frac pq,\qquad
\kappa_{\mathrm{AF}}:=\gamma_{\mathrm{AF}}\kappa_0,\qquad
\kappa_c:=-2\kappa_0.
\]
Here $\kappa_c$ is the critical form.  Arakawa--Frenkel use $\kappa_{\mathrm{AF}}$ as their form parameter, and their affine cocycle is $\kappa_{\mathrm{AF}}+\kappa_c=k\kappa_0$ \cite[Section~2]{ArakawaFrenkel}; after Section~3.1 they also use the scalar $\kappa_{\mathrm{AF}}/\kappa_0=\gamma_{\mathrm{AF}}$ for this parameter.  Thus their universal affine vertex algebra $V^{\kappa_{\mathrm{AF}}}(\mathfrak{sl}_2)$ is our $V^k(\mathfrak{sl}_2)$.  Their principal $W$-algebra $\mathscr W^{\kappa_{\mathrm{AF}}}(\mathfrak{sl}_2)$ is identified with $W^k(\mathfrak{sl}_2)$ by Lemma~\ref{lem:BRST-convention-comparison} below.

For $r\ge1$, put $\lambda_r:=(r-1)\varpi$ and choose a nonzero highest-weight vector $v_r$ in the horizontal top $M_r$ of $V_r$.  The Weyl module $\mathbb V_{\lambda_r,\kappa_{\mathrm{AF}}}$ of \cite[Section~2.2]{ArakawaFrenkel} is our generalized Verma module $V_r$: both are induced from $M_r$, with positive affine modes acting trivially and affine cocycle $k\kappa_0$.

We also fix the affine Verma module used later.  Let $\mathfrak b=\C e\oplus\C h$ and \(\widehat{\mathfrak b}:=\mathfrak b\otimes1\oplus(\mathfrak{sl}_2\otimes t\C[t])\oplus\C\mathbf{k}.\) Let $\C_{r-1,k}=\C\,1$ be the one-dimensional $\widehat{\mathfrak b}$-module on which $e(0)$ and $\mathfrak{sl}_2\otimes t\C[t]$ act trivially, $h(0)$ acts by $r-1$, and $\mathbf{k}$ acts by $k$.  Define
\[
\mathbb M^k_{r-1}:=U(\widehat{\mathfrak{sl}}_2)\otimes_{U(\widehat{\mathfrak b})}\C_{r-1,k}.
\]
With the choice of $v_r$, its highest-weight generator maps to $v_r$, giving the surjection $\mathbb M^k_{r-1}\twoheadrightarrow V_r$.  Its kernel is the submodule generated by $f(0)^r(1\otimes1)$: this imposes the relation defining the finite-dimensional horizontal quotient $M_r$.

For affine-root labels, adjoin the degree derivation $\mathbf d$ with $[\mathbf d,a(n)]=n a(n)$ and $[\mathbf d,\mathbf k]=0$, and use the extended Cartan subalgebra $\widehat{\mathfrak h}:=\mathfrak h\oplus\C\mathbf k\oplus\C\mathbf d$ \cite[Section~4.1]{ArakawaW}.  Extend finite weights by zero on $\mathbf k$ and $\mathbf d$.  Define the null root $\delta$ by $\delta|_{\mathfrak h\oplus\C\mathbf k}=0$ and $\delta(\mathbf d)=1$, and the affine fundamental weight $\Lambda_0$ by $\Lambda_0|_{\mathfrak h\oplus\C\mathbf d}=0$ and $\Lambda_0(\mathbf k)=1$.  The finite simple coroot is $\alpha^\vee=h$, and $(-\alpha+\delta)^\vee=\mathbf k-h$.  Put
\[
\widehat\lambda_r:=\lambda_r+k\Lambda_0=(r-1)\varpi+k\Lambda_0,
\qquad \widehat\rho:=\varpi+2\Lambda_0.
\]
Here $\widehat\rho$ is the affine Weyl vector.  Write $L(\widehat\lambda_r)$ for the irreducible affine highest-weight module with this weight, restricted to the affine algebra without $\mathbf d$; then $L_r=L(\widehat\lambda_r)$.

Finally, let $H^\bullet_{DS,0,\mathrm{AF}}(M)$ denote the cohomology of the Arakawa--Frenkel complex with twisting coweight $\check\mu=0$ \cite[Section~2.1]{ArakawaFrenkel}.  The subscript $0$ records this parameter; the character in its differential is still the principal character \eqref{eq:principal-character}.

\begin{lem}\label{lem:BRST-convention-comparison}
With the conventions above, for every smooth level-$k$ affine module $M$ the $\check\mu=0$ Arakawa--Frenkel complex is naturally isomorphic to $(C_+(M),d_M)$ as a cochain complex.  The isomorphism preserves cohomological degree.  On the vacuum module it induces a conformal identification $\mathscr W^{\kappa_{\mathrm{AF}}}(\mathfrak{sl}_2)\cong W^k(\mathfrak{sl}_2)$; under this identification the natural isomorphisms \(H^i_{DS,0,\mathrm{AF}}(M)\cong H^i_{DS,+}(M)\) intertwine the $W$-algebra actions for every $i\in\Z$.
\end{lem}

\begin{proof}
Let $e_\alpha^*\in\mathfrak n_+^*$ be the functional with $e_\alpha^*(e_\alpha)=1$; the normalization $(e,f)=1$ identifies it with $f$.  Write $\mathcal F^{\mathrm{AF}}_{\mathrm{gh}}$ and $\one^{\mathrm{AF}}_{\mathrm{gh}}$ for the Arakawa--Frenkel ghost space and its vacuum.  Their Clifford generators $\psi^{\mathrm{AF}}_{\alpha,n}$ and $(\psi^{\mathrm{AF}}_{\alpha,n})^*$ correspond to $e_\alpha t^n$ and $e_\alpha^*t^{n-1}dt$, respectively \cite[Section~2]{ArakawaFrenkel}.  The residue pairing is
\[
\left\langle e_\alpha^*t^{n-1}dt,e_\alpha t^m\right\rangle
=\Res_{t=0}t^{m+n-1}dt=\delta_{m+n,0}.
\]
Thus their Clifford relations agree with ours with the same mode indices.  Their vacuum annihilation ranges and cohomological degrees \cite[(2.3) and the following paragraph]{ArakawaFrenkel} also agree with those fixed above.  The identification of Clifford generators therefore induces the unique graded Fock-space isomorphism $\iota_{\mathrm{gh}}:\mathcal F^{\mathrm{AF}}_{\mathrm{gh}}\xrightarrow{\sim}\mathcal F_{\mathrm{gh}}$ satisfying, for $n\in\Z$,
\[
\iota_{\mathrm{gh}}\psi^{\mathrm{AF}}_{\alpha,n}
=\psi_\alpha(n)\iota_{\mathrm{gh}},\,
\iota_{\mathrm{gh}}(\psi^{\mathrm{AF}}_{\alpha,n})^*
=\psi_{-\alpha}(n)\iota_{\mathrm{gh}},\,
\iota_{\mathrm{gh}}\one^{\mathrm{AF}}_{\mathrm{gh}}
=\one_{\mathrm{gh}}.
\]
Write $C_{\mathrm{AF}}(M):=M\otimes\mathcal F^{\mathrm{AF}}_{\mathrm{gh}}$ and let $d^{\mathrm{AF}}_M$ denote its differential at $\check\mu=0$.  Since $\mathfrak n_+$ is abelian, \cite[(2.4), (2.7)]{ArakawaFrenkel} and \cite[Section~4.5]{ArakawaW} give
\[
d^{\mathrm{AF}}_M
=\sum_{n\in\Z}e(-n)\otimes(\psi^{\mathrm{AF}}_{\alpha,n})^*
+\id_M\otimes(\psi^{\mathrm{AF}}_{\alpha,1})^*.
\]
Consequently $I_M:=\id_M\otimes\iota_{\mathrm{gh}}$ satisfies $I_Md^{\mathrm{AF}}_M=d_MI_M$.  For every affine-module map $g:M\to N$,
\(
I_N(g\otimes\id_{\mathcal F^{\mathrm{AF}}_{\mathrm{gh}}})
=(g\otimes\id_{\mathcal F_{\mathrm{gh}}})I_M,
\)
which proves naturality.

The ghost mode identifications also identify the generating ghost fields with those displayed above \cite[Section~4.4]{ArakawaW}.  Hence $I_{V^k(\mathfrak{sl}_2)}$ is a vacuum-preserving isomorphism of differential graded vertex superalgebras, and $I_M$ intertwines their actions on the module complexes.  The formula for $\omega_C$ is preserved term by term.  At $\check\mu=0$, the Arakawa--Frenkel prescription for the module action leaves every affine and ghost generator unchanged \cite[Section~2.1]{ArakawaFrenkel}, so its induced $W$-action is exactly this cohomology action.  Passing to cohomology gives the conformal $W$-algebra identification and the asserted module isomorphisms.
\end{proof}

We apply Arakawa--Frenkel results through Lemma~\ref{lem:BRST-convention-comparison}.  The affine derivative term in $\omega_C$ contributes $-h(0)/2$ to the Virasoro zero mode, while the ghost terms vanish on the ghost vacuum in this mode.  Since $h(0)v_r=(r-1)v_r$, the weight associated with $v_r\otimes\one_{\mathrm{gh}}$ is
\begin{equation}\label{eq:DS-weight-dictionary}
\Delta^{\mathrm{aff}}_r-\frac{r-1}{2}
=\frac{(qr-p)^2-(p-q)^2}{4pq}=h_r,
\end{equation}
using \eqref{eq:affine-lowest-weight}.  This is the rank-one conformal-weight formula of \cite[Section~4.4]{ArakawaFrenkel}, with horizontal highest weight $r-1$, scalar parameter $\gamma_{\mathrm{AF}}=k+2$, and twisting coweight zero.

For a nonzero grading-restricted Virasoro module $M$ with $\operatorname{Spec}(M)\subset h+\Z_{\ge0}$ and $M_{[h]}\ne0$, the lowest conformal weight $h$ is uniquely determined.  Define its relative graded character by \(\ch^{\mathrm{rel}}_zM:=z^{-h}\ch_zM=\sum_{n\ge0}\dim M_{[h+n]}z^n,\) with the graded-character convention of Subsection~\ref{subsec:one-row-residue}.

\begin{lem}
Let $M$ be a smooth level-$k$ affine module.\label{lem:BRST-Hamiltonian}  Write $L^{C(M)}(z)$ for the field of $\omega_C$ on $C_+(M)$, and define the modes by
\begin{align*}
&L^{C(M)}(z)=\sum_{n\in\Z}L^{C(M)}_nz^{-n-2},\,
Y_{\mathrm{gh}}(\omega_{\mathrm f},z)=\sum_{n\in\Z}L^{\mathrm f}_nz^{-n-2},\\
&J^{\mathrm{gh}}(z):=Y_{\mathrm{gh}}(j_{\mathrm{gh}},z)
=\sum_{n\in\Z}J^{\mathrm{gh}}_nz^{-n-1}.
\end{align*}
Thus $L^{C(M)}_0$ and $L^{\mathrm f}_0$ are coefficients of $z^{-2}$, whereas $J^{\mathrm{gh}}_0$ is the coefficient of $z^{-1}$.  Operators on either tensor factor are extended to $C_+(M)$ by the identity on the other factor.  Then
\begin{equation}\label{eq:BRST-Hamiltonian}
L^{C(M)}_0=L^{\mathrm{aff}}_0-\frac12h(0)+L^{\mathrm f}_0-J^{\mathrm{gh}}_0.
\end{equation}
This operator commutes with $d_M$ and induces $L^{DS}_0$ on every $H^i_{DS,+}(M)$.
\end{lem}

\begin{proof}
The displayed formula for $\omega_C$, together with the translation identity for vertex operators, gives
\[
L^{C(M)}(z)=Y_M(\omega_{\mathrm{aff}},z)
+Y_{\mathrm{gh}}(\omega_{\mathrm f},z)
+\partial_z\left(\tfrac12h(z)+J^{\mathrm{gh}}(z)\right),\,
h(z)=\sum_{n\in\Z}h(n)z^{-n-1}.
\]
For a field $a(z)=\sum_n a_nz^{-n-1}$, its derivative is \(\partial_z a(z)=\sum_n-(n+1)a_nz^{-n-2}\). Its coefficient of $z^{-2}$ is therefore $-a_0$.  Extracting that coefficient proves \eqref{eq:BRST-Hamiltonian}.  BRST closure of $\omega_C$ \cite[Section~4.17]{ArakawaW} gives $[d_M,L^{C(M)}(z)]=0$ on the module complex.  Hence $L^{C(M)}_0$ descends to cohomology, where it is the zero mode of $[\omega_C]=\omega_{DS}$.
\end{proof}

The following proposition collects the external BRST input and the grading calculations used in Section~\ref{sec:DS}.

\begin{prop}\label{prop:DS-literature-package}
At the fixed level $k=-2+p/q$, with the conventions above, the following hold for every $r\ge1$.
\begin{enumerate}
\item One has $H^i_{DS,+}(L_r)=0$ for $i\ne0$, and $H(L_r)$ is a nonzero irreducible highest-weight Virasoro module of lowest conformal weight $h_r$.
\item One has $H^i_{DS,+}(V_r)=0$ for $i\ne0$ and \(H^0_{DS,+}(\mathbb M^k_{r-1})\cong M(c_{p,q},h_r).\) The canonical quotient $\mathbb M^k_{r-1}\twoheadrightarrow V_r$ induces a surjection \(M(c_{p,q},h_r)\twoheadrightarrow H(V_r)\) that sends a highest-weight generator to $[v_r\otimes\one_{\mathrm{gh}}]$.
\item The module $H(V_r)$ is ordinary, with lowest conformal weight $h_r$ and relative character \(\ch^{\mathrm{rel}}_zH(V_r)=\frac{1-z^r}{\prod_{m\ge1}(1-z^m)}\). The auxiliary grading used for this character in Arakawa--Frenkel is given on $C_+(V_r)$ by
\(G_{\mathrm{AF}}=L^{C(V_r)}_0-h_r\id\), and on cohomology by $L^{DS}_0-h_r\id$.
\item On the ghost-monomial basis fixed above, $L^{\mathrm f}_0$ and $J^{\mathrm{gh}}_0$ are simultaneously diagonalizable.  Both annihilate $\one_{\mathrm{gh}}$, and for $n\in\Z$,
\[
\begin{aligned}
[L^{\mathrm f}_0,\psi_\alpha(n)]&=-n\psi_\alpha(n),&
[L^{\mathrm f}_0,\psi_{-\alpha}(n)]&=-n\psi_{-\alpha}(n),\\
[J^{\mathrm{gh}}_0,\psi_\alpha(n)]&=\psi_\alpha(n),&
[J^{\mathrm{gh}}_0,\psi_{-\alpha}(n)]&=-\psi_{-\alpha}(n).
\end{aligned}
\]
In particular, $J^{\mathrm{gh}}_0$ has integral spectrum and \(\operatorname{Spec}(L^{\mathrm f}_0-J^{\mathrm{gh}}_0)\subset\Z_{\ge0}\). \item For every smooth level-$k$ affine module $M$, the complex $(C_+(M),d_M)$ is the $+$ quantum Drinfeld--Sokolov complex of Frenkel--Kac--Wakimoto.  The two reduction functors in the McRae--Yang conjecture are this construction at the affine levels $k=-2+p/q$ and $k'=-2+q/p$, respectively.  The same $+$ convention is used at both levels.
\end{enumerate}
\end{prop}

\begin{proof}
For (i), the level parameter of \cite{ArakawaVanishing} is the shifted level $k+2=p/q$ in our notation.  The sufficient nonintegrality hypothesis of \cite[Theorem~8.3(i)]{ArakawaVanishing} and \cite[Theorems~9.1.3--9.1.4]{ArakawaW} specializes in type $A_1$ to
\[
\left\langle\widehat\lambda_r+\widehat\rho,(-\alpha+\delta)^\vee\right\rangle
=k+2-r=\frac pq-r\notin\Z.
\]
Here the coroot is $\mathbf k-h$, and the last assertion follows from $q\ge2$ and $\gcd(p,q)=1$.  The cited results give higher-cohomology vanishing and a nonzero irreducible highest-weight reduction.  Its conformal weight is \(\frac{(r-1)(r+1)}{4(k+2)}-\frac{r-1}{2}=h_r\), the rank-one weight formula \eqref{eq:DS-weight-dictionary}.

For (ii), apply the Weyl-module vanishing theorem \cite[Theorem~2.1]{ArakawaFrenkel} through Lemma~\ref{lem:BRST-convention-comparison}.  The rank-one Verma-reduction statement, with the indicated highest-weight generator, is \cite[Theorem~6.3.1]{ArakawaSuperconformal}, recalled in \cite[Section~4.4]{ArakawaFrenkel}; its Virasoro weight is again \eqref{eq:DS-weight-dictionary}.  To verify surjectivity, put \(N_r:=\ker(\mathbb M^k_{r-1}\twoheadrightarrow V_r)\). By the induced-module description above and PBW, $N_r$ is the affine Verma module generated by $f(0)^r(1\otimes1)$, whose horizontal highest weight is $-r-1$.  For the highest-weight vanishing theorem, extend these Verma modules to the degree derivation $\mathbf d$ by letting it annihilate the highest-weight generator and using $[\mathbf d,a(n)]=n a(n)$; this extension does not change their BRST complexes.  Thus $H^1_{DS,+}(N_r)=0$ by \cite[Theorem~5.7]{ArakawaVanishing}.  The long exact cohomology sequence of \(0\to N_r\to\mathbb M^k_{r-1}\to V_r\to0\) therefore gives the asserted surjection.  This uses vanishing on an affine Verma module and the Weyl-module theorem at the present level.

For (iv), the ghost conformal structure \cite[Section~4.4]{ArakawaW} and the Clifford relations give the displayed commutators; the vacuum normalization follows from the states $\omega_{\mathrm f}$ and $j_{\mathrm{gh}}$ fixed above.  If
\[
u=\psi_\alpha(n_1)\cdots\psi_\alpha(n_s)
\psi_{-\alpha}(m_1)\cdots\psi_{-\alpha}(m_t)\one_{\mathrm{gh}}
\]
is a basis monomial, with $s,t\in\Z_{\ge0}$, $n_j<0$ $(1\le j\le s)$, and $m_j\le0$ $(1\le j\le t)$, then
\[
L^{\mathrm f}_0u=\left(-\sum_{j=1}^s n_j-\sum_{j=1}^t m_j\right)u,
\qquad J^{\mathrm{gh}}_0u=(s-t)u.
\]
The eigenvalue of $L^{\mathrm f}_0-J^{\mathrm{gh}}_0$ is \(\sum_{j=1}^s(-n_j-1)+\sum_{j=1}^t(-m_j+1)\in\Z_{\ge0}\). This proves simultaneous diagonalizability and both spectral assertions.

For (iii), Section~4.3 of \cite{ArakawaFrenkel} assigns auxiliary degree zero to $v_r\otimes\one_{\mathrm{gh}}$ and assigns to the modes the degrees
\[
\begin{aligned}
e(n),\ \psi_\alpha(n)&:\ -n-1,&
f(n),\ \psi_{-\alpha}(n)&:\ -n+1,&
h(n)&:\ -n.
\end{aligned}
\]
These are changes in the auxiliary degree, separately from the cohomological degrees fixed earlier.  Let $G_{\mathrm{AF}}$ act by the auxiliary degree; thus, for example, $[G_{\mathrm{AF}},e(n)]=(-n-1)e(n)$, and similarly for the other listed modes.  Lemma~\ref{lem:BRST-Hamiltonian}, the affine mode relations, and (iv) show that $L^{C(V_r)}_0-h_r\id$ has exactly these commutators with every affine and ghost mode.  By \eqref{eq:DS-weight-dictionary} it also annihilates $v_r\otimes\one_{\mathrm{gh}}$.  This vector generates $C_+(V_r)$ under the affine modes and ghost creation modes, so induction on a PBW monomial gives
\(G_{\mathrm{AF}}=L^{C(V_r)}_0-h_r\id\) on the complex, and hence the asserted grading on cohomology.  The character formula of \cite[Section~4.3]{ArakawaFrenkel} holds for every value of its level parameter.  At $\lambda_r=(r-1)\varpi$ and twisting coweight zero, its two Weyl-group terms are \(z^{r/2}(z^{-r/2}-z^{r/2})=1-z^r\), and its denominator is $\prod_{m\ge1}(1-z^m)$.  This proves the character formula with the stated Virasoro grading.  The quotient in (ii) inherits the semisimple $L_0$-action and finite-dimensional weight spaces of the Virasoro Verma module; the constant term of the relative character is $1$.  Thus it is ordinary and has lowest conformal weight $h_r$.

For (v), identify the FKW modes $\varphi_\alpha(n)$ and $\varphi_\alpha^*(n)$ with $\psi_\alpha(n)$ and $\psi_{-\alpha}(n)$, respectively.  Their $+$ vacuum conditions are exactly ours \cite[Section~2.3]{FKW}, and their character term is $\varphi_\alpha^*(1)$ \cite[Section~2.2]{FKW}.  Since $\mathfrak n_+$ is abelian, their differential is therefore precisely \(\sum_n e(-n)\otimes\psi_{-\alpha}(n)+\id_M\otimes\psi_{-\alpha}(1)=d_M\). Their rank-one Virasoro field \cite[Section~3.1]{FKW} is the field of $\omega_C$.  This also agrees with the character convention in \cite[Introduction and Section~4]{ArakawaVanishing}.  FKW's $-$ reduction vanishes when $f(0)$ acts locally nilpotently \cite[Theorem~2.3]{FKW}, as it does on $\KL^k$ by horizontal local finiteness in Proposition~\ref{prop:MY-affine-literature-package}(i).  The reduction in \cite[Conjecture~7.16]{McRaeYang} is consequently the $+$ construction at each of its two affine levels.  The ordered target normalizations are those already fixed in Subsection~\ref{subsec:MY-input}, with the parameter-swap identity \eqref{eq:Kac-parameter-swap}; the calculation $H(V_2)\cong K_2$ is proved in Proposition~\ref{prop:H-generalized Verma}.
\end{proof}

\subsection{Tilting and diagrammatic data}\label{subsec:TL-prelim}

We collect the tilting and Temperley--Lieb input needed for faithfulness on projectives.  Set
\[
\zeta:=e^{\pi iq/p},\qquad \tau_{\mathrm{TL}}:=-\zeta,\qquad
\delta_{\mathrm{TL}}:=\tau_{\mathrm{TL}}+\tau_{\mathrm{TL}}^{-1}=-\zeta-\zeta^{-1}.
\]
For a root of unity $a\in\C^\times$, let $\operatorname{ord}_{\times}(a)$ be its multiplicative order.  For $a\in\C^\times\setminus\{\pm1\}$ and $m\ge1$, define the \emph{quantum integer} and, when it exists, the \emph{critical index} by
\[
[m]_a:=\frac{a^m-a^{-m}}{a-a^{-1}},\qquad
\ell_{\mathrm{TL}}(a):=\min\{m\ge1:[m]_a=0\},\qquad
\delta(a):=a+a^{-1}.
\]
Since $(p,q)=1$, $\zeta^2$ and $\tau_{\mathrm{TL}}^2$ have multiplicative order $p$.  Moreover, $[m]_{-a}=(-1)^{m-1}[m]_a$, so
\[
\ell_{\mathrm{TL}}(\zeta)=\ell_{\mathrm{TL}}(\tau_{\mathrm{TL}})=p,
\qquad \delta_{\mathrm{TL}}=\delta(\tau_{\mathrm{TL}})=-[2]_\zeta.
\]

Let $\mathcal C(\zeta,\mathfrak{sl}_2)$ be the quantum-group category of \cite[Section~6.1]{McRaeYang}, realized as finite-dimensional comodules for the quantized coordinate Hopf algebra $SL_\zeta(2)$.  We identify the integral weight $m\varpi$ with its integer label $m\in\Z$, and let $\mathbf X_\zeta$ be the two-dimensional standard object, with weights $1$ and $-1$.  Define $\mathcal T_\zeta$ to be the full subcategory whose objects are direct summands of finite direct sums of tensor powers $\mathbf X_\zeta^{\otimes n}$, $n\ge0$.  Following \cite[Section~6.1]{McRaeYang} and \cite[Remark~2.11]{SuttonTubbenhauerWedrichZhu}, we use this concrete description of the \emph{tilting category}.  Write $T_\lambda$ for the indecomposable object of highest weight $\lambda\ge0$ in the classification recalled below.  It is the object denoted $\texttt{T}(\lambda)$ in \cite{SuttonTubbenhauerWedrichZhu}.

For $\delta\in\C$, the strict $\C$-linear \emph{Temperley--Lieb category} $\TL_\delta$ has objects $n\in\Z_{\ge0}$ and tensor product $m\otimes n=m+n$, with unit $0$.  The space $\Hom_{\TL_\delta}(m,n)$ has as basis planar noncrossing matching diagrams with $m$ lower and $n$ upper boundary points; it is zero when $m+n$ is odd.  Composition is vertical concatenation, read from bottom to top, and tensor product is horizontal juxtaposition.  Each closed circle is removed with a factor $\delta$, and $\id_n$ is the diagram of $n$ vertical strands.  Cups $0\to2$ and caps $2\to0$ give the duality maps.  Closing an endomorphism diagram defines its scalar trace $\operatorname{tr}$.  A morphism $f:m\to n$ is \emph{negligible} if $\operatorname{tr}(gf)=0$ for all $g:n\to m$.  A \emph{tensor ideal} $J$ is a family of linear subspaces $J(m,n)\subseteq\Hom_{\TL_\delta}(m,n)$ stable under precomposition, postcomposition, and tensoring on either side by arbitrary morphisms.  These are the conventions of \cite[Section~2.2]{SuttonTubbenhauerWedrichZhu}, with circle value $\delta$, and \cite[Section~1]{GoodmanWenzl}.

For a $\C$-linear monoidal category $\mathcal A$, its \emph{additive Karoubi envelope} $\operatorname{Kar}^{\oplus}(\mathcal A)$ has objects finite direct sums of pairs $(X,e)$, where $e^2=e\in\End_{\mathcal A}(X)$.  Its morphisms between two such pairs are \(\Hom\bigl((X,e),(Y,f)\bigr)=f\Hom_{\mathcal A}(X,Y)e\), with identity $e$ on $(X,e)$; morphisms between sums are matrices.  Tensor product is induced by $(X,e)\otimes(Y,f)=(X\otimes Y,e\otimes f)$.  A \emph{splitting} of $e$ consists of maps $R\xrightarrow{i}X\xrightarrow{r}R$ satisfying $ri=\id_R$ and $ir=e$; $R$ is the corresponding \emph{retract}.  In an idempotent-complete category every idempotent admits a splitting.

\begin{prop}\label{prop:tilting-literature-package}
With these conventions:
\begin{enumerate}
\item $\mathcal T_\zeta$ is an additive idempotent-complete rigid monoidal category.  Its indecomposables are, up to isomorphism, the objects $T_\lambda$, $\lambda\ge0$, uniquely determined by their highest weights; $T_0$ is the unit and $T_1=\mathbf X_\zeta$.  Every $T_\lambda$ is a retract of $\mathbf X_\zeta^{\otimes\lambda}$.  There is a $\C$-linear monoidal equivalence, which we fix,
\begin{equation}\label{eq:tilting-projective-equivalence}
\mathscr E_\zeta:\mathcal T_\zeta\xrightarrow{\sim}\mathcal P^k,
\qquad
\mathscr E_\zeta(\mathbf X_\zeta)\cong V_2,
\qquad
\mathscr E_\zeta(T_\lambda)\cong P_{\lambda+1}.
\end{equation}
\item There is a $\C$-linear strong monoidal functor, which we also fix,
\[
\mathscr D_\zeta^{\mathrm{TL}}:\TL_{\delta_{\mathrm{TL}}}\longrightarrow\mathcal T_\zeta,
\qquad \mathscr D_\zeta^{\mathrm{TL}}(n)=\mathbf X_\zeta^{\otimes n}.
\]
It is fully faithful and extends to a monoidal equivalence
\(\operatorname{Kar}^{\oplus}(\TL_{\delta_{\mathrm{TL}}})\xrightarrow{\sim}\mathcal T_\zeta\).
\end{enumerate}
\end{prop}

\begin{proof}
The classification and rigidity are recalled in \cite[Section~6.1]{McRaeYang}.  There the highest weight $\lambda$ occurs with multiplicity one in $T_\lambda$, and tensoring by $\mathbf X_\zeta$ exhibits $T_{\lambda+1}$ as a summand of $\mathbf X_\zeta\otimes T_\lambda$.  Starting with $T_0$, this proves the asserted tensor-power retracts.  Proposition~6.4 and Theorem~6.6 of that paper give \eqref{eq:tilting-projective-equivalence}.  For (ii), the complex specialization of \cite[Section~2.2]{SuttonTubbenhauerWedrichZhu} at quantum parameter $\zeta$ has circle value $-[2]_\zeta=\delta_{\mathrm{TL}}$, so it is identified with our diagram category by the identity on diagrams.  Their Proposition~2.13 supplies the functor and its equivalence after additive idempotent completion.  The canonical embedding into that completion is fully faithful, so the original functor is fully faithful as well.
\end{proof}

We next define the ordinary \emph{Jones--Wenzl idempotents} for circle value $\delta(a)=a+a^{-1}$.  For fixed $n\ge2$, let $u_i\in\End_{\TL_{\delta(a)}}(n)$, $1\le i<n$, be the diagram with a cap joining lower points $i,i+1$, a cup joining upper points $i,i+1$, and all other strands vertical.  Thus $u_i^2=\delta(a)u_i$.  Set $JW_0(a)=\id_0$ and $JW_1(a)=\id_1$.  Whenever $[j]_a\ne0$ for $1\le j\le n$, define recursively
\begin{align*}
Q_n(a)&:=JW_{n-1}(a)\otimes\id_1,\\
JW_n(a)&:=Q_n(a)-\frac{[n-1]_a}{[n]_a}\,Q_n(a)u_{n-1}Q_n(a)
\qquad(n\ge2).
\end{align*}
This is the usual Jones--Wenzl recursion.  The idempotent has coefficient $1$ on the identity diagram and satisfies
\begin{align*}
JW_n(a)^2=JW_n(a),\,
u_iJW_n(a)=JW_n(a)u_i=0\,(1\le i<n),\,
\operatorname{tr}(JW_n(a))=[n+1]_a.
\end{align*}
These properties and the recursion are recalled in \cite[Introduction and Proposition~2.17]{SuttonTubbenhauerWedrichZhu}; replacing $a$ with $-a$ gives their circle value $-[2]_a$.  In particular, $JW_{p-1}(\tau_{\mathrm{TL}})$ is defined, since its denominators have indices $j<p$.  For comparison, its recursion coefficient is
\[
-\frac{[n-1]_{\tau_{\mathrm{TL}}}}{[n]_{\tau_{\mathrm{TL}}}}
=\frac{[n-1]_\zeta}{[n]_\zeta}\qquad(2\le n<p),
\]
as in that paper's convention with circle value $-[2]_\zeta$.

\begin{lem}\label{lem:TL-ideal-detection}
Let $a\in\C^\times$ satisfy $\operatorname{ord}_{\times}(a^2)=p\ge3$.  Then $\ell_{\mathrm{TL}}(a)=p$, and every nonzero proper tensor ideal of $\TL_{\delta(a)}$ contains the nonzero idempotent $JW_{p-1}(a)$.
\end{lem}

\begin{proof}
We have $a\ne\pm1$, and $[m]_a=0$ precisely when $a^{2m}=1$.  Thus $[p]_a=0$ and $[j]_a\ne0$ for $1\le j<p$, which proves both the critical-index assertion and existence of $JW_{p-1}(a)$.  Write $w:=JW_{p-1}(a)$.  Its identity-diagram coefficient is $1$, so $w\ne0$.  The algebra $\End_{\TL_{\delta(a)}}(p-1)$ is generated by its identity and the cup--cap diagrams $u_i$.  Since $wu_i=u_iw=0$, every endomorphism $g$ satisfies $wgw=c_gw$ for a scalar $c_g$.  Cyclicity of the diagrammatic trace therefore gives \(\operatorname{tr}(gw)=\operatorname{tr}(wgw)=c_g\operatorname{tr}(w)=c_g[p]_a=0\). Hence $w$ is negligible.  Theorem~3.3 of \cite{GoodmanWenzl} says that the negligible morphisms form the unique nonzero proper tensor ideal at a root of unity; their critical parameter is the order of $a^2$ \cite[Section~2.4]{GoodmanWenzl}.  Our hypothesis $p\ge3$ places us in this specialization and gives $\delta(a)\ne0$.  Every nonzero proper tensor ideal is consequently that negligible ideal and contains $w$.
\end{proof}

\begin{lem}\label{lem:first-critical-tilting-projector}
Assume $p\ge3$.  The idempotent
\[
e_{p-1}^{\mathrm{tilt}}:=\mathscr D_\zeta^{\mathrm{TL}}\bigl(JW_{p-1}(\tau_{\mathrm{TL}})\bigr)
\in\End_{\mathcal T_\zeta}(\mathbf X_\zeta^{\otimes(p-1)})
\]
has retract $T_{p-1}$.
\end{lem}

\begin{proof}
Put $n=p-1$ and $w=JW_n(\tau_{\mathrm{TL}})$, and let $R$ be the retract of $\mathscr D_\zeta^{\mathrm{TL}}(w)$.  The annihilation relations above and generation of $\End_{\TL_{\delta_{\mathrm{TL}}}}(n)$ by the $u_i$ give \(w\End_{\TL_{\delta_{\mathrm{TL}}}}(n)w=\C w\). Since $w\ne0$, full faithfulness in Proposition~\ref{prop:tilting-literature-package}(ii) gives $R\ne0$ and $\End_{\mathcal T_\zeta}(R)\cong\C$.  Thus $R$ is indecomposable.  The highest weight of $\mathbf X_\zeta^{\otimes n}$ is $n$, with multiplicity one.  Every nonidentity basis diagram on $n$ strands factors through the object $n-2$, so its image under $\mathscr D_\zeta^{\mathrm{TL}}$ factors through $\mathbf X_\zeta^{\otimes(n-2)}$, which has no weight $n$.  These maps preserve the integral weights.  The identity-diagram coefficient of $w$ is $1$, so $\mathscr D_\zeta^{\mathrm{TL}}(w)$ acts as the identity on the weight-$n$ line.  Consequently $R$ has highest weight $n$, and the classification in Proposition~\ref{prop:tilting-literature-package}(i) gives $R\cong T_n=T_{p-1}$.
\end{proof}

\section{Drinfeld--Sokolov reduction}\label{sec:DS}

Retain the principal $+$ BRST complex $(C_+(M),d_M)$ and the notation $H=H^0_{DS,+}$ from Subsection~\ref{subsec:DS-prelim}. The external reduction results are collected in Proposition~\ref{prop:DS-literature-package}; this section derives the properties of $H$ needed for comparison with $F$.

\subsection{BRST vanishing and exactness}

\begin{lem}\label{lem:KL-smooth}
Every object of $\KL^k(\mathfrak{sl}_2)$ is a smooth level-$k$ affine module, and the forgetful functor to smooth level-$k$ affine modules is exact.
\end{lem}

\begin{proof}
Let $M\in\KL^k(\mathfrak{sl}_2)$ and $v\in M$. The lower-truncation axiom for the affine vertex operators gives an integer $N(v)\ge0$ such that $e(n)v=f(n)v=h(n)v=0$ for every $n\ge N(v)$. Since $e,f,h$ form a basis of $\mathfrak{sl}_2$, one has $(\mathfrak{sl}_2\otimes t^{N(v)}\C[t])v=0$, which is smoothness as defined in Subsection~\ref{subsec:DS-prelim}. Kernels and cokernels in $\KL^k(\mathfrak{sl}_2)$ are the usual module kernels and cokernels, so the forgetful functor is exact.
\end{proof}

\begin{prop}\label{prop:H-simple}
For every $r\ge1$, $H(L_r)\cong S_r$ and $H^i_{DS,+}(L_r)=0$ for $i\ne0$. In particular, every simple object has nonzero simple reduction.
\end{prop}

\begin{proof}
Proposition~\ref{prop:DS-literature-package}(i) gives the vanishing and a nonzero irreducible highest-weight Virasoro reduction of lowest conformal weight $h_r$. By the definition $S_r=L(c_{p,q},h_r)$ in Subsection~\ref{subsec:MY-input}, this reduction is $S_r$.
\end{proof}

\begin{prop}\label{prop:H-exact}
For every $M\in\KL^k(\mathfrak{sl}_2)$, one has $H^i_{DS,+}(M)=0$ for $i\ne0$. The functor $H:\KL^k(\mathfrak{sl}_2)\longrightarrow\Oc_{c_{p,q}}$ is $\C$-linear and exact. If $\ell(-)$ denotes Jordan--H\"older length, namely the number of simple factors in a composition series, then
\begin{equation}\label{eq:H-length}
\ell(H(M))=\ell(M).
\end{equation}
\end{prop}

\begin{proof}
By Lemma~\ref{lem:KL-smooth}, every short exact sequence in $\KL^k(\mathfrak{sl}_2)$ is a short exact sequence of smooth level-$k$ affine modules. In the construction of Subsection~\ref{subsec:DS-prelim}, the cochain functor $C_+^j(-)$ is tensoring with the fixed vector space $\mathcal F_{\mathrm{gh}}^j$, and an affine-module map $f$ induces the chain map $f\otimes\id_{\mathcal F_{\mathrm{gh}}}$. Tensoring over $\C$ is exact, so every such sequence gives a degreewise short exact sequence of BRST complexes. Its long exact cohomology sequence consists of Virasoro-module maps for the induced action fixed in that subsection. At this stage the cohomology modules are weak $V^{\mathrm{Vir}}_{c_{p,q}}$-modules. The construction is $\C$-linear on morphisms, as is passage to cohomology.

Induct on $\ell(M)$. The zero module has zero cohomology. For $M\ne0$, choose an exact sequence $0\longrightarrow M_0\longrightarrow M\longrightarrow L_r\longrightarrow0$ with $\ell(M_0)<\ell(M)$; the simple quotient has this form by Proposition~\ref{prop:MY-affine-literature-package}(i). For $i\ne0$, the long exact sequence contains \(H^i_{DS,+}(M_0)\longrightarrow H^i_{DS,+}(M)\longrightarrow H^i_{DS,+}(L_r)\). The two outer terms vanish by induction and Proposition~\ref{prop:H-simple}, so the middle term vanishes as well.

For an arbitrary short exact sequence $0\to A\to B\to C\to0$ in $\KL^k(\mathfrak{sl}_2)$, the same long exact sequence contains
\[
H^{-1}_{DS,+}(C)\longrightarrow H(A)\longrightarrow H(B)
\longrightarrow H(C)\longrightarrow H^1_{DS,+}(A).
\]
The two outer terms vanish by the preceding argument. Thus $H$ is exact as a functor to weak $V^{\mathrm{Vir}}_{c_{p,q}}$-modules.

Let $0=M_0\subset M_1\subset\cdots\subset M_\ell=M$ be a composition series with $M_i/M_{i-1}\cong L_{r_i}$. Exactness gives a filtration $0=H(M_0)\subset H(M_1)\subset\cdots\subset H(M_\ell)=H(M)$ with successive quotients $S_{r_i}$ by Proposition~\ref{prop:H-simple}. These are nonzero simple modules, so the filtration is a composition series and \eqref{eq:H-length} follows.

Each $S_{r_i}$ is a grading-restricted highest-weight module. Applying Lemma~\ref{lem:gr-extension} successively to this filtration proves that $H(M)$ is grading-restricted generalized. Moreover, the one-row quotient description in Subsection~\ref{subsec:MY-input} supplies a nonzero singular vector of relative degree $r_i$ in $M(c_{p,q},h_{r_i})$, so this Verma module is reducible. The finite-length characterization of $\Oc_{c_{p,q}}$ in that subsection therefore gives $H(M)\in\Oc_{c_{p,q}}$. Together with the weak-module exactness above, this proves the asserted exact functor with target $\Oc_{c_{p,q}}$.
\end{proof}

\begin{cor}\label{cor:H-faithful}
The functor $H$ is faithful.
\end{cor}

\begin{proof}
Let $0\ne f:M\to N$ be a morphism in $\KL^k(\mathfrak{sl}_2)$. Exactness gives $\operatorname{Im}H(f)\cong H(\operatorname{Im}f)$. Since $\operatorname{Im}f\ne0$, equation~\eqref{eq:H-length} gives $\ell(H(\operatorname{Im}f))=\ell(\operatorname{Im}f)>0$. Hence $H(f)\ne0$.
\end{proof}

\subsection{Reduction of affine generalized Verma modules}

\begin{prop}\label{prop:H-generalized Verma}
For every $r\geq1$, $H(V_r)\cong K_r.$
\end{prop}

\begin{proof}
Fix $r\ge1$ and put $\mathcal M:=M(c_{p,q},h_r)$.  Proposition~\ref{prop:DS-literature-package}(ii) gives a Virasoro-module surjection \(\pi:\mathcal M\twoheadrightarrow H(V_r)\). The target is ordinary by part~(iii), and $\pi$ commutes with $L_0$, hence preserves the conformal-weight spaces.

The PBW basis of $\mathcal M$, Proposition~\ref{prop:DS-literature-package}(iii), and Proposition~\ref{prop:Virasoro-literature-package}(i) give
\[
\ch_z\mathcal M=\frac{z^{h_r}}{\prod_{m\ge1}(1-z^m)},
\qquad
\ch_zH(V_r)=\ch_zK_r=(1-z^r)\ch_z\mathcal M.
\]
The identification $G_{\mathrm{AF}}=L^{DS}_0-h_r\id$ in the first of these propositions ensures that the character of $H(V_r)$ uses the conformal grading.  All comparisons are coefficientwise after removing the common factor $z^{h_r}$.

Put $N:=\ker\pi$.  Exactness on conformal-weight spaces gives \(\ch_zN=\ch_z\mathcal M-\ch_zH(V_r) =\frac{z^{h_r+r}}{\prod_{m\ge1}(1-z^m)}\). Thus $N_{[h_r+n]}=0$ for integers $n<r$, and $\dim_\C N_{[h_r+r]}=1$.
Choose $0\ne u\in N_{[h_r+r]}$.  Since $N$ is a Virasoro submodule and $[L_0,L_m]=-mL_m$, one has \(L_m u\in N_{[h_r+r-m]}=0\;(m>0)\). Hence $u$ is singular.

By the one-row Verma-quotient description in Subsection~\ref{subsec:MY-input}, there is a nonzero singular vector $v_{r,1}$ of relative degree $r$ such that \(K_r\cong\mathcal M/U(\Vir)v_{r,1}\), Uniqueness of the singular line at that degree, recorded in the same subsection, shows that $u$ is a nonzero scalar multiple of $v_{r,1}$.  Consequently $U(\Vir)v_{r,1}\subseteq N$, and $\pi$ induces a surjection $\bar\pi:K_r\twoheadrightarrow H(V_r)$.  On each conformal-weight space, this is a surjective map between finite-dimensional spaces of equal dimension by the character formula above.  It is therefore bijective on every such space, so $H(V_r)\cong K_r$.
\end{proof}

\subsection{The exponential comparison}

\begin{prop}\label{prop:H-twist}
For every $M\in\KL^k(\mathfrak{sl}_2)$, the affine minus twist from \eqref{eq:minus-twist} satisfies \(H(\theta^-_M)=e^{2\pi iL^{DS}_0}\big|_{H(M)}\). This equality is natural in $M$.
\end{prop}

\begin{proof}
First suppose $M\in\KL^k_{\bar\epsilon}(\mathfrak{sl}_2)$, with $\epsilon\in\{0,1\}$.  The operators $L^{\mathrm{aff}}_0$, $h(0)$, $L^{\mathrm f}_0$, and $J^{\mathrm{gh}}_0$ on $C_+(M)$ commute pairwise: $[L^{\mathrm{aff}}_0,h(0)]=0$, operators on separate tensor factors commute, and the ghost operators have a common eigenbasis by Proposition~\ref{prop:DS-literature-package}(iv).  Since $M$ is grading restricted, it is the algebraic direct sum of finite-dimensional generalized $L^{\mathrm{aff}}_0$-eigenspaces, each stable under $h(0)$.  Together with the common ghost eigenbasis, this shows that every vector of $C_+(M)$ lies in a finite-dimensional subspace invariant under all four operators.

Exponentials are taken on these finite-dimensional invariant spaces.  For a finite-dimensional operator $A$, write $A=a\id+R$ on a generalized eigenspace, with $a\in\C$ and $R^s=0$ for some $s\in\Z_{\ge1}$.  Then \(e^A=e^a\sum_{j=0}^{s-1}\frac{R^j}{j!}\). These matrix exponentials are compatible with inclusions and quotient maps of invariant spaces, so their definitions agree on overlaps.

The horizontal decomposition in Proposition~\ref{prop:MY-affine-literature-package}(i) makes $h(0)$ semisimple, and part~(v) gives $\operatorname{Spec}(h(0)|_M)\subset\epsilon+2\Z$.  Thus $e^{-\pi ih(0)}=(-1)^\epsilon\id_M$.  The integral ghost spectrum in Proposition~\ref{prop:DS-literature-package}(iv) gives $e^{2\pi i(L^{\mathrm f}_0-J^{\mathrm{gh}}_0)}=\id_{\mathcal F_{\mathrm{gh}}}$.  Applying the exponential identity for commuting matrices to \eqref{eq:BRST-Hamiltonian} therefore gives
\begin{align*}
e^{2\pi iL^{C(M)}_0}
=e^{2\pi iL^{\mathrm{aff}}_0}e^{-\pi ih(0)}
e^{2\pi i(L^{\mathrm f}_0-J^{\mathrm{gh}}_0)}=\theta^-_M\otimes\id_{\mathcal F_{\mathrm{gh}}}.
\end{align*}
The last operator is the BRST chain map induced by the affine-module automorphism $\theta^-_M$.

To pass to cohomology, let $c$ be a degree-zero cocycle and put $U_c:=\operatorname{span}_{\C}\{(L^{C(M)}_0)^nc:n\ge0\}$.  This space is finite dimensional by the joint invariance above and consists of cocycles because $[d_M,L^{C(M)}_0]=0$ by Lemma~\ref{lem:BRST-Hamiltonian}.  That lemma also gives $L^{DS}_0[c]=[L^{C(M)}_0c]$.  Hence the map $U_c\to H(M)$, $u\mapsto[u]$, intertwines the two Hamiltonians and their matrix exponentials.  Consequently \(\bigl[e^{2\pi iL^{C(M)}_0}c\bigr] =e^{2\pi iL^{DS}_0}[c]\). Taking degree-zero cohomology of the chain identity proves the asserted equality for homogeneous $M$.

For arbitrary $M$, apply this calculation to its two canonical parity summands from Proposition~\ref{prop:MY-affine-literature-package}(v), using additivity of $H$.  Finally, an affine-module map $f:M\to N$ induces $f\otimes\id_{\mathcal F_{\mathrm{gh}}}$ on the complexes.  By \eqref{eq:BRST-Hamiltonian}, this map intertwines their Hamiltonians, hence their exponentials.  Passing to cohomology proves naturality.
\end{proof}

\section{Images of projective modules}\label{sec:images}

Write $F=F_{p,q}$.  For $j\in\Z_{\ge1}$ put $Q_j:=F(P_j)\in\Oc_{c_{p,q}}$, and extend this notation by $Q_j=0$ for $j\in\Z_{\le0}$.  For $M\in\Oc_{c_{p,q}}$ and $h\in\C$, write $M^{[h]}:=M^{[h+\Z]}$ for the conformal-residue sector defined in Section~\ref{sec:prelim}.

\subsection{Small generalized Verma modules}

\begin{lem}\label{lem:F-small-generalized Verma}
For $1\leq r\leq p$, $F(V_r)\cong K_r.$
\end{lem}

\begin{proof}
Proposition~\ref{prop:MY-functor-package}(i) gives the additive strong monoidal functor $F$, with $F(V_2)\cong K_2$.  Its unit isomorphism and $K_1\cong\one_{\Oc}$ give $F(V_1)\cong K_1$.  If $p=2$, these two cases prove the assertion.

Assume $p\ge3$.  Let $2\le r\le p-1$ and suppose that $F(V_j)\cong K_j$ for $1\le j\le r$.  Proposition~\ref{prop:MY-affine-literature-package}(i) gives $P_j=V_j$ for $1\le j\le p$.  Taking $n=0$ in \eqref{eq:MY-projective-recursion-pge3} therefore gives \(V_2\boxtimes V_r\cong V_{r-1}\oplus V_{r+1}\). At $r=p-1$, the extra summand $P_{-p}$ in that recursion is zero by the standing convention.  Since $p\nmid r$, the Virasoro fusion sequence \eqref{eq:K2-fusion} splits.  Additivity and strong monoidality of $F$, together with the induction hypothesis, give
\begin{align*}
K_{r-1}\oplus F(V_{r+1})
\cong F(V_2\boxtimes V_r)
\cong K_2\boxtimes K_r\cong K_{r-1}\oplus K_{r+1}.
\end{align*}
All these objects belong to $\Oc_{c_{p,q}}$.  Krull--Schmidt uniqueness in Lemma~\ref{lem:Oc-Hom-finite-KS} allows cancellation of the common summand $K_{r-1}$, giving $F(V_{r+1})\cong K_{r+1}$.
\end{proof}

\subsection{Residue bookkeeping and transported recursions}

Write $[h]:=h+\Z\in\C/\Z$ for $h\in\C$. For $M\in\Oc_{c_{p,q}}$ and a finite subset $\Sigma\subset\C/\Z$, saying that $M$ is \emph{supported in $\Sigma$} means $M=\bigoplus_{\xi\in\Sigma}M^{[\xi]}$.

\begin{lem}\label{lem:residue-sector-peeling}
Let $A_\pm,B_\pm,X\in\Oc_{c_{p,q}}$ and let $\xi_-,\xi_+\in\C/\Z$ be distinct. Suppose $A_\pm=A_\pm^{[\xi_\pm]}$, $B_\pm=B_\pm^{[\xi_\pm]}$, and
\[
0\longrightarrow A_-\oplus A_+\longrightarrow X
\longrightarrow B_-\oplus B_+\longrightarrow0
\]
is exact. Then $X=X^{[\xi_-]}\oplus X^{[\xi_+]}$, with exact sequences
\[
0\longrightarrow A_\pm\longrightarrow X^{[\xi_\pm]}
\longrightarrow B_\pm\longrightarrow0.
\]
For $\xi\in\{\xi_-,\xi_+\}$, suppose further that $X=D\oplus D'$, $D=D^{[\xi]}$, and $\ell(D)=\ell(X^{[\xi]})$. Then $D'^{[\xi]}=0$, and the inclusion $D\hookrightarrow X$ identifies $D$ with $X^{[\xi]}$. More generally, every direct summand of an object supported in a finite set $\Sigma\subset\C/\Z$ is supported in $\Sigma$.
\end{lem}

\begin{proof}
By Lemma~\ref{lem:block-projection}, conformal-residue projection is exact. Applying it to the given sequence yields the displayed sector sequences and $X^{[\eta]}=0$ for $\eta\notin\{\xi_-,\xi_+\}$, proving the decomposition. Projection is also additive, so the chosen biproduct gives \(X^{[\xi]}=D\oplus D'^{[\xi]}\). Additivity of Jordan--H\"older length gives
$\ell(D'^{[\xi]})=\ell(X^{[\xi]})-\ell(D)=0$.
Since every object of $\Oc_{c_{p,q}}$ has finite length, $D'^{[\xi]}=0$, proving the assertion about the inclusion. Finally, if $X=D_1\oplus D_2$ and $X$ is supported in $\Sigma$, then for $\eta\notin\Sigma$ one has $0=X^{[\eta]}=D_1^{[\eta]}\oplus D_2^{[\eta]}$. Hence both summands are supported in $\Sigma$.
\end{proof}

\begin{lem}\label{lem:strip-residue-pairing}
Let $n\ge1$ and $1\le r\le p-1$, and write $a=a_{n,r}$ and $b=b_{n,r}$ as in \eqref{eq:abc}. Then $[h_a]=[h_b]$. If $1\le r\le p-2$, then \([h_{a+1}]=[h_{b-1}]\) and \([h_{a-1}]=[h_{b+1}]\), and these two residue classes are distinct. At the far wall, \([h_{(n-1)p}]=[h_{(n+1)p}]\) and \([h_{(n-1)p+2}]=[h_{(n+1)p-2}]\), and these two residue classes are distinct. For $p=2$, this reads $[h_{2n}]\ne[h_{2n-2}]=[h_{2n+2}]$.
\end{lem}

\begin{proof}
Subtracting the weights in \eqref{eq:weights} gives \(h_u-h_v=\frac{(u-v)\bigl(q(u+v)-2p\bigr)}{4p} \;(u,v\in\Z)\). With $a=np-r$ and $b=np+r$, this yields
\begin{align*}
h_b-h_a&=r(nq-1),\\
h_{b-1}-h_{a+1}&=(r-1)(nq-1),\\
h_{b+1}-h_{a-1}&=(r+1)(nq-1).
\end{align*}
These integer differences give the asserted congruences, including the far-wall pairings on taking $r=p-1$. To separate the paired residue classes, compute
\begin{align*}
h_{b-1}-h_{a-1}&=r(nq-1)-\frac{rq}{p},\\
h_{(n+1)p-2}-h_{(n-1)p}&=(p-1)(nq-1)-\frac{(p-1)q}{p}.
\end{align*}
The first difference is nonintegral for $1\le r\le p-2$, and the second is nonintegral for every $p\ge2$, because $\gcd(p,q)=1$. For $p=2$, the far-wall congruences give $[h_{2n-2}]=[h_{2n+2}]$, while $h_{2n}-h_{2n-2}=nq-1-q/2$ is nonintegral since $q$ is odd.
\end{proof}

For $n\ge1$, apply the additive strong monoidal functor $F$ of Proposition~\ref{prop:MY-functor-package}(i) to the direct-sum tensor decompositions in Proposition~\ref{prop:MY-affine-literature-package}(iii). Using $F(V_2)\cong K_2$, one obtains the following isomorphisms. If $p\ge3$,
\begin{align}
K_2\boxtimes Q_{np}&\cong Q_{np+1},\label{eq:transport-wall-first}\\
K_2\boxtimes Q_{np+1}&\cong Q_{np}^{\oplus2}\oplus Q_{np+2},\notag\\
K_2\boxtimes Q_{np+r}&\cong Q_{np+r-1}\oplus Q_{np+r+1},
&&2\le r\le p-2,\label{eq:transport-interior}\\
K_2\boxtimes Q_{(n+1)p-1}&\cong Q_{(n-1)p}\oplus Q_{(n+1)p-2}\oplus Q_{(n+1)p}. \label{eq:transport-far-wall}
\end{align}
If $p=2$,
\begin{align}
K_2\boxtimes Q_{2n}&\cong Q_{2n+1},\label{eq:transport-p2-even}\\
K_2\boxtimes Q_{2n+1}&\cong Q_{2n-2}\oplus Q_{2n}^{\oplus2}\oplus Q_{2n+2}. \label{eq:transport-p2-odd}
\end{align}
The terms with index zero vanish by the convention $Q_0=0$.

\subsection{Projective reconstruction}

\begin{prop}\label{prop:projective-images}
For $n\ge1$ and $1\le r\le p-1$, there is an exact sequence in $\Oc_{c_{p,q}}$
\begin{equation}\label{eq:F-projective-filtration}
0\longrightarrow K_{np-r}\longrightarrow Q_{np+r}
\longrightarrow K_{np+r}\longrightarrow0.
\end{equation}
With $\xi_{n,r}:=[h_{np-r}]=[h_{np+r}]$, its middle term satisfies $Q_{np+r}=Q_{np+r}^{[\xi_{n,r}]}$. Moreover, $Q_{np}\cong K_{np}$ for every $n\ge1$. Consequently, in $K_0(\Oc_{c_{p,q}})$,
\([Q_{np+r}]=[S_{np-r}]+2[S_{np+r}]+[S_{(n+2)p-r}]\).
\end{prop}

\begin{proof}
Proposition~\ref{prop:MY-affine-literature-package}(i) and Lemma~\ref{lem:F-small-generalized Verma} give $Q_p\cong K_p$, initializing an induction on $n$. Assume that $Q_{mp}\cong K_{mp}$ for $1\le m\le n$, and that the claimed exact sequences and residue support hold for $1\le m<n$ and $1\le r<p$. We construct them for the indices $np+1,\ldots,(n+1)p-1$ and then identify $Q_{(n+1)p}$.

For the first index, by \eqref{eq:transport-wall-first} when $p\ge3$ and \eqref{eq:transport-p2-even} when $p=2$, \(Q_{np+1}\cong K_2\boxtimes Q_{np}\cong K_2\boxtimes K_{np}\). Transporting \eqref{eq:K2-fusion} through this isomorphism gives
\[
0\longrightarrow K_{np-1}\longrightarrow Q_{np+1}
\longrightarrow K_{np+1}\longrightarrow0.
\]
The end terms have the same conformal residue by Lemma~\ref{lem:strip-residue-pairing}. Exact residue projection from Lemma~\ref{lem:block-projection} therefore gives $Q_{np+1}=Q_{np+1}^{[h_{np+1}]}$.

For the remaining indices when $p\ge3$, suppose the assertions have been established for $Q_{np+t}$ with $1\le t\le r$, where $1\le r\le p-2$. Put $a=np-r$, $b=np+r$, and $M:=K_2\boxtimes Q_b$. The known sequence is $0\to K_a\to Q_b\to K_b\to0$. By Proposition~\ref{prop:Virasoro-literature-package}(iii)--(iv), tensoring with $K_2$ is exact and the fusion sequences for $K_a$ and $K_b$ split, since $p\nmid a$ and $p\nmid b$. Thus
\[
0\longrightarrow K_{a-1}\oplus K_{a+1}\longrightarrow M
\longrightarrow K_{b-1}\oplus K_{b+1}\longrightarrow0.
\]
Since $K_j=K_j^{[h_j]}$, Lemma~\ref{lem:strip-residue-pairing} gives the distinct residues
$\xi_-=[h_{a+1}]=[h_{b-1}]$ and $\xi_+=[h_{a-1}]=[h_{b+1}]$.
Lemma~\ref{lem:residue-sector-peeling} gives $M=M^{[\xi_-]}\oplus M^{[\xi_+]}$ and
\begin{align*}
0&\longrightarrow K_{a+1}\longrightarrow M^{[\xi_-]}
\longrightarrow K_{b-1}\longrightarrow0,\\
0&\longrightarrow K_{a-1}\longrightarrow M^{[\xi_+]}
\longrightarrow K_{b+1}\longrightarrow0.
\end{align*}
The transported recursions give $M\cong D\oplus Q_{b+1}$, where
\[
D:=\begin{cases}
K_{np}^{\oplus2},&r=1,\\
Q_{b-1},&2\le r\le p-2.
\end{cases}
\]
Here $D=D^{[\xi_-]}$. For $r\ge2$, the inner induction gives $0\to K_{a+1}\to D\to K_{b-1}\to0$; for $r=1$, both these end terms are $K_{np}$. Hence
$\ell(D)=\ell(K_{a+1})+\ell(K_{b-1})=\ell(M^{[\xi_-]})$.
Lemma~\ref{lem:residue-sector-peeling} identifies the chosen summand $D$ with $M^{[\xi_-]}$. Since $M$ is supported in $\{\xi_-,\xi_+\}$, it also gives
$Q_{b+1}=Q_{b+1}^{[\xi_+]}$ and $M^{[\xi_+]}\cong Q_{b+1}$.
The second sector sequence therefore becomes $0\to K_{a-1}\to Q_{b+1}\to K_{b+1}\to0$, completing the inner induction.

For the next wall when $p\ge3$, put $u=(n-1)p$, $v=(n+1)p$, and $M:=K_2\boxtimes Q_{v-1}$. The completed inner induction gives
$0\to K_{u+1}\to Q_{v-1}\to K_{v-1}\to0$.
Exact tensoring with $K_2$, the split fusion sequences for its end terms, and Lemma~\ref{lem:strip-residue-pairing} give the distinct residues
$\xi_0=[h_u]=[h_v]$ and $\xi_1=[h_{u+2}]=[h_{v-2}]$, with
\begin{align*}
0&\longrightarrow K_u\longrightarrow M^{[\xi_0]}
\longrightarrow K_v\longrightarrow0,\\
0&\longrightarrow K_{u+2}\longrightarrow M^{[\xi_1]}
\longrightarrow K_{v-2}\longrightarrow0.
\end{align*}
Moreover, $M$ is supported in $\{\xi_0,\xi_1\}$.
By \eqref{eq:transport-far-wall} and the outer induction, \(M\cong K_u\oplus Q_{v-2}\oplus Q_v\). The inner induction places $Q_{v-2}$ in $\xi_1$ and gives
$\ell(Q_{v-2})=\ell(K_{u+2})+\ell(K_{v-2})=\ell(M^{[\xi_1]})$.
Lemma~\ref{lem:residue-sector-peeling} therefore identifies the chosen summand $Q_{v-2}$ with $M^{[\xi_1]}$. As $K_u$ is supported in $\xi_0$, it follows that $Q_v=Q_v^{[\xi_0]}$ and
$M^{[\xi_0]}\cong K_u\oplus Q_v$.
Comparison with the first sector sequence gives $\ell(Q_v)=\ell(K_v)=1$, since $K_v$ is simple by Proposition~\ref{prop:Virasoro-literature-package}(ii). Additivity of composition-factor multiplicities identifies its unique factor with $K_v$, so $Q_v\cong K_v$. The zero conventions $K_0=Q_0=0$ cover $n=1$ in the same argument.

For the next wall when $p=2$, the first-index step already gives $0\to K_{2n-1}\to Q_{2n+1}\to K_{2n+1}\to0$. Put $u=2n-2$, $v=2n+2$, and $M:=K_2\boxtimes Q_{2n+1}$. Exact tensoring and the split fusion sequences at the odd labels give
\[
0\longrightarrow K_u\oplus K_{2n}\longrightarrow M
\longrightarrow K_{2n}\oplus K_v\longrightarrow0.
\]
By Lemma~\ref{lem:strip-residue-pairing}, $\xi_0=[h_u]=[h_v]$ and $\xi_1=[h_{2n}]$ are distinct. Exact projection gives
\begin{align*}
0&\longrightarrow K_u\longrightarrow M^{[\xi_0]}
\longrightarrow K_v\longrightarrow0,\\
0&\longrightarrow K_{2n}\longrightarrow M^{[\xi_1]}
\longrightarrow K_{2n}\longrightarrow0,
\end{align*}
and $M$ is supported in $\{\xi_0,\xi_1\}$. Meanwhile, \eqref{eq:transport-p2-odd} and the induction hypothesis give
$M\cong K_u\oplus K_{2n}^{\oplus2}\oplus Q_v$.
The summand $K_{2n}^{\oplus2}$ is supported in $\xi_1$ and has length $\ell(M^{[\xi_1]})$. Lemma~\ref{lem:residue-sector-peeling} identifies it with that sector. Hence $Q_v=Q_v^{[\xi_0]}$ and $M^{[\xi_0]}\cong K_u\oplus Q_v$. The same length and composition-factor comparison as above gives $Q_v\cong K_v$, also for $n=1$ by $K_0=Q_0=0$.

Finally, Proposition~\ref{prop:Virasoro-literature-package}(ii) gives
$[K_{np-r}]=[S_{np-r}]+[S_{np+r}]$ and
$[K_{np+r}]=[S_{np+r}]+[S_{(n+2)p-r}]$.
Adding these classes in \eqref{eq:F-projective-filtration} proves the stated composition-factor formula.
\end{proof}

\subsection{Faithfulness on projectives}

\begin{prop}\label{prop:F-faithful-projectives}
For every $p\ge2$, the restriction $F|_{\mathcal P^k}$ is faithful.
\end{prop}

\begin{proof}
Suppose first that $p\ge3$. Use the notation of Subsection~\ref{subsec:TL-prelim}, and set \(G:=F\circ\mathscr E_\zeta:\mathcal T_\zeta\longrightarrow\Oc_{c_{p,q}}\). Since $\mathscr E_\zeta$ is an equivalence by Proposition~\ref{prop:tilting-literature-package}(i), it suffices to prove that $G$ is faithful. Suppose otherwise. By additivity and the finite indecomposable decompositions in $\mathcal T_\zeta$, there are indecomposable tiltings $T,U$ and a nonzero morphism $f:T\to U$ with $G(f)=0$. By Proposition~\ref{prop:tilting-literature-package}(i), choose retract data
\[
T\xrightarrow{j_T}\mathbf X_\zeta^{\otimes m}\xrightarrow{r_T}T,
\qquad U\xrightarrow{j_U}\mathbf X_\zeta^{\otimes n}\xrightarrow{r_U}U,
\qquad r_Tj_T=\id_T,\quad r_Uj_U=\id_U.
\]
The morphism $g:=j_Ufr_T$ is nonzero because $r_Ugj_T=f$, whereas $G(g)=0$. For $m,n\ge0$, define
\[
J(m,n):=\{a\in\Hom_{\TL_{\delta_{\mathrm{TL}}}}(m,n):
(G\circ\mathscr D_\zeta^{\mathrm{TL}})(a)=0\}.
\]
The $\C$-linearity, functoriality and strong monoidality of $G\circ\mathscr D_\zeta^{\mathrm{TL}}$ make $J$ a tensor ideal in the sense of Subsection~\ref{subsec:TL-prelim}. Fullness in Proposition~\ref{prop:tilting-literature-package}(ii) lifts $g$ to a nonzero element of $J(m,n)$, so $J\ne0$. The ideal is proper: the image of $\id_0$ is the identity of $G(T_0)\cong F(P_1)\cong K_1\ne0$, using \eqref{eq:tilting-projective-equivalence} and Lemma~\ref{lem:F-small-generalized Verma}.

Since $\operatorname{ord}_{\times}(\tau_{\mathrm{TL}}^2)=p\ge3$, Lemma~\ref{lem:TL-ideal-detection} gives \(JW_{p-1}(\tau_{\mathrm{TL}})\in J(p-1,p-1)\). By Lemma~\ref{lem:first-critical-tilting-projector}, its image \(e:=\mathscr D_\zeta^{\mathrm{TL}}(JW_{p-1}(\tau_{\mathrm{TL}}))\) has splitting maps
\[
T_{p-1}\xrightarrow{i}\mathbf X_\zeta^{\otimes(p-1)}
\xrightarrow{\pi}T_{p-1},
\qquad \pi i=\id_{T_{p-1}},\quad i\pi=e.
\]
Now $G(e)=0$, so $G(i)=G(e)G(i)=0$ and \(\id_{G(T_{p-1})}=G(\pi)G(i)=0\). This contradicts \(G(T_{p-1})\cong F(P_p)=Q_p\cong K_p\ne0\), by \eqref{eq:tilting-projective-equivalence} and Proposition~\ref{prop:projective-images}.

Now suppose that $p=2$. The non-wall reflection chain is $s_i=2i+1$, $i\ge0$. For $m\ge1$, put $P=P_{2m+1}$ and $Q=F(P)=Q_{2m+1}$. Equation~\eqref{eq:transport-p2-even} and Proposition~\ref{prop:projective-images} give \(Q\cong K_2\boxtimes Q_{2m}\cong K_2\boxtimes K_{2m}\). This module is logarithmic by Proposition~\ref{prop:Virasoro-literature-package}(iv). Hence its Virasoro twist $\theta^{\mathrm{Vir}}_Q=e^{2\pi iL_0}|_Q$ is not diagonalizable by Lemma~\ref{lem:exponential-logarithmic}. Twist preservation in Proposition~\ref{prop:MY-functor-package}(i), together with \eqref{eq:source-twist-nilpotent}, gives
\[
\theta^{\mathrm{Vir}}_Q=F(\theta^-_P)
=\lambda_P\bigl(\id_Q+\gamma_PF(N_P)\bigr),
\qquad \lambda_P,\gamma_P\in\C^\times.
\]
If $F(N_P)=0$, this twist would be scalar. Thus $F(N_{s_i})\ne0$ for every $i\ge1$; moreover, $F(N_{s_i})^2=0$ because $N_{s_i}^2=0$.

Choose adjacent arrows $x_i:P_{s_i}\to P_{s_{i+1}}$ and $y_i:P_{s_{i+1}}\to P_{s_i}$ as in Corollary~\ref{cor:source-Hom-package}. The composite $\nu_{i+1}:=x_i y_i$ is a nonzero scalar multiple of $N_{s_{i+1}}$, so \(0\ne F(\nu_{i+1})=F(x_i)F(y_i)\). Both arrows therefore have nonzero images, proving injectivity on each one-dimensional adjacent Hom space. For every $i\ge1$, the images of the basis $\id_{P_{s_i}},N_{s_i}$ of $\End(P_{s_i})$ are linearly independent: a nonzero square-zero endomorphism $F(N_{s_i})$ cannot be a scalar multiple of the identity. At $i=0$, $\End(P_1)=\C\id_{P_1}$ and $F(P_1)\cong K_1\ne0$ by Lemma~\ref{lem:F-small-generalized Verma}. Each wall projective $P_{2n}$ has endomorphism space $\C\id_{P_{2n}}$ and nonzero image $Q_{2n}\cong K_{2n}$ by Proposition~\ref{prop:projective-images}. Corollary~\ref{cor:source-Hom-package} lists all the remaining Hom spaces between distinct indecomposable projectives as zero. Thus $F$ is injective on every Hom space between the $P_j$. By Proposition~\ref{prop:tilting-literature-package}(i), every object of $\mathcal P^k$ is a finite direct sum of these indecomposables. Writing morphisms between such sums as matrices proves faithfulness on $\mathcal P^k$.
\end{proof}

\begin{prop}\label{prop:Q-logarithmic}
For $n\ge1$ and $1\le r\le p-1$, the module $Q_{np+r}$ is logarithmic. Consequently the sequence \eqref{eq:F-projective-filtration} is nonsplit.
\end{prop}

\begin{proof}
Put $P=P_{np+r}$ and $Q=F(P)=Q_{np+r}$. For the fixed radical generator $N_P$ of Subsection~\ref{subsec:affine-projectives}, set $T:=F(N_P)$. Proposition~\ref{prop:F-faithful-projectives} gives $T\ne0$, while $N_P^2=0$ gives $T^2=0$. By twist preservation in Proposition~\ref{prop:MY-functor-package}(i) and \eqref{eq:source-twist-nilpotent},
\[
\theta^{\mathrm{Vir}}_Q=e^{2\pi iL_0}|_Q
=\lambda_P\bigl(\id_Q+\gamma_PT\bigr),
\qquad \lambda_P,\gamma_P\in\C^\times.
\]
Thus $(\theta^{\mathrm{Vir}}_Q-\lambda_P\id_Q)^2=0$ but $\theta^{\mathrm{Vir}}_Q-\lambda_P\id_Q\ne0$, so $\theta^{\mathrm{Vir}}_Q$ is not diagonalizable. Lemma~\ref{lem:exponential-logarithmic} implies that $Q$ is logarithmic. Both end terms $K_{np-r}$ and $K_{np+r}$ of \eqref{eq:F-projective-filtration} are ordinary by Proposition~\ref{prop:Virasoro-literature-package}(i). If the sequence split, $Q$ would be their direct sum and hence ordinary, a contradiction.
\end{proof}

\section{Logarithmic images of affine projectives and full faithfulness}\label{sec:fullness}

Section~\ref{sec:images} establishes the Kac filtrations, logarithmicity and projective faithfulness. We now identify the modules $Q_j$ and compute morphisms between them, using the one-row extension interface of Proposition~\ref{prop:one-row-staggered-quotient}. The notation $K(\tau)$ and $P(\tau)$ is that fixed in Subsection~\ref{subsec:Nakano-extension-interface}.

Throughout this section, unqualified $\Hom$ and Yoneda $\Ext^1$ spaces are computed in $\Oc_{c_{p,q}}$. For endpoints in $\mathscr C$, we use the canonical identification with $\Ext^1_{\mathscr C}$ from Lemma~\ref{lem:Nakano-exact-compatibility}.

\begin{thm}\label{thm:projective-classification}
For every indecomposable projective $P_j$, $j\ge1$, \[F(P_j)\cong
\begin{cases}
K_j,&1\le j\le p-1,\\
K_j=S_j,&p\mid j,\\
P(\tau_{n,r}),&j=np+r,\ n\ge1,\ 1\le r\le p-1.
\end{cases}\]
Moreover, as $\C$-algebras, \[\End_{\Oc_{c_{p,q}}} F(P_j)\cong
\begin{cases}
\C,&1\le j\le p-1\text{ or }p\mid j,\\
\C[\varepsilon]/(\varepsilon^2), &j=np+r,\ n\ge1,\ 1\le r\le p-1,
\end{cases}\]
where $\varepsilon$ is an indeterminate. The modules in the third case are logarithmic and indecomposable, and $F|_{\mathcal P^k}$ is fully faithful.
\end{thm}

The theorem is proved at the end of the section. After identifying the logarithmic extensions and their distinguished quotients, we use the Hom--Yoneda exact sequences of Lemma~\ref{lem:prelim-Hom-Yoneda} to obtain target Hom upper bounds. Projective faithfulness supplies the matching lower bounds.

\subsection{Recognition of logarithmic extensions}

\begin{lem}\label{lem:Nakano-recognition}
Let $n\ge1$ and $1\le r\le p-1$, and use $a=a_{n,r}$ and $b=b_{n,r}$ from \eqref{eq:abc}. Suppose that $E\in\Oc_{c_{p,q}}$ is logarithmic and fits into a short exact sequence \(0\longrightarrow K_a\xrightarrow{\iota}E\longrightarrow K_b\longrightarrow0\). Identify $K_a$ with its image under $\iota$, and regard the fixed submodule $S_b\hookrightarrow K_a$ of Proposition~\ref{prop:Virasoro-literature-package}(ii) as a submodule of $E$. Then the induced sequence \(0\longrightarrow S_a\longrightarrow E/S_b\longrightarrow K_b\longrightarrow0\) is nonsplit, and $E/S_b\cong K(\tau_{n,r})$ and $E\cong P(\tau_{n,r})$.
\end{lem}

\begin{proof}
Put $\tau:=\tau_{n,r}$ and $D:=E/S_b$. Proposition~\ref{prop:Felder-labels} gives $E,D\in\mathscr C$ and the endpoint identifications
\(
K_b\cong B(\tau), S_a\cong\mathsf L_1(\tau),
S_b\cong\mathsf L_2(\tau).
\)
By Proposition~\ref{prop:one-row-staggered-quotient}(i), the induced sequence $0\to S_a\to D\to K_b\to0$ is nonsplit. Part~(ii) of the same proposition gives $\Ext^1_{\mathscr C}(K_b,S_a)\cong\C$. Under the endpoint identifications, this sequence and the defining sequence \eqref{eq:Nakano-tower-K} for $K(\tau)$ therefore represent nonzero classes in the same one-dimensional Ext space. Lemma~\ref{lem:one-dimensional-ext} gives an isomorphism $\phi:D\xrightarrow{\sim}K(\tau)$.

Let $\rho:E\twoheadrightarrow D$ be the quotient map. Composing with $\phi$ gives a short exact sequence in $\mathscr C$,
\[
0\longrightarrow S_b\longrightarrow E
\xrightarrow{\ \phi\circ\rho\ }K(\tau)\longrightarrow0.
\]
This sequence is nonsplit: a retraction $E\to S_b$ would restrict to a retraction of $S_b\hookrightarrow K_a$, contradicting the nonsplit sequence $0\to S_b\to K_a\to S_a\to0$ of Proposition~\ref{prop:Virasoro-literature-package}(ii). Its class is thus nonzero in $\Ext^1_{\mathscr C}(K(\tau),S_b)\cong\C$, by Proposition~\ref{prop:one-row-staggered-quotient}(ii). Using $S_b\cong\mathsf L_2(\tau)$, the definition of $P(\tau)$ in Subsection~\ref{subsec:Nakano-extension-interface} and Lemma~\ref{lem:one-dimensional-ext} give $E\cong P(\tau)$.
\end{proof}

\begin{cor}\label{cor:Q-Ptau}
For $n\ge1$ and $1\le r\le p-1$, write $a=a_{n,r}$ and $b=b_{n,r}$ as in \eqref{eq:abc}. Regard $S_b$ as a submodule of $Q_b$ via the inclusion $K_a\hookrightarrow Q_b$ in \eqref{eq:F-projective-filtration} and the fixed inclusion $S_b\hookrightarrow K_a$. Then $Q_b\cong P(\tau_{n,r})$, $Q_b/S_b\cong K(\tau_{n,r})$, and the induced sequence $0\longrightarrow S_a\longrightarrow Q_b/S_b\longrightarrow K_b\longrightarrow0$ is nonsplit.
\end{cor}

\begin{proof}
The filtration \eqref{eq:F-projective-filtration} and the logarithmicity in Proposition~\ref{prop:Q-logarithmic} verify the hypotheses of Lemma~\ref{lem:Nakano-recognition} with $E=Q_b$.
\end{proof}

\subsection{Hom spaces along a reflection chain}

Fix $r\in\{1,\ldots,p-1\}$ and the reflection chain $s_i=s_i(r)$, $i\ge0$, of \eqref{eq:reflection-chain}. Write
\(
\mathsf S_i:=S_{s_i}, \mathsf K_i:=K_{s_i},
\mathsf Q_i:=Q_{s_i}=F(P_{s_i}).
\)
The labels on this chain are distinct, with $s_0=r<p$ and $s_i>p$ for $i\ge1$. Thus Corollary~\ref{cor:positive-label-collisions} implies that the simples $\mathsf S_i$ are pairwise nonisomorphic.

Write $s_i=n_ip+r_i$ with $n_i\ge0$ and $1\le r_i\le p-1$. Equation~\eqref{eq:reflection-chain} gives $s_{i+1}=(n_i+2)p-r_i$ and, for $i\ge1$, $s_{i-1}=n_ip-r_i$. Proposition~\ref{prop:Virasoro-literature-package}(ii) therefore gives the fixed nonsplit sequences
\begin{equation}\label{eq:A-chain}
0\longrightarrow\mathsf S_{i+1}\longrightarrow\mathsf K_i
\longrightarrow\mathsf S_i\longrightarrow0\qquad(i\ge0).
\end{equation}
For $i\ge1$, one has $n_i\ge1$, so Proposition~\ref{prop:projective-images} and Proposition~\ref{prop:Q-logarithmic} give nonsplit sequences
\begin{equation}\label{eq:Q-chain}
0\longrightarrow\mathsf K_{i-1}\longrightarrow\mathsf Q_i
\longrightarrow\mathsf K_i\longrightarrow0.
\end{equation}
At $i=0$, Proposition~\ref{prop:MY-affine-literature-package}(i) and Lemma~\ref{lem:F-small-generalized Verma} identify $\mathsf Q_0$ with $\mathsf K_0=K_r$.

For $i\ge1$, let $\sigma_i:\mathsf S_i\hookrightarrow\mathsf Q_i$ be the composite of the fixed inclusions $\mathsf S_i\hookrightarrow\mathsf K_{i-1}$ in \eqref{eq:A-chain} and $\mathsf K_{i-1}\hookrightarrow\mathsf Q_i$ in \eqref{eq:Q-chain}. Corollary~\ref{cor:Q-Ptau}, with $(n,r)=(n_i,r_i)$, says that the induced sequence
\[
0\longrightarrow\mathsf S_{i-1}\longrightarrow
\mathsf Q_i/\sigma_i(\mathsf S_i)\longrightarrow\mathsf K_i\longrightarrow0
\]
is nonsplit. The filtrations above also give, in $K_0(\Oc_{c_{p,q}})$,
\[
[\mathsf Q_0]=[\mathsf S_0]+[\mathsf S_1],\qquad
[\mathsf Q_i]=[\mathsf S_{i-1}]+2[\mathsf S_i]+[\mathsf S_{i+1}]
\quad(i\ge1).
\]

\begin{lem}\label{lem:disjoint-constituents}
If two objects of $\Oc_{c_{p,q}}$ have no common simple Jordan--H\"older constituent, then there is no nonzero morphism between them in either direction.
\end{lem}

\begin{proof}
The image of a nonzero morphism is a nonzero finite-length module and therefore has a simple constituent common to the source and the target.
\end{proof}

\begin{lem}\label{lem:length-two-Hom}
Let $0\longrightarrow T_0\longrightarrow A\xrightarrow{\pi}T\longrightarrow0$ be nonsplit in $\Oc_{c_{p,q}}$, where $T,T_0$ are nonisomorphic simple objects. Then
\(\End(A)=\C\id_A,\; \Hom(A,T_0)=0,\; \Hom(A,T)=\C\pi\).
\end{lem}

\begin{proof}
The simple highest-weight modules satisfy $\End(T)=\C\id_T$, $\End(T_0)=\C\id_{T_0}$ and $\Hom(T,T_0)=\Hom(T_0,T)=0$. By Lemma~\ref{lem:prelim-Hom-Yoneda}, applying $\Hom(-,T_0)$ gives
\[
0\longrightarrow\Hom(A,T_0)\longrightarrow\End(T_0)
\xrightarrow{\partial}\Ext^1(T,T_0).
\]
The connecting map sends $\id_{T_0}$ to the nonzero class of the given extension. Its domain is one-dimensional, so it is injective and $\Hom(A,T_0)=0$. Applying $\Hom(-,T)$ gives
\(0\to\End(T)\to\Hom(A,T)\to\Hom(T_0,T)=0\), hence $\Hom(A,T)=\C\pi$. Finally, the left-exact sequence \(0\to\Hom(A,T_0)\to\End(A)\to\Hom(A,T)\) embeds $\End(A)$ into a one-dimensional space. Since $\id_A\ne0$, it spans $\End(A)$.
\end{proof}

\begin{lem}\label{lem:Kac-Kac-Hom}
For $i,j\ge0$, the space $\Hom(\mathsf K_i,\mathsf S_j)$ is one-dimensional if $j=i$ and zero otherwise. Moreover,
\[
\dim_\C\Hom(\mathsf K_i,\mathsf K_j)=
\begin{cases}
1,&j=i,\\
1,&j=i-1\text{ and }i\ge1,\\
0,&\text{otherwise}.
\end{cases}
\]
\end{lem}

\begin{proof}
The simples on the chain are pairwise nonisomorphic. Lemma~\ref{lem:length-two-Hom} applied to \eqref{eq:A-chain} gives $\dim\Hom(\mathsf K_i,\mathsf S_i)=1$ and $\Hom(\mathsf K_i,\mathsf S_{i+1})=0$. For $j\notin\{i,i+1\}$, Lemma~\ref{lem:disjoint-constituents} gives $\Hom(\mathsf K_i,\mathsf S_j)=0$. This proves the first assertion.

Applying $\Hom(\mathsf K_i,-)$ to \eqref{eq:A-chain} for $\mathsf K_j$ gives
\[
0\longrightarrow\Hom(\mathsf K_i,\mathsf S_{j+1})
\longrightarrow\Hom(\mathsf K_i,\mathsf K_j)
\longrightarrow\Hom(\mathsf K_i,\mathsf S_j).
\]
For $j=i-1$, the left term is one-dimensional and the right term is zero; the nonzero maps are scalar multiples of $\mathsf K_i\twoheadrightarrow\mathsf S_i\hookrightarrow\mathsf K_{i-1}$. For $j=i$, Lemma~\ref{lem:length-two-Hom} gives $\End(\mathsf K_i)=\C\id_{\mathsf K_i}$. For every other $j$, both outer terms vanish.
\end{proof}

\begin{lem}\label{lem:Q-simple-Hom}
For $i\ge1$ and $j\ge0$,
\[
\dim_\C\Hom(\mathsf Q_i,\mathsf S_j)=
\begin{cases}
1,&j=i,\\
0,&j\ne i.
\end{cases}
\]
\end{lem}

\begin{proof}
Applying $\Hom(-,\mathsf S_j)$ to \eqref{eq:Q-chain}, using Lemma~\ref{lem:prelim-Hom-Yoneda}, gives
\[
0\longrightarrow\Hom(\mathsf K_i,\mathsf S_j)
\longrightarrow\Hom(\mathsf Q_i,\mathsf S_j)
\longrightarrow\Hom(\mathsf K_{i-1},\mathsf S_j)
\xrightarrow{\partial_j}\Ext^1(\mathsf K_i,\mathsf S_j).
\]
By Lemma~\ref{lem:Kac-Kac-Hom}, the Hom space is zero for $j\notin\{i-1,i\}$ and one-dimensional for $j=i$, generated by $\mathsf Q_i\twoheadrightarrow\mathsf K_i\twoheadrightarrow\mathsf S_i$.

For $j=i-1$, let $\pi_{i-1}:\mathsf K_{i-1}\twoheadrightarrow\mathsf S_{i-1}$ be the head projection in \eqref{eq:A-chain}. It spans $\Hom(\mathsf K_{i-1},\mathsf S_{i-1})$. The class $\partial_{i-1}(\pi_{i-1})$ is represented by the pushout of \eqref{eq:Q-chain} along $\pi_{i-1}$. Since $\ker\pi_{i-1}=\mathsf S_i$, this pushout has middle term $\mathsf Q_i/\sigma_i(\mathsf S_i)$: the quotient maps form a pushout square by the universal property of the cokernel. The induced sequence is nonsplit by Corollary~\ref{cor:Q-Ptau}, as noted above. Thus $\partial_{i-1}(\pi_{i-1})\ne0$, and $\partial_{i-1}$ is injective because its domain is one-dimensional. Exactness gives $\Hom(\mathsf Q_i,\mathsf S_{i-1})=0$.
\end{proof}

\begin{lem}\label{lem:Q-Kac-Hom}
For $i\ge1$ and $j\ge0$,
\[
\dim_\C\Hom(\mathsf Q_i,\mathsf K_j)=
\begin{cases}
1,&j=i-1\text{ or }j=i,\\
0,&\text{otherwise}.
\end{cases}
\]
\end{lem}

\begin{proof}
Applying $\Hom(\mathsf Q_i,-)$ to \eqref{eq:A-chain} gives
\[
0\longrightarrow\Hom(\mathsf Q_i,\mathsf S_{j+1})
\longrightarrow\Hom(\mathsf Q_i,\mathsf K_j)
\longrightarrow\Hom(\mathsf Q_i,\mathsf S_j).
\]
By Lemma~\ref{lem:Q-simple-Hom}, both outer terms vanish unless $j=i-1$ or $j=i$. For $j=i-1$, the left term is one-dimensional and the right term is zero, proving the assertion. For $j=i$, the left term is zero and the right term is one-dimensional, giving an upper bound of one. The epimorphism $\mathsf Q_i\twoheadrightarrow\mathsf K_i$ in \eqref{eq:Q-chain} is nonzero, so this bound is attained.
\end{proof}

\begin{prop}\label{prop:target-Hom-table}
For $i,j\ge0$ on the fixed reflection chain,
\begin{equation}\label{eq:target-Hom-table}
\dim_\C\Hom(\mathsf Q_i,\mathsf Q_j)=
\begin{cases}
2,&i=j\ge1,\\
1,&i=j=0,\\
1,&|i-j|=1,\\
0,&|i-j|\ge2.
\end{cases}
\end{equation}
\end{prop}

\begin{proof}
Let $d_{ij}$ denote the right-hand side of \eqref{eq:target-Hom-table}. For $i,j\ge1$, applying $\Hom(\mathsf Q_i,-)$ to \eqref{eq:Q-chain} for $\mathsf Q_j$ gives
\[
0\longrightarrow\Hom(\mathsf Q_i,\mathsf K_{j-1})
\longrightarrow\Hom(\mathsf Q_i,\mathsf Q_j)
\longrightarrow\Hom(\mathsf Q_i,\mathsf K_j).
\]
Hence
\[
\dim\Hom(\mathsf Q_i,\mathsf Q_j)
\le\dim\Hom(\mathsf Q_i,\mathsf K_{j-1})
+\dim\Hom(\mathsf Q_i,\mathsf K_j)=d_{ij},
\]
where the last equality follows from Lemma~\ref{lem:Q-Kac-Hom}.

At the boundary, $\mathsf Q_0=\mathsf K_0$ and Lemma~\ref{lem:Kac-Kac-Hom} gives $\dim\End(\mathsf Q_0)=1$. For $j\ge1$, applying $\Hom(\mathsf K_0,-)$ to \eqref{eq:Q-chain} gives \(\Hom(\mathsf Q_0,\mathsf Q_j)\cong\Hom(\mathsf K_0,\mathsf K_{j-1})\), since $\Hom(\mathsf K_0,\mathsf K_j)=0$. Thus this space is one-dimensional for $j=1$ and zero for $j\ge2$. Lemma~\ref{lem:Q-Kac-Hom} gives the same dimensions for $\Hom(\mathsf Q_j,\mathsf Q_0)$. This proves all the upper bounds.

Finally, Corollary~\ref{cor:source-Hom-package} and the projective faithfulness of Proposition~\ref{prop:F-faithful-projectives} give
\[
d_{ij}=\dim_\C\Hom_{\KL^k}(P_{s_i},P_{s_j})
\le\dim_\C\Hom(\mathsf Q_i,\mathsf Q_j)\le d_{ij}.
\]
Equality holds throughout.
\end{proof}

\begin{cor}\label{cor:target-local-endomorphisms}
For $n\ge1$ and $1\le r\le p-1$,
\[
\End(Q_{np+r})=\C\id_{Q_{np+r}}\oplus\C F(N_{np+r})
\cong\C[\varepsilon]/(\varepsilon^2).
\]
In particular, $Q_{np+r}$ is indecomposable.
\end{cor}

\begin{proof}
By Lemma~\ref{lem:reflection-partition}, the label $np+r$ has chain index at least one. Equation~\eqref{eq:target-Hom-table} therefore gives $\dim\End(Q_{np+r})=2$. Proposition~\ref{prop:F-faithful-projectives} gives $F(N_{np+r})\ne0$, and $N_{np+r}^2=0$ gives $F(N_{np+r})^2=0$. Thus $\id_{Q_{np+r}}$ and $F(N_{np+r})$ are linearly independent and form a basis. The assignment $\varepsilon\mapsto F(N_{np+r})$ defines the displayed algebra isomorphism. This algebra is local, with radical generated by the class of $\varepsilon$, so it has no nontrivial idempotents. A nontrivial direct-sum decomposition of $Q_{np+r}$ would give such an idempotent, proving indecomposability.
\end{proof}

The following identifies the canonical radical element used in Section~\ref{sec:comparison}.

\begin{cor}\label{cor:target-conformal-radical}
Let $P=P_{np+r}$ and $Q=F(P)$, with $n\ge1$ and $1\le r\le p-1$. Let $D_P$ be the affine conformal nilpotent of Lemma~\ref{lem:source-L0-radical}, and let $D_Q$ be the canonical Virasoro conformal nilpotent defined at the beginning of Section~\ref{sec:prelim}. Then
\(F(D_P)=D_Q,\; \rad\End(Q)=\C D_Q=\C F(N_P)\).
\end{cor}

\begin{proof}
Both endpoints of \eqref{eq:F-projective-filtration} are ordinary by Proposition~\ref{prop:Virasoro-literature-package}(i). Since the maps commute with $L_0$ and preserve generalized eigenspaces, $D_Q$ vanishes on $K_{np-r}$ and induces zero on $K_{np+r}$. Thus $D_Q(Q)\subseteq K_{np-r}$ and $D_Q^2=0$. Proposition~\ref{prop:Q-logarithmic} gives $D_Q\ne0$. Corollary~\ref{cor:target-local-endomorphisms} now implies \(\rad\End(Q)=\C D_Q=\C F(N_P)\). By Lemma~\ref{lem:source-L0-radical}, $D_P$ is a nonzero scalar multiple of $N_P$, so $F(D_P)$ spans this radical as well.

Set $\mu_Q:=e^{2\pi i h_{np+r}}$. Proposition~\ref{prop:projective-images} places every conformal weight of $Q$ in $[h_{np+r}]$. On each generalized eigenspace, $L_0=\lambda\id+D_Q$ and $D_Q^2=0$, giving $\theta_Q^{\mathrm{Vir}}=\mu_Q(\id_Q+2\pi iD_Q)$. On the other hand, twist preservation in Proposition~\ref{prop:MY-functor-package}(i) and \eqref{eq:source-twist-nilpotent} give
\[
\theta_Q^{\mathrm{Vir}}
=\lambda_P\bigl(\id_Q+2\pi iF(D_P)\bigr)
=\mu_Q\bigl(\id_Q+2\pi iD_Q\bigr).
\]
Reducing modulo $\rad\End(Q)$ gives $\lambda_P=\mu_Q$. This scalar is nonzero, so subtracting the scalar parts of the two expressions gives $F(D_P)=D_Q$.
\end{proof}

\subsection{Separation of reflection chains and walls}

For $1\le u\le p-1$ and $i\ge0$, put $Q_i^u:=Q_{s_i(u)}$, where $s_i(u)$ is the reflection chain of \eqref{eq:reflection-chain}.

\begin{lem}\label{lem:cross-chain}
For distinct $u,v\in\{1,\ldots,p-1\}$ and all $i,j\ge0$,
\(\Hom(Q_i^u,Q_j^v)=0\).
\end{lem}

\begin{proof}
Equations \eqref{eq:A-chain}--\eqref{eq:Q-chain} show that every simple constituent of $Q_i^u$ has a label on the $u$-chain, and that its initial simple $S_u$ occurs only for $i=0,1$. Distinct chains have disjoint labels by Lemma~\ref{lem:reflection-partition}. Since every chain label of index at least one exceeds $p$, Corollary~\ref{cor:positive-label-collisions} implies that two distinct chains can share a simple only when $q=2$ and $u+v=p$; their only common simple is then $S_u\cong S_v$. Lemma~\ref{lem:disjoint-constituents} proves the required vanishing except in this case with $i,j\in\{0,1\}$.

Assume $q=2$ and $u+v=p$; coprimality makes $p$ odd. Fix an isomorphism $\iota:S_v\xrightarrow{\sim}S_u$, and write $S:=S_u$. Put $A:=Q_0^u=K_u$, $B:=Q_0^v=K_v$, $T:=S_{2p-u}$ and $R:=S_{2p-v}$. Using $\iota$ on the head of $B$, \eqref{eq:A-chain} gives nonsplit sequences
\[
0\longrightarrow T\longrightarrow A\xrightarrow{\pi_A}S\longrightarrow0,
\qquad
0\longrightarrow R\longrightarrow B\xrightarrow{\pi_B}S\longrightarrow0,
\]
The simples $S,T,R$ are pairwise nonisomorphic by Corollary~\ref{cor:positive-label-collisions}.

For $f:A\to B$, Lemma~\ref{lem:disjoint-constituents} gives $f|_T=0$, so $f=\bar f\pi_A$ for a morphism $\bar f:S\to B$. Since $\Hom(S,R)=0$, a nonzero $\bar f$ would have $\pi_B\bar f=c\id_S$ with $c\in\C^\times$. Then $c^{-1}\bar f$ would split $\pi_B$, a contradiction. Thus $f=0$. Interchanging $u,v$ also gives $\Hom(B,A)=0$.

Put $U:=Q_1^u$ and $V:=Q_1^v$. Equation~\eqref{eq:Q-chain} gives
\[
0\longrightarrow A\longrightarrow U\longrightarrow K_{2p-u}\longrightarrow0,
\qquad
0\longrightarrow B\longrightarrow V\longrightarrow K_{2p-v}\longrightarrow0.
\]
Lemma~\ref{lem:Q-simple-Hom}, applied at chain indices $1,0$, gives $\Hom(U,S)=0$. Hence any $f:U\to B$ has image in $R$. This simple is not a constituent of $U$, so Lemma~\ref{lem:disjoint-constituents} gives $f=0$. Conversely, $B$ and $K_{2p-u}$ have disjoint constituents: by \eqref{eq:A-chain}, the latter has chain indices $1,2$ and does not contain $S$. Thus the composite of any $g:B\to U$ with $U\twoheadrightarrow K_{2p-u}$ is zero. It factors through $A$, and $\Hom(B,A)=0$ gives $g=0$.

Finally, apply $\Hom(U,-)$ to the sequence defining $V$. Its initial segment is
\[
0\longrightarrow\Hom(U,B)\longrightarrow\Hom(U,V)
\longrightarrow\Hom(U,K_{2p-v}).
\]
The left term is zero by the preceding argument. The right term is zero by disjoint constituents, since $K_{2p-v}$ does not contain the only possible common simple $S$. Hence $\Hom(U,V)=0$. Interchanging $u,v$ gives $\Hom(A,V)=\Hom(V,A)=\Hom(V,U)=0$, completing all four boundary cases in both directions.
\end{proof}

\begin{samepage}
\begin{lem}\label{lem:wall-Hom}
For $m,n\ge1$,
\[
\dim_\C\Hom(Q_{np},Q_{mp})=
\begin{cases}
1,&m=n,\\
0,&m\ne n.
\end{cases}
\]
For $n\ge1$, $1\le u\le p-1$ and $i\ge0$, one also has
\(\Hom(Q_{np},Q_i^u)=\Hom(Q_i^u,Q_{np})=0\).
\end{lem}
\end{samepage}

\begin{proof}
Proposition~\ref{prop:projective-images} and Proposition~\ref{prop:Virasoro-literature-package}(ii) give $Q_{np}\cong K_{np}\cong S_{np}$. A wall label is at least $p$, so its sum with any positive label exceeds $p$. Corollary~\ref{cor:positive-label-collisions} therefore excludes collisions between distinct wall labels or between wall and non-wall labels. All simple constituents of $Q_i^u$ have non-wall labels by \eqref{eq:A-chain}--\eqref{eq:Q-chain}. Simplicity gives the displayed wall Hom dimensions, and Lemma~\ref{lem:disjoint-constituents} gives the remaining vanishings.
\end{proof}

\subsection{Full faithfulness and morphisms between projectives}

\begin{thm}\label{thm:F-fully-faithful}
The restriction $F|_{\mathcal P^k}:\mathcal P^k\longrightarrow\Oc_{c_{p,q}}$ is fully faithful.
\end{thm}

\begin{proof}
For $i,j\ge1$, Proposition~\ref{prop:F-faithful-projectives} gives an injective linear map
\[
\Hom_{\KL^k}(P_i,P_j)\longrightarrow
\Hom_{\Oc_{c_{p,q}}}(F(P_i),F(P_j)),\qquad f\longmapsto F(f).
\]
Corollary~\ref{cor:source-Hom-package} gives the source dimensions. Proposition~\ref{prop:target-Hom-table} gives the same dimensions within a reflection chain, while Lemmas~\ref{lem:cross-chain} and~\ref{lem:wall-Hom} give the matching dimensions between distinct chains and for wall objects. Thus the displayed map is an isomorphism for every pair $P_i,P_j$.

By Proposition~\ref{prop:tilting-literature-package}(i), every projective is a finite direct sum of the $P_j$. The functor $F$ is additive and $\C$-linear by Proposition~\ref{prop:MY-functor-package}(i), so it sends a matrix of component morphisms between such sums to the matrix of their images. The componentwise isomorphisms therefore prove full faithfulness on $\mathcal P^k$.
\end{proof}

\begin{cor}\label{cor:finite-zigzag-algebra}
Fix a reflection chain $\{s_i\}_{i\ge0}$. Choose nonzero arrows $x_i:P_{s_i}\to P_{s_{i+1}}$ and $y_i:P_{s_{i+1}}\to P_{s_i}$ as in Corollary~\ref{cor:source-Hom-package}, and put
\[
a_i:=F(x_i),\qquad b_i:=F(y_i)\quad(i\ge0),\qquad
\rho_i:=a_{i-1}\circ b_{i-1}\quad(i\ge1).
\]
For every $N\ge0$, $F$ induces a $\C$-algebra isomorphism
\[
\End_{\KL^k}\!\left(\bigoplus_{i=0}^N P_{s_i}\right)
\xrightarrow{\ \sim\ }
\End_{\Oc_{c_{p,q}}}\!\left(\bigoplus_{i=0}^N Q_{s_i}\right).
\]
Moreover, $\rho_i\ne0$ and, for $i\ge1$,
\[
\End(Q_{s_i})=\C\id_{Q_{s_i}}\oplus\C\rho_i,\qquad
\rad\End(Q_{s_i})=\C\rho_i,\qquad \rho_i^2=0.
\]
The arrow composites satisfy
\[
b_0\circ a_0=0,\qquad b_i\circ a_i\in\C\rho_i\quad(i\ge1),
\qquad
a_{i+1}\circ a_i=b_i\circ b_{i+1}=0\quad(i\ge0).
\]
\end{cor}

\begin{proof}
The algebra isomorphism follows from Theorem~\ref{thm:F-fully-faithful}, using the canonical identification of $F(\bigoplus_iP_{s_i})$ with $\bigoplus_iQ_{s_i}$. Functoriality preserves composition and identities. Corollary~\ref{cor:source-Hom-package} gives $0\ne x_{i-1}y_{i-1}\in\rad\End(P_{s_i})$ for $i\ge1$. This element is a nonzero scalar multiple of $N_{s_i}$, so projective faithfulness and Corollary~\ref{cor:target-local-endomorphisms} give the assertions about $\rho_i$.

Every $Q_{s_i}$ is indecomposable: the initial object is $K_{s_0}$, with scalar endomorphisms by Proposition~\ref{prop:Virasoro-literature-package}(i), and the other objects are covered by Corollary~\ref{cor:target-local-endomorphisms}. Adjacent objects are nonisomorphic, since the fully faithful functor $F$ reflects isomorphisms and their source projectives have distinct simple heads. If $b_i a_i$ were invertible, then $(b_i a_i)^{-1}b_i$ would be a left inverse of $a_i$, making $Q_{s_i}$ a direct summand of the indecomposable $Q_{s_{i+1}}$. This would force them to be isomorphic. Thus $b_i a_i$ is a nonunit; it is zero at $i=0$ and belongs to $\C\rho_i$ for $i\ge1$. Finally, both $a_{i+1}a_i$ and $b_i b_{i+1}$ have endpoints at chain distance two, so Proposition~\ref{prop:target-Hom-table} makes them zero.
\end{proof}

\begin{proof}[Proof of Theorem~\ref{thm:projective-classification}]
For $1\le j<p$, Proposition~\ref{prop:MY-affine-literature-package}(i) gives $P_j=V_j$, and Lemma~\ref{lem:F-small-generalized Verma} gives $F(P_j)\cong K_j$. For $j=np$, Proposition~\ref{prop:projective-images} and Proposition~\ref{prop:Virasoro-literature-package}(ii) give $F(P_j)\cong K_j\cong S_j$. In both cases the endomorphism algebra is $\C$ by Proposition~\ref{prop:Virasoro-literature-package}(i). For $j=np+r$, $n\ge1$ and $1\le r<p$, Corollary~\ref{cor:Q-Ptau} identifies $F(P_j)$ with $P(\tau_{n,r})$, Proposition~\ref{prop:Q-logarithmic} gives logarithmicity, and Corollary~\ref{cor:target-local-endomorphisms} gives the stated endomorphism algebra and indecomposability. Full faithfulness is Theorem~\ref{thm:F-fully-faithful}.
\end{proof}

For a collection $\mathscr S$ of objects in an additive category, write $\operatorname{add}\mathscr S$ for the full subcategory whose objects are isomorphic to direct summands of finite direct sums of objects in $\mathscr S$.

\begin{cor}\label{cor:projective-essential-image}
Define the full subcategory
\[
\mathcal Q_{p,q}:=\operatorname{add}\!\left(
\{K_r:1\le r<p\}\cup\{K_{np}:n\ge1\}
\cup\{P(\tau_{n,r}):n\ge1,\ 1\le r<p\}
\right)
\]
of $\Oc_{c_{p,q}}$. It contains the tensor unit and is closed under $\boxtimes$. With the inherited associativity, unit and braiding constraints, \(F|_{\mathcal P^k}:\mathcal P^k\xrightarrow{\sim}\mathcal Q_{p,q}\) is a $\C$-linear braided strong monoidal equivalence.
\end{cor}

\begin{proof}
Theorem~\ref{thm:projective-classification} and the projective decomposition in Proposition~\ref{prop:tilting-literature-package}(i) show that $F(P)\in\mathcal Q_{p,q}$ for every projective $P$. Conversely, any $X\in\mathcal Q_{p,q}$ is, up to isomorphism, a direct summand of $F(P)$ for a finite direct sum $P$ of indecomposable projectives. Choose maps $i:X\to F(P)$ and $r:F(P)\to X$ with $ri=\id_X$, and put $e:=ir$.

By Theorem~\ref{thm:F-fully-faithful}, $e$ has a unique lift $\widetilde e\in\End_{\KL^k}(P)$. Since $F(\widetilde e^2-\widetilde e)=e^2-e=0$, faithfulness makes $\widetilde e$ idempotent. It splits in the abelian source category: there are maps $P'\xrightarrow{j}P\xrightarrow{t}P'$ with $tj=\id_{P'}$ and $jt=\widetilde e$. The object $P'$ is projective as a direct summand of $P$. Applying $F$ gives a splitting of $e$ with retract $F(P')$. Explicitly, $F(t)i:X\to F(P')$ and $rF(j):F(P')\to X$ are inverse isomorphisms. Hence every object of $\mathcal Q_{p,q}$ is isomorphic to the image of a projective, proving the $\C$-linear equivalence.

By Proposition~\ref{prop:MY-affine-literature-package}(iii), the tensor product of projectives is projective. Since $F$ is braided strong monoidal by Proposition~\ref{prop:MY-functor-package}(i), for $X\cong F(P)$ and $Y\cong F(R)$ one has \(X\boxtimes Y\cong F(P\boxtimes R)\in\mathcal Q_{p,q}\). The unit belongs to this subcategory because $\one_{\Oc}\cong K_1\cong F(P_1)$. Its fullness and closure under tensor products allow the associativity, unit and braiding constraints of $\Oc_{c_{p,q}}$ to restrict to it. The structure maps of $F$ restrict as well, giving the asserted braided strong monoidal equivalence.
\end{proof}

\section{Comparison with Drinfeld--Sokolov reduction}\label{sec:comparison}

We first identify the reductions $H(P_j)$ with the projective images $F(P_j)$ of Theorem~\ref{thm:projective-classification}, using the exactness and faithfulness of $H$ from Proposition~\ref{prop:H-exact} and Corollary~\ref{cor:H-faithful}. We then compare the canonical conformal nilpotents under both functors and use the projective Hom spaces to construct a natural isomorphism on $\mathcal P^k$. Proposition~\ref{prop:right-exact-extension} extends it to $\KL^k$ using the right exactness of $F$.

\subsection{Projective objects under reduction}

Apply the exact functor $H$ to \eqref{eq:MY-projective-Verma-filtration}. Choosing the endpoint identifications from Proposition~\ref{prop:H-generalized Verma} gives, for $n\ge1$ and $1\le r\le p-1$,
\begin{equation}\label{eq:H-projective-filtration}
0\longrightarrow K_{np-r}\longrightarrow H(P_{np+r})
\longrightarrow K_{np+r}\longrightarrow0.
\end{equation}

\begin{prop}\label{prop:projective-object-comparison}
For every indecomposable projective $P_j$, one has Virasoro-module isomorphisms
\[
F(P_j)\cong H(P_j)\cong
\begin{cases}
K_j,&1\le j\le p-1,\\
K_j\cong S_j,&p\mid j,\\
P(\tau_{n,r}),&j=np+r,\ n\ge1,\ 1\le r\le p-1.
\end{cases}
\]
In the third case both images are logarithmic. Compatibility of these objectwise isomorphisms with source morphisms is established in the next subsection.
\end{prop}

\begin{proof}
The images under $F$ are given by Theorem~\ref{thm:projective-classification}. If $1\le j<p$ or $p\mid j$, Proposition~\ref{prop:MY-affine-literature-package}(i) gives $P_j=V_j$. Thus Proposition~\ref{prop:H-generalized Verma} gives $H(P_j)\cong K_j$, and Proposition~\ref{prop:Virasoro-literature-package}(ii) identifies $K_j$ with $S_j$ when $p\mid j$.

Now let $P=P_{np+r}$, with $n\ge1$ and $1\le r\le p-1$, and put $E:=H(P)$. By Lemma~\ref{lem:source-L0-radical}, $D_P\ne0$ and $D_P^2=0$. Corollary~\ref{cor:H-faithful} therefore gives $H(D_P)\ne0$, while functoriality gives $H(D_P)^2=0$. Proposition~\ref{prop:H-twist} and \eqref{eq:source-twist-nilpotent} give
\[
e^{2\pi iL^{DS}_0}\big|_E
=H(\theta^-_P)
=\lambda_P\bigl(\id_E+2\pi iH(D_P)\bigr),
\qquad \lambda_P\in\C^\times.
\]
Consequently this exponential has a nonzero square-zero difference from $\lambda_P\id_E$ and cannot be diagonalizable. Lemma~\ref{lem:exponential-logarithmic} shows that $E$ is logarithmic. With $a=a_{n,r}=np-r$ and $b=b_{n,r}=np+r$ as in \eqref{eq:abc}, the sequence \eqref{eq:H-projective-filtration} has exactly the form $0\to K_a\to E\to K_b\to0$ required by Lemma~\ref{lem:Nakano-recognition}. That lemma gives $E\cong P(\tau_{n,r})$, completing the comparison.
\end{proof}

\begin{cor}\label{cor:H-projective-radical}
Let $P=P_{np+r}$ with $n\ge1$ and $1\le r\le p-1$, and set $E:=H(P)$. Let $D_P$ and $D_E$ be the canonical conformal nilpotents defined at the beginning of Section~\ref{sec:prelim}, using $L^{\mathrm{aff}}_0$ on $P$ and $L^{DS}_0$ on $E$. Then \(H(D_P)=D_E,\; \rad\End_{\Oc_{c_{p,q}}}(E)=\C D_E=\C H(N_P)\). In particular,
\[
\End_{\Oc_{c_{p,q}}}(E)=\C\id_E\oplus\C D_E
\cong\C[\varepsilon]/(\varepsilon^2).
\]
\end{cor}

\begin{proof}
Both endpoints of \eqref{eq:H-projective-filtration} are ordinary by Proposition~\ref{prop:Virasoro-literature-package}(i). Since its maps preserve generalized conformal weights, $D_E$ vanishes on $K_{np-r}$ and induces zero on $K_{np+r}$. Hence $D_E(E)\subseteq K_{np-r}$ and $D_E^2=0$. Proposition~\ref{prop:projective-object-comparison} gives $D_E\ne0$. An objectwise isomorphism $E\cong F(P)$ induces an algebra isomorphism on endomorphisms by conjugation. Corollary~\ref{cor:target-local-endomorphisms} therefore identifies $\End_{\Oc_{c_{p,q}}}(E)$ with $\C[\varepsilon]/(\varepsilon^2)$. Every nonzero square-zero element of this algebra spans its radical, so \(\rad\End_{\Oc_{c_{p,q}}}(E)=\C D_E=\C H(D_P)\), where $H(D_P)\ne0$ and $H(D_P)^2=0$ as in the preceding proof.

The endpoints have the same conformal residue, since \eqref{eq:weights} gives \(h_{np+r}-h_{np-r}=r(nq-1)\in\Z\). Each $K_j$ is supported in $[h_j]$, so exact residue projection from Lemma~\ref{lem:block-projection} gives $E=E^{[h_{np+r}]}$. Put $\mu_E:=e^{2\pi i h_{np+r}}$. On each generalized conformal-weight space the definition of $D_E$ and the equality $D_E^2=0$ give
\[
e^{2\pi iL^{DS}_0}\big|_E
=\mu_E\bigl(\id_E+2\pi iD_E\bigr)
=\lambda_P\bigl(\id_E+2\pi iH(D_P)\bigr),
\]
where the second equality follows from Proposition~\ref{prop:H-twist} and \eqref{eq:source-twist-nilpotent}. Reducing this identity modulo the radical gives $\mu_E=\lambda_P$. Subtracting the scalar parts and dividing by the nonzero scalar $2\pi i\mu_E$ yields $H(D_P)=D_E$. Finally, Lemma~\ref{lem:source-L0-radical} gives $D_P=\delta_PN_P$ with $\delta_P\in\C^\times$, proving the remaining radical equality and the displayed endomorphism algebra.
\end{proof}

\subsection{Canonical conformal radicals and naturality}

On any reflection chain $s_i=s_i(r)$ from \eqref{eq:reflection-chain}, write $\Pi_i:=P_{s_i}$ and choose nonzero maps
\[
x_i:\Pi_i\longrightarrow\Pi_{i+1},\qquad
y_i:\Pi_{i+1}\longrightarrow\Pi_i\qquad(i\ge0).
\]
By Corollary~\ref{cor:source-Hom-package}, the composite $\nu_{i+1}:=x_i\circ y_i$ is nonzero and spans $\rad\End_{\KL^k}(\Pi_{i+1})$.

\begin{lem}\label{lem:twist-locking}
Let $P=P_{np+r}$ with $n\ge1$ and $1\le r\le p-1$. Every Virasoro-module isomorphism $\eta_P:F(P)\xrightarrow{\sim}H(P)$ satisfies
\(\eta_PF(u)=H(u)\eta_P \;\text{for every }u\in\rad\End_{\KL^k}(P)\).
\end{lem}

\begin{proof}
The isomorphism $\eta_P$ commutes with $L_0$ and preserves its generalized eigenspaces. Thus, for each $\lambda\in\operatorname{Spec}(F(P))$, one has
\(\eta_P(L_0-\lambda\id)=(L_0-\lambda\id)\eta_P\) on $F(P)_{[\lambda]}$,
and hence $\eta_PD_{F(P)}=D_{H(P)}\eta_P$ by the definition of the canonical conformal nilpotent. Corollaries~\ref{cor:target-conformal-radical} and~\ref{cor:H-projective-radical} identify these nilpotents with $F(D_P)$ and $H(D_P)$, respectively. Lemma~\ref{lem:source-L0-radical} gives $\rad\End_{\KL^k}(P)=\C D_P$, so linearity proves the assertion.
\end{proof}

\begin{prop}\label{prop:natural-on-projectives}
There is a natural isomorphism of $\C$-linear functors \(\eta:F|_{\mathcal P^k}\xRightarrow{\ \sim\ }H|_{\mathcal P^k}\). The freedom in choosing this comparison is described in Remark~\ref{rmk:comparison-noncanonical}.
\end{prop}

\begin{proof}
Fix a reflection chain and use the notation above. Proposition~\ref{prop:projective-object-comparison} gives objectwise isomorphisms $F(\Pi_i)\cong H(\Pi_i)$. Pre- and postcomposition with such isomorphisms, together with Theorem~\ref{thm:F-fully-faithful}, identify the dimensions of \(\Hom_{\Oc_{c_{p,q}}}(\mathscr G(\Pi_i),\mathscr G'(\Pi_j))\) with those of $\Hom_{\KL^k}(\Pi_i,\Pi_j)$ for $\mathscr G,\mathscr G'\in\{F,H\}$. In particular, these target Hom spaces are one-dimensional when $|i-j|=1$, by Corollary~\ref{cor:source-Hom-package}.

Choose an isomorphism $\eta_0:F(\Pi_0)\to H(\Pi_0)$. Suppose $\eta_i$ is chosen and take any isomorphism $\eta'_{i+1}:F(\Pi_{i+1})\to H(\Pi_{i+1})$. Projective faithfulness of $F$ and faithfulness of $H$, from Proposition~\ref{prop:F-faithful-projectives} and Corollary~\ref{cor:H-faithful}, make both $\eta'_{i+1}F(x_i)$ and $H(x_i)\eta_i$ nonzero. They lie in the same one-dimensional Hom space, so there is a unique $a_i\in\C^\times$ with \(\eta'_{i+1}F(x_i)=a_iH(x_i)\eta_i\). Set $\eta_{i+1}:=a_i^{-1}\eta'_{i+1}$. Then \(\eta_{i+1}F(x_i)=H(x_i)\eta_i\). For the same reason, there is a unique $c_i\in\C^\times$ such that \(\eta_iF(y_i)=c_iH(y_i)\eta_{i+1}\). Composing these relations gives \(\eta_{i+1}F(\nu_{i+1}) =H(x_i)\eta_iF(y_i) =c_iH(\nu_{i+1})\eta_{i+1}\). Lemma~\ref{lem:twist-locking} also gives \(\eta_{i+1}F(\nu_{i+1})=H(\nu_{i+1})\eta_{i+1}\). Since $\nu_{i+1}\ne0$, faithfulness of $H$ and invertibility of $\eta_{i+1}$ give $H(\nu_{i+1})\eta_{i+1}\ne0$. Therefore $c_i=1$, proving naturality for both $x_i$ and $y_i$.

Corollary~\ref{cor:source-Hom-package} lists all nonzero Hom spaces on the chain: adjacent spaces are generated by $x_i,y_i$, the initial endomorphism algebra is scalar, and for $i\ge1$ the endomorphism algebra has basis $\id_{\Pi_i},\nu_i$. Naturality holds for $\nu_i$ by Lemma~\ref{lem:twist-locking} and for identities automatically. Linearity therefore gives naturality for every morphism between the $\Pi_i$.

Repeat the construction on each reflection chain. For each wall projective $P_{np}$, choose any isomorphism $\eta_{P_{np}}:F(P_{np})\to H(P_{np})$ from Proposition~\ref{prop:projective-object-comparison}. Corollary~\ref{cor:source-Hom-package} gives scalar endomorphisms for these objects and zero Hom spaces between distinct chains, between distinct wall objects, and between wall and non-wall objects. Thus all the chosen components are compatible with every morphism between the $P_j$, which exhaust the indecomposable projectives by Proposition~\ref{prop:tilting-literature-package}(i).

Every projective is a finite direct sum of these indecomposables, by the same proposition. For each nonzero projective $R$, choose a decomposition \(\alpha_R:R\xrightarrow{\sim}\bigoplus_jP_j^{\oplus m_j(R)}\) with finitely many nonzero multiplicities, taking $\alpha_{P_j}=\id_{P_j}$ for the chosen representatives. Using the canonical finite-direct-sum identifications for the additive functors $F,H$, define
\[
\eta_R:=H(\alpha_R)^{-1}
\left(\bigoplus_j\eta_{P_j}^{\oplus m_j(R)}\right)F(\alpha_R),
\]
and take the unique map for $R=0$. For $f:R\to R'$, the matrix of $\alpha_{R'}f\alpha_R^{-1}$ has entries between the $P_j$, on which naturality is proved. Componentwise matrix multiplication gives $\eta_{R'}F(f)=H(f)\eta_R$. Applying the same calculation to two decompositions of one object shows that $\eta_R$ is independent of $\alpha_R$ once the components on the $P_j$ are fixed.
\end{proof}

\begin{lem}\label{lem:radical-annihilates-arrows}
With the reflection-chain notation above, one has
\[
\begin{aligned}
H(\nu_{i+1})H(x_i)&=0,& H(y_i)H(\nu_{i+1})&=0
&& (i\ge0),\\
H(x_i)H(\nu_i)&=0,& H(\nu_i)H(y_i)&=0
&& (i\ge1).
\end{aligned}
\]
Suppose isomorphisms $\eta_i:F(\Pi_i)\to H(\Pi_i)$ satisfy $\eta_jF(f)=H(f)\eta_i$ for every $f:\Pi_i\to\Pi_j$. For any scalars $a_i\in\C$, $i\ge1$, the modified family
\[
\eta'_0:=\eta_0,\qquad
\eta'_i:=\bigl(\id_{H(\Pi_i)}+a_iH(\nu_i)\bigr)\eta_i
\quad(i\ge1)
\]
is again a natural family of isomorphisms on the chain.
\end{lem}

\begin{proof}
Corollary~\ref{cor:source-Hom-package} and Lemma~\ref{lem:source-L0-radical} give $\nu_i^2=0$ for $i\ge1$. Pre- or postcomposition by $\nu_i$ therefore acts as a square-zero linear operator on each adjacent source Hom space involving $\Pi_i$. Such a space is one-dimensional by Corollary~\ref{cor:source-Hom-package}, and a square-zero operator on a one-dimensional space is zero. Applying $H$ to these four source identities gives the displayed relations.

Since $H(\nu_i)^2=0$, the factor $\id+a_iH(\nu_i)$ is invertible with inverse $\id-a_iH(\nu_i)$. The four relations make every correction term in the adjacent naturality equations vanish, including when both endpoints are modified. These factors also commute with $H(\nu_i)$, so naturality for $\nu_i$ is preserved. Identities, radical generators and adjacent maps span all source Hom spaces on the chain by Corollary~\ref{cor:source-Hom-package}; hence the modified family is natural for every morphism on it.
\end{proof}

\subsection{Extension from projectives}

\begin{prop}\label{prop:right-exact-extension}
Let $\mathcal A,\mathcal B$ be $\C$-linear abelian categories, and let $\mathcal P\subset\mathcal A$ be the full subcategory of projective objects. Assume that $\mathcal A$ has enough projectives: every object admits an epimorphism from a projective object.

Let $G_1,G_2:\mathcal A\to\mathcal B$ be additive $\C$-linear functors. Suppose both are right exact, meaning that each sends an exact sequence $A\to B\to C\to0$ to an exact sequence $G_i(A)\to G_i(B)\to G_i(C)\to0$. Write $\operatorname{Nat}(G_1,G_2)$ for their space of natural transformations. Restriction induces a bijection \(\operatorname{Nat}(G_1,G_2)\xrightarrow{\ \sim\ }\operatorname{Nat}(G_1|_{\mathcal P},G_2|_{\mathcal P})\). Every natural isomorphism on $\mathcal P$ extends uniquely to a natural isomorphism on $\mathcal A$.
\end{prop}

\begin{proof}
First, restriction is injective. If two natural transformations $\theta,\theta':G_1\Rightarrow G_2$ agree on $\mathcal P$, choose, for any $M\in\mathcal A$, an epimorphism $p:P\twoheadrightarrow M$ with $P$ projective. Naturality gives \((\theta_M-\theta'_M)G_1(p)=G_2(p)(\theta_P-\theta'_P)=0\). Right exactness makes $G_1(p)$ an epimorphism, so $\theta_M=\theta'_M$.

Let $\eta^{\mathcal P}:G_1|_{\mathcal P}\Rightarrow G_2|_{\mathcal P}$ be a natural transformation. For $M\in\mathcal A$, choose epimorphisms $p:P_0\twoheadrightarrow M$ and $u:P_1\twoheadrightarrow\ker p$ with $P_0,P_1$ projective. Let $d:P_1\to P_0$ be $u$ followed by the inclusion of $\ker p$. Then \(P_1\xrightarrow{d}P_0\xrightarrow{p}M\longrightarrow0\) is exact; such a sequence is called a projective presentation of $M$. Right exactness identifies $G_1(M)$, with the map $G_1(p)$, as a cokernel of $G_1(d)$. Naturality for $d$ gives \(G_2(p)\eta^{\mathcal P}_{P_0}G_1(d) =G_2(p)G_2(d)\eta^{\mathcal P}_{P_1}=0\). The cokernel universal property therefore defines a unique morphism $\eta_M:G_1(M)\to G_2(M)$ satisfying
\begin{equation}\label{eq:eta-projective-presentation}
\eta_MG_1(p)=G_2(p)\eta^{\mathcal P}_{P_0}.
\end{equation}
This characterization also shows independence of the choice of $P_1$ and $u$.

To check independence of $p$, let $p':P'_0\twoheadrightarrow M$ be another projective epimorphism, and let $\eta'_M$ be the resulting morphism. Projectivity of $P_0$ gives $a:P_0\to P'_0$ with $p'a=p$. Naturality on projectives gives
\begin{align*}
\eta'_MG_1(p)
=G_2(p')\eta^{\mathcal P}_{P'_0}G_1(a)=G_2(p'a)\eta^{\mathcal P}_{P_0}
=G_2(p)\eta^{\mathcal P}_{P_0}
=\eta_MG_1(p).
\end{align*}
Since $G_1(p)$ is an epimorphism, $\eta'_M=\eta_M$. Thus \eqref{eq:eta-projective-presentation} holds for every projective epimorphism onto $M$.

For $f:M\to N$, choose projective epimorphisms $p:P\twoheadrightarrow M$ and $q:Q\twoheadrightarrow N$. Projectivity of $P$ gives $a:P\to Q$ with $qa=fp$. By \eqref{eq:eta-projective-presentation} and naturality on projectives,
\begin{align*}
G_2(f)\eta_MG_1(p)
=G_2(qa)\eta^{\mathcal P}_P=G_2(q)\eta^{\mathcal P}_QG_1(a)=\eta_NG_1(qa)=\eta_NG_1(f)G_1(p).
\end{align*}
Cancelling the epimorphism $G_1(p)$ gives $G_2(f)\eta_M=\eta_NG_1(f)$, so the family $\eta$ is natural. For projective $P$, taking $p=\id_P$ in \eqref{eq:eta-projective-presentation} gives $\eta_P=\eta^{\mathcal P}_P$. This proves surjectivity; uniqueness follows from the injectivity proved above.

If $\eta^{\mathcal P}$ is a natural isomorphism, extend its inverse by the same construction to $\zeta:G_2\Rightarrow G_1$. The composites $\zeta\eta$ and $\eta\zeta$ restrict to identity transformations on $\mathcal P$, so injectivity of restriction makes them identity transformations on $\mathcal A$. Hence $\eta$ is a natural isomorphism.
\end{proof}

\subsection{Main theorem and consequences}

\begin{proof}[Proof of Theorem~\ref{thm:main}]
The source category has enough projectives by Proposition~\ref{prop:MY-affine-literature-package}(i), and the target is abelian by Subsection~\ref{subsec:MY-input}. The functor $F$ is additive, $\C$-linear and right exact by Proposition~\ref{prop:MY-functor-package}(i), while $H$ is additive, $\C$-linear and exact by Proposition~\ref{prop:H-exact}. Proposition~\ref{prop:right-exact-extension} therefore extends the projective comparison of Proposition~\ref{prop:natural-on-projectives} uniquely to
\begin{equation}\label{eq:main-natural-iso}
\eta:F_{p,q}\xRightarrow{\ \sim\ }
H^0_{DS,+}\big|_{\KL^k(\mathfrak{sl}_2)}.
\end{equation}
Exactness of $H$ transfers to $F$ through this natural isomorphism. For a morphism $f:M\to N$ in $\KL^k$, naturality gives $H(f)=\eta_NF(f)\eta_M^{-1}$. Thus faithfulness of $H$ from Corollary~\ref{cor:H-faithful} implies faithfulness of $F$. Full faithfulness on projectives is Theorem~\ref{thm:F-fully-faithful}. Proposition~\ref{prop:DS-literature-package}(v) identifies $H$ with the reduction in the McRae--Yang conjecture. The assertions for the second functor are established in Corollary~\ref{cor:transposed-reduction} below.
\end{proof}

\begin{cor}\label{cor:transposed-reduction}
Let $k'=-2+q/p$. The second McRae--Yang functor satisfies \(F_{q,p}\cong H^0_{DS,+}\big|_{\KL^{k'}(\mathfrak{sl}_2)}\) as $\C$-linear functors into $\Oc_{c_{p,q}}$. It is exact and faithful, and its restriction to projective objects is fully faithful.
\end{cor}

\begin{proof}
Apply the argument just completed for $F^{(p,q)}$ to the ordered pair $(q,p)$. It gives $F^{(q,p)}\cong H^0_{DS,+}|_{\KL^{k'}(\mathfrak{sl}_2)}$, together with exactness, faithfulness and full faithfulness on projectives. Since $c_{q,p}=c_{p,q}$, Subsection~\ref{subsec:MY-input} identifies the two target braided tensor categories.

Set $A:=\Kac^{(q,p)}_{2,1}$ and $B:=\Kac^{(p,q)}_{1,2}$, with the chosen evaluation and coevaluation maps $e_A,i_A,e_B,i_B$ of Subsection~\ref{subsec:MY-input}. Equation~\eqref{eq:Kac-parameter-swap} gives an isomorphism $\phi_0:A\to B$. Proposition~\ref{prop:Virasoro-literature-package}(i), applied to the ordered pair $(q,p)$, gives $\End(A)=\C\id_A$, and conjugation by $\phi_0$ gives the same assertion for $B$. Both objects are rigid and self-dual by Subsection~\ref{subsec:MY-input}. Duality therefore identifies each of $\Hom(A\boxtimes A,\one_{\Oc})$ and $\Hom(\one_{\Oc},A\boxtimes A)$ with $\End(A)$, and similarly for $B$. These spaces are one-dimensional, so there are $a,b\in\C^\times$ with
\[
e_B(\phi_0\boxtimes\phi_0)=a e_A,\qquad
(\phi_0^{-1}\boxtimes\phi_0^{-1})i_B=b i_A.
\]
The transported maps $ae_A,bi_A$ satisfy the duality identities, which give $ab=1$. Choose $t\in\C^\times$ with $t^2=a^{-1}$ and set $\phi:=t\phi_0$. Then \(e_B(\phi\boxtimes\phi)=e_A,\; (\phi^{-1}\boxtimes\phi^{-1})i_B=i_A\). Composing the generator identification $F^{(q,p)}(V_2)\cong A$ with $\phi$ makes $F^{(q,p)}$ a right exact strong monoidal functor normalized by $(B,e_B,i_B)$. Proposition~\ref{prop:MY-functor-package}(ii) now gives a tensor natural isomorphism $F^{(q,p)}\cong F_{q,p}$. Composing the underlying $\C$-linear comparisons proves the reduction identity and transfers all three stated properties to $F_{q,p}$.
\end{proof}

\begin{samepage}
\begin{cor}\label{cor:main-consequences}
For every $r\ge1$, one has $F_{p,q}(V_r)\cong K_r$ and $F_{p,q}(L_r)\cong S_r$.
\end{cor}

\begin{proof}
Apply \eqref{eq:main-natural-iso} and Propositions~\ref{prop:H-generalized Verma} and~\ref{prop:H-simple}.
\end{proof}
\end{samepage}

\begin{cor}\label{cor:DS-projective-fully-faithful}
For projective objects $P,Q$, reduction induces an isomorphism
\[
\Hom_{\KL^k}(P,Q)\xrightarrow{\ \sim\ }
\Hom_{\Oc_{c_{p,q}}}(H(P),H(Q)),\qquad f\longmapsto H(f).
\]
On each reflection chain it induces an equivalence
\[
H:\operatorname{add}\{P_{s_i}:i\ge0\}
\xrightarrow{\ \sim\ }
\operatorname{add}\{H(P_{s_i}):i\ge0\}.
\]
For every $N\ge0$, it also induces an algebra isomorphism between the endomorphism algebras of $\bigoplus_{i=0}^N P_{s_i}$ and $\bigoplus_{i=0}^N H(P_{s_i})$.
\end{cor}

\begin{proof}
Naturality of \eqref{eq:main-natural-iso} gives $H(f)=\eta_QF(f)\eta_P^{-1}$. Pre- and postcomposition with $\eta_P^{-1}$ and $\eta_Q$ give a vector-space isomorphism, and Theorem~\ref{thm:F-fully-faithful} makes $f\mapsto F(f)$ an isomorphism. Thus $H$ is fully faithful on projectives.

Here $\operatorname{add}$ has the finite-direct-sum and direct-summand meaning fixed before Corollary~\ref{cor:projective-essential-image}. Additivity sends the source closure into the target closure. To see essential surjectivity on the displayed additive closures, let $X$ be a direct summand of $H(R)$, where $R$ is a finite direct sum of projectives on the chain. Full faithfulness lifts its defining idempotent uniquely to an idempotent of $R$, as in the proof of Corollary~\ref{cor:projective-essential-image}. This splits in the abelian source, with retract $R'\in\operatorname{add}\{P_{s_i}:i\ge0\}$. Applying the additive functor $H$ to the splitting gives $H(R')\cong X$. Hence the restricted functor is an equivalence. Finally, functoriality preserves identities and composition, so full faithfulness gives the finite-segment algebra isomorphism, as for $F$ in Corollary~\ref{cor:finite-zigzag-algebra}.
\end{proof}

\begin{rmk}\label{rmk:comparison-noncanonical}
Fix a natural isomorphism $\eta:F\xRightarrow{\sim}H$. Any other natural isomorphism $\eta'$ determines a natural automorphism $u_P:=\eta'_P\eta_P^{-1}$ of $H$ on projectives. On a reflection chain, use $\Pi_i,x_i,y_i,\nu_i$ from the proof of Proposition~\ref{prop:natural-on-projectives}. At $i=0$ the endomorphism algebra of $H(\Pi_0)$ is scalar, by Proposition~\ref{prop:projective-object-comparison} and Proposition~\ref{prop:Virasoro-literature-package}(i). Thus $u_{\Pi_0}=\alpha_0\id$ with $\alpha_0\in\C^\times$. For $i\ge1$, Corollary~\ref{cor:H-projective-radical} gives
\[
u_{\Pi_i}=\alpha_i\bigl(\id+a_iH(\nu_i)\bigr),
\qquad \alpha_i\in\C^\times,\ a_i\in\C.
\]
Naturality for $x_i$ and Lemma~\ref{lem:radical-annihilates-arrows} give $(\alpha_{i+1}-\alpha_i)H(x_i)=0$. Faithfulness of $H$ forces $\alpha_{i+1}=\alpha_i$. Thus one common nonzero scalar may be chosen on each chain, and the parameters $a_i$, $i\ge1$, are independent. Conversely, Lemma~\ref{lem:radical-annihilates-arrows} shows that all these choices preserve naturality. Each wall component may be scaled independently, since Corollary~\ref{cor:source-Hom-package} gives zero morphisms between it and every other indecomposable projective. Proposition~\ref{prop:right-exact-extension} extends each such projective comparison uniquely to the whole category.
\end{rmk}

\begin{samepage}
\begin{cor}\label{cor:transported-tensor}
Fix a $\C$-linear natural isomorphism $\eta:F_{p,q}\xRightarrow{\sim}H$ as in \eqref{eq:main-natural-iso}. There is a unique braided strong monoidal structure on $H$ for which $\eta$, with the given structure on $F_{p,q}$, is a monoidal natural isomorphism. It preserves twists:
\(H(\theta^-_M)=\theta^{\mathrm{Vir}}_{H(M)} \;(M\in\KL^k)\).
\end{cor}
\end{samepage}

\begin{proof}
Let $\Phi^F_{X,Y}$ and $\Phi^F_0$ be the structure maps in the convention of Subsection~\ref{subsec:MY-input}. For $X,Y\in\KL^k$, define
\begin{align*}
\Phi^H_{X,Y}:=\eta_{X\boxtimes Y}\circ\Phi^F_{X,Y}
\circ(\eta_X^{-1}\boxtimes\eta_Y^{-1}),\,
\Phi^H_0:=\eta_{V_1}\circ\Phi^F_0.
\end{align*}
These are invertible and natural. Substitution in the associativity and unit coherence diagrams reduces them to the corresponding diagrams for $F$. Writing $c^{\KL}$ and $c^{\Oc}$ for the source and target braidings, naturality of $\eta$ and braidedness of $F$ likewise give \(H(c^{\KL}_{X,Y})\circ\Phi^H_{X,Y} =\Phi^H_{Y,X}\circ c^{\Oc}_{H(X),H(Y)}\). Thus the maps define a braided strong monoidal structure. Compatibility of $\eta$ with the tensorator and unit constraint forces the displayed formulas, proving uniqueness for this fixed $\eta$. The twist identity is Proposition~\ref{prop:H-twist}, with the target twist defined in Subsection~\ref{subsec:MY-input}.
\end{proof}

\subsection*{AI use statement}
OpenAI's GPT-5.6 Sol was used as a research-assistance tool for literature exploration, conceptual brainstorming, notation and consistency checks, LaTeX editing, and language polishing. The author independently checked the mathematical statements, proofs, and citations and takes responsibility for the final manuscript.

\end{document}